\documentclass[letterpaper,11pt]{article}
\usepackage[margin=3cm]{geometry}

\usepackage[textsize=scriptsize, textcolor=red, linecolor=gray, bordercolor=orange, textwidth=3cm, backgroundcolor=white, disable]{todonotes}

\usepackage[utf8]{inputenc}
\usepackage{amsmath, amsthm, amssymb, mathtools, bm, bigints}
\usepackage{graphicx}
\usepackage[shortlabels]{enumitem}
\usepackage{tikz, tkz-euclide, tikz-cd, tikz-dimline}  \usetikzlibrary{arrows, decorations.pathreplacing, positioning, shapes}
\usepackage{hyperref}
\usepackage[round]{natbib}
\usepackage[plain,noend]{algorithm2e}

\SetCommentSty{mycommfont}
\usepackage{booktabs, multirow}
\usepackage{authblk}
\usepackage{xspace}

\usepackage[capitalise]{cleveref}

\theoremstyle{plain}
\newtheorem{theorem}{Theorem}
\newtheorem{corollary}{Corollary}
\newtheorem{proposition}{Proposition}
\newtheorem{lemma}{Lemma}

\newtheorem{assumption}{Assumption}
\theoremstyle{definition}
\newtheorem{definition}{Definition}

\newtheorem{example}{Example}
\theoremstyle{remark}
\newtheorem{remark}{Remark}

\usepackage{chngcntr}
\usepackage{apptools}
\AtAppendix{\counterwithin{theorem}{section}}
\AtAppendix{\counterwithin{figure}{section}}
\AtAppendix{\counterwithin{table}{section}}
\AtAppendix{\counterwithin{lemma}{section}}
\AtAppendix{\counterwithin{corollary}{section}}

\makeatletter
\renewenvironment{proof}[1][\relax]{\par
  \pushQED{\qed}\normalfont \topsep6\p@\@plus6\p@\relax
  \trivlist
  \item[\hskip\labelsep\itshape
    \ifx#1\relax \proofname\else\proofname{} #1\fi\@addpunct{.}]\ignorespaces
}{\popQED\endtrivlist\@endpefalse
}
\makeatother

\newcommand{\supp}{Appendix\xspace}
\newcommand{\supps}{Appendices\xspace}

\newcommand{\indep}{\mathrel{\text{\scalebox{1.07}{$\perp\mkern-10mu\perp$}}}}
\DeclareMathOperator{\var}{\mathrm{var}}

\def\model{{\mathcal{M}}}

\DeclareMathOperator{\E}{\mathbb{E}}

\newcommand*\dd{\mathop{}\!\mathrm{d}}

\newcommand{\I}{\mathbb{I}}
\newcommand{\dcup}{\,\dot{\cup}\,}

\newcommand{\X}{\mathfrak{X}}
\newcommand{\F}{\mathcal{F}}

\newcommand{\unif}{\mathrm{unif}}

\newcommand{\N}{\mathcal{N}}

\newcommand{\g}[1][G]{\mathcal{#1}}

\newcommand{\grev}[1][G]{\check{\mathcal{#1}}}

\newcommand{\Vset}{V}
\newcommand{\Cset}{C}
\newcommand{\Lset}{L}
\newcommand{\Nset}{N}
\newcommand{\Mset}{M}
\newcommand{\Iset}{I}
\newcommand{\Wset}{W}
\newcommand{\Oset}{O}

\newcommand{\Uset}{U}
\newcommand{\Omin}{\Oset_{\min}}
\newcommand{\opt}{\Oset}

\newcommand{\plus}{E^{+}}
\newcommand{\minus}{E^{-}}

\DeclareMathOperator{\ceq}{\stackrel{c}{\sim}}

\DeclareMathOperator{\De}{De}

\DeclareMathOperator{\Ch}{Ch}

\DeclareMathOperator{\An}{An}

\DeclareMathOperator{\Pa}{Pa}
\DeclareMathOperator{\pa}{Pa}

\DeclareMathOperator{\Adj}{Adj}

\title{Variable elimination, graph reduction and efficient g-formula}

\author[1]{F. Richard Guo \thanks{ricguo@statslab.cam.ac.uk}}
\author[2]{Emilija Perkovi\'c \thanks{perkovic@uw.edu}}
\author[3]{Andrea Rotnitzky \thanks{arotnitzky@utdt.edu}}
\affil[1]{Statistical Laboratory, University of Cambridge, Cambridge, UK}
\affil[2]{Department of Statistics, University of Washington, Seattle, USA}
\affil[3]{Department of Economics, Universidad Torcuato Di Tella, Buenos Aires, Argentina}
\date{}

\begin{document}

\maketitle
\vspace{-1em}

\begin{abstract} 
We study efficient estimation of an interventional mean associated with a point exposure treatment under a causal graphical model represented by a directed acyclic graph without hidden variables. Under such a model, it may happen that a subset of the variables are uninformative in that failure to measure them neither precludes identification of the interventional mean nor changes the semiparametric variance bound for regular estimators of it. We develop a set of graphical criteria that are sound and complete for eliminating all the uninformative variables so that the cost of measuring them can be saved without sacrificing estimation efficiency, which could be useful when designing a planned observational or randomized study. Further, we construct a reduced directed acyclic graph on the set of informative variables only. We show that the interventional mean is identified from the marginal law by the g-formula \citep{robins1986new} associated with the reduced graph, and the semiparametric variance bounds for estimating the interventional mean under the original and the reduced graphical model agree. 
This g-formula is an irreducible, efficient identifying formula in the sense that the nonparametric estimator of the formula, under regularity conditions, is asymptotically efficient under the original causal graphical model, and no formula with such property exists that only depends on a strict subset of the variables.  
\end{abstract}


\section{Introduction} \label{sec:intro}

This paper contributes to a growing literature on efficient estimation of causal effects under causal graphical models \citep{rotnitzky2020efficient,bhattacharya2020semiparametric,smucler2021efficient,guo2020linear,henckel2019graphical,witte2020efficient,kuipers2022var}. We consider estimating the interventional mean of an outcome associated with a point exposure treatment when a nonparametric causal graphical model, represented by a directed acyclic graph, is assumed. Such a causal model induces a semiparametric model on the factual data law known as a Bayesian network, which associates each vertex of the graph with a random variable. Under the Bayesian network model, every variable is conditionally independent of its non-descendants given its parents in the graph.
Further, under the causal graphical model, the interventional mean is identified by a smooth functional of the factual data law given by the g-formula \citep{robins1986new}. This functional is the mean of the outcome taken with respect to a truncated law which agrees with the factual law except that the probability of treatment given its parents in the graph is replaced by a point mass  at the intervened level of the treatment. The semiparametric variance bound for this functional under the induced Bayesian network model gives the lowest benchmark for the asymptotic variance of any regular estimator of the functional and, as such, it quantifies the efficiency with which, under regularity conditions, one can hope to estimate the interventional mean under the model without posing additional assumptions. 

\citet{rotnitzky2020efficient} identified a class of directed acyclic graphs under which the semiparametric variance bound for the interventional mean is equal to the variance bound under a simpler causal graphical model, which is a directed acyclic graph consisting of the treatment, the outcome and a special set of covariates known as the optimal adjustment set \citep{henckel2019graphical}.
This implies that all the remaining variables in the original graph are uninformative in that failure to measure them has no impact on the efficiency with which one can hope to estimate the interventional mean. 
However, this earlier work left unanswered the question of identifying uninformative variables in an arbitrary directed acyclic graph that does not belong to their special class, which is the goal of this paper. 

We prove theoretical results that can guide practitioners in the design and analysis of an observational or sequentially randomized study. First, at the stage of designing a study, it informs which variables should be measured for optimally estimating the effect of interest. Designers of a study often employ directed acyclic graphs to incorporate substantive causal assumptions, including hypotheses on potential confounders and causal paths \citep[\S6]{hernan2020causal}. Our \cref{thm:criteria} provides a graphical criterion that allows the designer to read off from the graph the set of informative variables, which is the minimal set of variables to measure that permits estimating the effect of interest with maximum efficiency. 
This is useful because the cost associated with measuring uninformative variables can be saved. 

Second, for analyzing a study, our \cref{alg:graph} produces a reduced graph that assists the data analyst in constructing an efficient estimator of the effect of interest. The reduced graph is a directed acyclic graph that only contains informative variables. As formalized in \cref{thm:algo}, the reduced graph encodes all the modeling constraints required for optimally estimating the effect. In fact, among all the possible ways to identify the effect from data, we show that the g-formula associated with the reduced graph is the most efficient. This leads to developing efficient estimators that involve the fewest number of variables, and presumably, the fewest number of nuisance parameters. Even when such an estimator is considerably simpler than an efficient estimator constructed using the full graph and full data, there is no loss in performance; see \cref{sec:app-simu} for a simulation example. Finally, the whole process of variable elimination, graph reduction and deriving the associated g-formula is automated by our R package \texttt{reduceDAG}.

\section{Motivation} \label{sec:motivation}
To motivate the development in this paper, consider the causal agnostic graphical model \citep{spirtes2000causation,robins2010alternative} represented by graph $\g$ in \cref{fig:example}(a). Suppose $Y$ is an outcome and $A$ is a discrete treatment whose causal effect on $Y$ we are interested in estimating. The causal model implies the Bayesian network model on the factual data, denoted as $\model(\g,\Vset),$ for the law of $\Vset=\{A,Y,I_{1},O_{1},W_{1},W_{2},W_{3},W_{4}\}$, which is defined by the sole restriction that the joint density of $V$, with respect to some dominating measure, factorizes as 
\begin{equation*} 
p(v)=p(y\mid a,o_{1})\,p(a\mid i_{1})\,p(i_{1}\mid w_{4})\,p(o_{1}\mid
w_{4})\,p(w_{4}\mid w_{2},w_{3})p(w_{3})\,p(w_{2}\mid w_{1})\,p(w_{1}).
\end{equation*}
Each factor is either a marginal density if $V_j$ has no parent in $\g$ or a conditional density of the form $p\{v_{j}\mid \Pa(v_{j},\g)\}$, where $\Pa(v_{j},\g)$ denotes the set of parents of $V_j$ in $\g$.
These densities are unrestricted under model $\model(\g,\Vset) $ and they parameterize the model.

\begin{figure}[!htb]
\centering
\begin{tikzpicture}
\tikzset{rv/.style={circle,inner sep=1pt,fill=gray!20,draw,font=\sffamily}, 
sv/.style={circle,inner sep=1pt,fill=gray!20,draw,font=\sffamily,minimum size=1mm}, 
node distance=12mm, >=stealth, every node/.style={scale=0.9}}
\begin{scope}
\node[name=A, rv]{$A$};
\node[name=Y, rv, right of=A]{$Y$};
\node[name=I, rv, above of=A]{$I_1$};
\node[name=O, rv, above of=Y]{$O_1$};
\node[name=W1, rv, above of=I, xshift=6mm, yshift=-5mm]{$W_4$};
\node[name=W2, rv, above of=I, yshift=3mm]{$W_3$};
\node[name=W3, rv, above of=O, yshift=3mm]{$W_2$};
\node[name=W4, rv, right of=W3]{$W_1$};
\node[below of=A, xshift=6mm, yshift=4mm]{(a) $\g$};
\draw[->, very thick, color=blue] (A) -- (Y);
\draw[->, very thick, color=blue] (I) -- (A);
\draw[->, very thick, color=blue] (O) -- (Y);
\draw[->, very thick, color=blue] (W1) -- (I);
\draw[->, very thick, color=blue] (W1) -- (O);
\draw[->, very thick, color=blue] (W2) -- (W1);
\draw[->, very thick, color=blue] (W3) -- (W1);
\draw[->, very thick, color=blue] (W4) -- (W3);
\end{scope} \begin{scope}[xshift=3.5cm] \node[name=A, rv]{$A$};
\node[name=Y, rv, right of=A]{$Y$};
\node[name=W1, rv, above of=A]{$W_4$};
\node[name=O, rv, above of=Y]{$O_1$};
\node[name=W2, rv, above of=W1, yshift=-3mm, xshift=-6mm]{$W_3$};
\node[name=W3, rv, right of=W2]{$W_2$};
\node[below of=A, xshift=6mm, yshift=4mm]{(b) $\g'$};
\draw[->, very thick, color=blue] (A) -- (Y);
\draw[->, very thick, color=blue] (W1) -- (A);
\draw[->, very thick, color=blue] (O) -- (Y);
\draw[->, very thick, color=blue] (W1) -- (O);
\draw[->, very thick, color=blue] (W2) -- (W1);
\draw[->, very thick, color=blue] (W3) -- (W1);
\end{scope} \begin{scope}[xshift=6.5cm] \node[name=A, rv]{$A$};
\node[name=Y, rv, right of=A]{$Y$};
\node[name=O, rv, above of=Y]{$O_1$};
\node[name=W2, rv, above of=A, yshift=8mm, xshift=-6mm]{$W_3$};
\node[name=W3, rv, right of=W2]{$W_2$};
\node[below of=A, xshift=6mm, yshift=4mm]{(c) $\g^{\ast}$};
\draw[->, very thick, color=blue] (A) -- (Y);
\draw[->, very thick, color=blue] (O) -- (Y);
\draw[->, very thick, color=blue] (W2) -- (O);
\draw[->, very thick, color=blue] (W3) -- (O);
\draw[->, very thick, color=blue] (W2) -- (A);
\draw[->, very thick, color=blue] (W3) -- (A);
\draw[->, very thick, color=blue] (O) -- (A);
\end{scope} \begin{scope}[xshift=9.5cm] \node[name=A, rv]{$A$};
\node[name=Y, rv, right of=A]{$Y$};
\node[name=O, rv, above of=Y]{$O_1$};
\node[name=W2, rv, above of=A, yshift=8mm, xshift=-6mm]{$W_3$};
\node[name=W3, rv, right of=W2]{$W_2$};
\node[below of=A, xshift=6mm, yshift=4mm]{(d)};
\draw[->, very thick, color=blue] (A) -- (Y);
\draw[->, very thick, color=blue] (O) -- (Y);
\draw[->, very thick, color=blue] (W2) -- (O);
\draw[->, very thick, color=blue] (W3) -- (O);
\draw[->, very thick, color=blue] (W2) -- (A);
\draw[->, very thick, color=blue] (W3) -- (A);
\draw[<->, very thick, color=red] (O) -- (A);
\end{scope}\begin{scope}[xshift=12cm] 
\node[name=A, rv]{$A$};
\node[name=Y, rv, right of=A]{$Y$};
\node[name=I, rv, above of=A]{$I_1$};
\node[name=O, rv, above of=Y]{$O_1$};
\node[name=W1, rv, above of=I, xshift=6mm, yshift=-5mm]{$W_4$};
\node[name=W2, rv, above of=I, yshift=3mm]{$W_3$};
\node[name=W3, rv, above of=O, yshift=3mm]{$W_2$};
\node[name=W4, rv, right of=W3]{$W_1$};
\node[below of=A, xshift=6mm, yshift=4mm]{(e) $\grev$};
\draw[->, very thick, color=blue] (A) -- (Y);
\draw[->, very thick, color=blue] (I) -- (A);
\draw[->, very thick, color=blue] (O) -- (Y);
\draw[->, very thick, color=blue] (W1) -- (I);
\draw[->, very thick, color=blue] (W1) -- (O);
\draw[->, very thick, color=blue] (W2) -- (W1);
\draw[->, very thick, color=blue] (W3) -- (W1);
\draw[->, very thick, color=red] (W3) -- (W4);
\end{scope} \end{tikzpicture}
\caption{Causal graphs involved in the motivating example: (a) the original graph $\g$, where variables $\{I_1, W_1, W_4\}$ are uninformative, among which $\{I_1, W_1\}$ are redundant; (b) graph $\g'$ is obtained by projecting out the redundant variables $\{I_1, W_1\}$ from $\g$; (c) the reduced graph $\g^{\ast}$ that projects out all the uninformative variables using \cref{alg:graph}; (d) the latent projection \citep{verma1991equivalence} of $\g$ that marginalizes over $\{I_1, W_1, W_4\}$, where a bidirected edge between $A$ and $O$ is introduced to due to confounder $W_4$; (e) graph $\grev$ is causal Markov equivalent to $\g$, from which $\{I_1, W_1\}$ can be identified as redundant and hence uninformative.}
\label{fig:example}
\end{figure}

If $p(a\mid i_{1})>0$ for all $i_{1}$ in the range of $I_{1},$ the causal graphical model also implies that the joint density of the variables in the graph, when $A$ is intervened and set to $a$, is given by 
\begin{equation*}
p_{a}(v)=J_{a}(v) \,p(y\mid a,o_{1})\,p(i_{1}\mid w_{4})\,p(o_{1}\mid w_{4})\,p(w_{4}\mid w_{2},w_{3})p(w_{3})\,p(w_{2}\mid w_{1})\,p(w_{1}),
\end{equation*}
where $J_a(v)$ is the indicator function that the $A$ component of $\Vset$ is equal to $a$ when $V$ takes value $v$.
In particular, the mean of the outcome when $A$ is intervened and set to $a$, which we shall refer to throughout as the interventional mean and denote with $\E Y(a)$, is given by
\begin{multline} \label{eqs:ex-g-raw} 
\Psi _{a}(P;\g) \equiv \sum_{y,o,i,w_{1},w_{2},w_{3},w_{4}}y\,p(y\mid
a,o_{1})\,p(i_{1}\mid w_{4})\,p(o_{1}\mid w_{4})\,p(w_{4}\mid
w_{2},w_{3})\,p(w_{3})\,  \\
\times p(w_{2}\mid w_{1})\,p(w_{1})
\end{multline}
if all the components of $\Vset$ are discrete; otherwise $\Psi _{a}(P;\g)$ is defined with the summation replaced by an integral with respect to the dominating measure. 
We call \cref{eqs:ex-g-raw} the g-formula associated with graph $\g$ \citep{robins1986new}.

Our goal is to determine the variables in vector $\Vset$ that can be disposed of without affecting the asymptotic efficiency with which we can hope to estimate $\Psi_{a}(P;\g)$. With this goal in mind, we first observe that the term $p(i_1 \mid w_4) $ can be summed out from the right hand side of \cref{eqs:ex-g-raw} because $i_1$ does not appear in the conditioning set of any other conditional densities. Writing $p(w_{2}\mid w_{1})\,p(w_{1})=p(w_{1},w_{2})$, we also observe that we can sum out $w_{1}$ from \cref{eqs:ex-g-raw} as well. We then conclude that $\Psi _{a}(P;\g)$ is equal to 
\begin{equation} \label{eqs:ex-g-prime}
\sum_{y,o_{1},w_{2},w_{3},w_{4}}y\,p(y\mid a,o_{1})\,p(o_{1}\mid w_{4})\,p(w_{4}\mid w_{2},w_{3})\,p(w_{2})\,p(w_{3}).  
\end{equation}
Next, we notice that because both $p(w_{2}\mid w_{1})$ and $p(w_{1})$ are unrestricted under model $\model(\g,\Vset)$, so is $p\left( w_{2}\right)$. In fact, all the densities that remain in \cref {eqs:ex-g-prime} are also unconstrained under the model. Because the data on $\{I_{1},W_{1}\}$ does not help us estimate these densities, we conclude that we can discard $\{I_{1},W_{1}\}$ without affecting the efficiency in estimating $\Psi _{a}(P;\g)$. We recognize that expression \eqref{eqs:ex-g-prime} is precisely the g-formula $\Psi_{a}(P'; \g')$, where $\g'$ is the graph in \cref{fig:example}(b) and $P'$ is the marginal law of $\Vset' \equiv \Vset \setminus \{I_{1},W_{1}\}$. Moreover, under both $\model(\g, \Vset)$ and $\model(\g', \Vset')$, the densities in \cref{eqs:ex-g-prime} are unrestricted. 
Hence, as far as the efficient estimation of $\Psi_{a}(P; \g)$ is concerned, we can ignore $\{I_{1},W_{1}\}$ and pretend that our problem is to estimate the g-formula $\Psi_{a}(P';\g')$ based on a random sample of $V'$, under the assumption that $P'$ belongs to $\model(\g',\Vset')$.

In \cref{sec:g-formula}, we will review the notion of causal Markov equivalent graphs with respect to the effect of $A$ on $Y$. These are graphs that encode the same Bayesian network model and their associated g-formulae coincide under the model. For instance, graphs $\g$ and $\check{\g}$ in \cref{fig:example} are causal Markov equivalent. 
We will show that a variable for which there exists some causal Markov equivalent graph in which all directed paths towards $Y$ intersect $A$, such as $I_{1}$ in our example, is uninformative for estimating $\Psi_a(P; \g)$. Similarly, a variable that is non-ancestral to $Y$ in some causal Markov equivalent graph, such as $W_{1}$ in our example, is also uninformative. We refer to these two types of variables as redundant. 

Further, by traversing graphs in the causal Markov equivalent class, one can see that $\{I_1, W_1\}$ are the only redundant variables.  
One might believe that all variables in $\Vset'$ are needed to construct an asymptotically efficient estimator of $\Psi_a(P; \g)$. 
For instance, suppose $\Vset$ is discrete. Consider the maximum likelihood estimator $\Psi_{a}(\widehat{\mathbb{P}}_{n}';\g')$ with
\[ \widehat{\mathbb{ P}}_{n}'(a,y,o_{1},w_{4},w_{3},w_{2}) \equiv \mathbb{P}_{n}(y \mid a,o_1)\,\mathbb{P}_n(a \mid w_{4})\,\mathbb{P}_{n}(w_{4}\mid w_{2},w_{3})\,\mathbb{P}_{n}(o_{1}\mid w_{4})\,\mathbb{P}_{n}(w_{2})\,\mathbb{P}_{n}(w_{3}),\]
where $\mathbb{P}_{n}(\cdot \mid \cdot)$ and $\mathbb{P} _{n}(\cdot)$ denote, respectively, the empirical conditional and marginal probability operators. Law $\hat{\mathbb{P}}'_{n}$ is the maximum likelihood estimator for $P'$ under $\model(\g',\Vset')$. Clearly, one needs every variable in $\Vset'$ to compute this estimator. 

Surprisingly, in \cref{sec:graph} we will show that, even without using the data on $W_{4}$, we can construct an estimator with the same limiting distribution as the maximum likelihood estimator. 
Specifically, let $ P^{\ast}$ denote the marginal law of $\Vset^{\ast} \equiv \Vset^{\prime}\setminus \{ W_{4}\}$ for $\Vset' \sim P'$, and let $\g^{\ast}$ be the graph over $\Vset^{\ast}$ shown in \cref{fig:example}(c). We will show that the maximum likelihood estimator of the g-formula
\begin{equation} \label{eqs:ex-g-eff}
\Psi_{a}(P^{\ast };\g^{\ast }) \equiv \sum_{y,o_{1},w_{2},w_{3}}y\,p(y\mid a,o_{1})\,p(o_{1}\mid w_{2},w_{3})\,p(w_{2})\,p(w_{3}) 
\end{equation}
with respect to the Bayesian network model represented by $\g^{\ast}$ is asymptotically equivalent to the aforementioned $\Psi_{a}(\widehat{\mathbb{P}}_{n}';\g')$ under every law $P'$ in model $\model(\g', \Vset')$. The estimator based on \cref{eqs:ex-g-eff} does not require measuring $W_4$. This result can be useful even when $W_4$ is already measured but incorporating it into estimation is difficult, for example when $W_4$ is continuous while all the other variables are discrete. In such cases, using the maximum likelihood estimate of \cref{eqs:ex-g-eff} circumvents estimating $p(w_4 \mid w_2, w_3)$ and $p(o \mid w_4)$, which typically requires smoothing. 

More generally, we will show that
(i) when Bayesian networks are defined on a sufficiently large state space, graph $\g^{\ast}$ represents the marginal model of the law $P^{\ast}$ over $\Vset^{\ast}$ induced by $\model(\g', \Vset')$ or, equivalently, by the original $\model(\g, \Vset)$; 
(ii) $\Psi_a(P^{\ast}; \g^{\ast}) = \Psi_a(P; \g)$ for every $P \in \model(\g; \Vset)$ under a positivity condition introduced in \cref{sec:g-formula}; 
(iii) the semiparametric variance bound for $\Psi_{a}(P^\ast;\g^\ast)$ with respect to $\model(\g^{\ast},\Vset^{\ast})$ and the bound for $\Psi_{a}(P;\g)$ with respect to $\model(\g, \Vset)$ coincide. 
Therefore, for estimating the interventional mean, not only is $W_{4}$ asymptotically uninformative but, moreover, we can discard $\g$ and pretend it is the graph $\g^{\ast}$ that we started with. The same can be said for estimating the average treatment effect, e.g., $\E Y(1) - \E Y(0)$ when $A$ is binary. 
Also, $\g^{\ast}$ is different from the latent projection \citep{verma1991equivalence} of $\g$ onto $\Vset^{\ast}$, which introduces bidirected edges when a confounder is marginalized over; compare \cref{fig:example}(c) and (d).

Conceptually, the preceding results can be interpreted as follows. It is well-known that the Bayesian network $\model(\g, \Vset)$ is the set of laws that obey the conditional independencies implied by d-separations with respect to $\g$. Our results imply that estimating $\Psi_{a}(P;\g)$ under a supermodel $\bar{\model}$, which is specified by those conditional independencies in $\model(\g, \Vset)$ that do not involve variables $\{I_1, W_1, W_4\}$, is no more difficult than estimating it under $\model(\g, \Vset)$. In other words, $\model(\g, \Vset)$ is a least favorable submodel of $\bar{\model}$ \citep[\S 25.3]{van2000asymptotic} in the sense that the extra constraints it encodes are uninformative for the target parameter. 

Furthermore, in \cref{sec:uninfo}, we show that no variable can be further eliminated from $\Vset^{\ast}$ without impairing efficiency at some law in $\model(\g ,\Vset)$. It can then be argued that the g-formula associated with $\g^{\ast}$, such as \eqref{eqs:ex-g-eff}, is an irreducible, efficient identifying formula for $\Psi_a(P;\g)$. In particular, this implies that when all components of $\Vset$ are discrete, the plugin estimator of any other identifying formula either depends on a strict superset of $\Vset^{\ast}$, as is the case with \cref{eqs:ex-g-prime}, or has an asymptotic variance strictly greater than the Cram\'{e}r--Rao bound under some law in $\model(\g, \Vset)$. As an example of the latter, consider the class of adjustment formulae 
\begin{equation} \label{eqs:ex-backdoor}
\Psi_{a,L}^{\mathrm{ADJ}}(P; \g)\equiv \sum_{y,l} y\,p(y\mid a,L=l)\,p(l),
\end{equation}
which agrees with $\Psi_a(P;\g)$ in $\model(\g, \Vset)$, where $L$ is any set of variables non-descendant to $A$ that blocks all the back-door paths between $A$ and $Y$ in $\g$ \citep{pearl1993bayesian}, e.g., $L=\{O_{1}\}$, $L=\{I_{1}\}$, $L=\{W_{4}\}$ or $L = \{I_1,W_4\}$. These formulae lead to inefficient estimators $\Psi_{a,L}^{\mathrm{ADJ}}(\mathbb{P}_n; \g)$ when plugging in the empirical measure, as is confirmed by simulations in \cref{sec:app-simu}.

\section{Technical background} \label{sec:background}
\subsection{Relation to optimal adjustment} \label{sec:opt-adj}
Our problem is different from optimal adjustment. Our efficiency bound is defined relative to all regular, asymptotically linear estimators of $\Psi_a(P; \g)$ under the Bayesian network model $\model(\g, \Vset)$. In contrast, the literature on optimal adjustment (e.g., \citealp{kuroki2003covariate,hahn2004functional,rotnitzky2020efficient,henckel2019graphical}) restricts the class of estimators to those that estimate nonparametric target \cref{eqs:ex-backdoor} without imposing any conditional independence restrictions, and seeks one with the maximum efficiency within the class, which is called the optimal adjustment estimator. By definition, the asymptotic variance bound we consider is smaller than or equal to the asymptotic variance of the optimal adjustment estimator. For cases where the optimal adjustment estimator does not achieve the asymptotic variance bound we consider, see our motivating example in \cref{fig:example} and \cref{ex:two-Os,ex:three-Ms,ex:long} in \cref{sec:ex}. 

As mentioned in the introduction, under a Bayesian network model, there are certain graphs, characterized by \citet[Theorem 19]{rotnitzky2020efficient}, where the optimal adjustment estimator achieves the asymptotic variance bound considered here. In this paper we study general graphs beyond these cases. 

\subsection{Bayesian network, directed acyclic graph and vertex sets} \label{sec:sets}
For technical reasons that will be explained shortly, we define a Bayesian network model on a larger state space than typically required. For every random variable $V_j \in \Vset$, let its state space be 
\begin{equation} \label{eqs:state-space}
\X_j = \mathbb{R}^{d_j} \dcup \mathbb{W}, \quad d_j \geq 1, \quad \mathbb{W} = \{\omega_1,\omega_2,\dots\}, 
\end{equation}
where $\dcup$ denotes a disjoint union and the set $\mathbb{W}$ is a collection of symbols isomorphic to the natural numbers. That is, the state space $\X_j$ allows $V_j$ to be potentially Euclidean or discrete, or a mixed type of both, prior to observing the data on $V_j$. In \cref{sec:app-state-space}, measure $\mu_j$ and $\sigma$-algebra $\F_j$ for every $V_j$ are defined accordingly. The Bayesian network model is the set of probability measures on $\left(\X \equiv \times_{j: V_j \in \Vset} \X_j, \F \equiv \times_{j: V_j \in \Vset} \F_j \right)$ that factorize according to the graph, i.e., 
\begin{equation} \label{eqs:factorization}
\model(\g,\Vset)\equiv \left\{ P:\frac{\dd P}{\dd \mu }\left( v\right) \equiv p(v)=\prod_{j: V_{j}\in \Vset}p\{v_{j}\mid \Pa(v_{j},\g)\}\right\} ,
\end{equation}
where the density $p$ is taken with respect to the dominating measure $\mu \equiv \times_{j: V_j \in \Vset} \mu_j$. The symbol $\Pa(v_{j},\g)$ denotes the value taken by the set of parents of $V_{j}$ with respect to $\g$ when $V=v$. By the equivalence between factorization and the global Markov property, $\model(\g,\Vset)$ coincides with the set of laws that obey the conditional independences implied by d-separations with respect to $\g$; in addition, $\model(\g,\Vset)$ is also the set of laws that satisfy the local Markov property, namely a variable is independent of its non-descendants given its parents; see, e.g., \citet[Theorem 3.27]{lauritzen1996graphical}. We also refer to $\model(\g,\Vset)$ as the model represented by $\g$. 

\smallskip \begin{remark} \label{remark:marginal}
We introduce \cref{eqs:state-space} to ensure that the state space of every variable is sufficiently large so that it is essentially no different from an unconstrained state space. Consequently, the notion of an induced marginal model in the rest of the paper aligns with the notion of a marginal model typically used in the literature, where the state space of the marginalized variables is unspecified or unrestricted; see, e.g., \citet[Definition 6]{evans2016graphs}.
Following the discussion in \citet[\S 2.11]{cencov2000statistical}, a sufficiently large state space can be ensured if each $\X_j$ at least contains an interval of the real line. We make this technical requirement on the state space to rule out undesired, e.g., reduced-rank, constraints on the induced model when marginalizing out a variable with a finite or small state space \citep{mond2003stochastic}. 
\end{remark}

\begin{remark} \label{remark:mle}
The definition above by no means precludes discrete distributions that only put mass on vectors consisting of symbols in $\mathbb{W}$. In fact, when the data are discrete, the maximum likelihood estimate is well-defined and coincides with the maximum likelihood estimate under the commonly used model with $\X_j = \mathbb{W}$ for every $V_j \in \Vset$. For technical reasons, model $\model(\g, \Vset)$ considered here is larger than the commonly used Bayesian network model, but the difference is inconsequential in terms of data analysis.
\end{remark}

\medskip Throughout, we use upper-case letters to denote vertices of a graph or the random variables they represent. Lower-case letters are reserved for indices or values taken by random variables. 
We use standard notation for graphical models, summarized in \cref{sec:app-graph-notation}. Among others, we say that path $p$ from $V_1$ to $V_k$ is causal if it is of the form $V_{1}\rightarrow \dots \rightarrow V_{k}$. The notation $V_i \mapsto V_j$ is a shorthand for $V_i \in \An(V_j)$. 

For disjoint sets $A$, $B$ and $C$, we use $A \indep B \mid C$ to denote conditional independence between $A$ and $B$ given $C$ under a given law, and we use $A \indep_{\g} B \mid C$ to denote d-separation between $A$ and $B$ given $C$ in graph $\g$. For d-separation, we allow $A \cap C \neq \emptyset$ and $B \cap C \neq \emptyset$, in which case $A \indep_{\g} B \mid C$ is interpreted as $A \setminus C \indep_{\g} B \setminus C \mid C$. We also use the convention that $\emptyset \indep_{\g} B \mid C$ for any sets $B$ and $C$. Conditional independence and d-separation share similar properties: the former satisfies semi-graphoid axioms, while the latter satisfies the stronger compositional graphoid axioms; see \citet[Theorem 1 and 11]{pearl1988book}.

Two directed acyclic graphs $\g$ and $\g'$ on the same vertex set $V$ are called Markov equivalent if $\model(\g,\Vset)=\model(\g',\Vset)$.
It is well-known that two graphs are Markov equivalent if and only if they share the same adjacencies and unshielded colliders \citep{verma1991equivalence,andersson1997characterization}. Further, a Markov equivalence class can be graphically represented by a completed partially directed acyclic graph, also known as an essential graph \citep{meek1995causal,andersson1997characterization}.

\medskip 
\begin{assumption} \label{assump:A-goes-to-Y}
In directed acyclic graph $\g$, $A \mapsto Y$.
\end{assumption}

We will make this assumption throughout; otherwise the model already assumes $A$ has no effect on $Y$. As we will see, the information carried by a variable depends crucially on its ancestral relations with respect to treatment $A$ and outcome $Y$. To ease the exposition, we introduce the following taxonomy of vertices, which is illustrated in \cref{fig:taxonomy}(a).

\medskip \noindent (i) Non-ancestors of $Y$: $\Nset(\g) \equiv \Vset \setminus \An(Y,\g)$.

\noindent (ii) Indirect ancestors of $Y$: $\Iset(\g) \equiv \left\{V_{j} \in \Vset:\, V_{j} \neq A,\, V_j \mapsto Y \text{ only through } A \right\}$. These are also conditional instruments given $\Pa(\Iset,\g) \setminus \Iset$ \citep{didelez2007mendelian}.

\noindent (iii) Baseline covariates: non-descendants of $A$, but ancestors of $Y$ not only through $A$, i.e.,
\begin{equation} \label{eqs:Wset}
\Wset(\g)\equiv \left\{V_{j} \in \Vset:\, A \not\mapsto V_{j},V_{j}\mapsto Y,\,V_{j}\notin \Iset (\g) \right\}.
\end{equation}
In contrast to $\Iset(\g)$, for each $W_{j}\in \Wset(\g)$ there is a
causal path from $W_{j}$ to $Y$ that does not contain $A$.

\noindent (iv) Mediators: $\Mset(\g)\equiv \{V_{j} \in \Vset:\, V_{j}\neq A,\,A\mapsto
V_{j}\mapsto Y\}$. These are the variables that lie on the causal paths between $A$
and $Y$. With a slight abuse of the term mediators, the set $\Mset(\g)$ also contains $Y$.

\medskip It follows that the set of variables is partitioned as $\Vset=\{A\}\dcup \Nset(\g)\dcup\Iset(\g)\dcup\Wset(\g)\dcup\Mset(\g)$. The following subset of $\Wset$ is also important: the optimal adjustment set \citep{henckel2019graphical} 
\begin{equation} \label{eqs:Oset}
\Oset(\g)\equiv \Pa\{\Mset(\g),\g\} \setminus \left\{ \Mset(\g) \cup \{A\} \right\}.
\end{equation}
The set $\Oset(\g)$ consists of the parents of mediators that are not themselves mediators or the treatment; see \citet{witte2020efficient} for other
characterizations. By definition it can be empty. 
The set of baseline covariates $\Wset(\g)$ is related to its subset $\Oset(\g)$ by the
following lemma; further properties of $\Oset(\g)$ can be found in the next subsection.

\begin{lemma} \label{lem:osetandw}
Under \cref{assump:A-goes-to-Y}, $W(\g) = \An\{\Oset(\g), \g\}$.
\end{lemma}

We also define the following subset of $\Oset(\g)$ that will be useful later:
\begin{equation*}
\Omin(\g) \,\equiv\, \text{the inclusion minimal $\Oset^{\prime }\subseteq \Oset$}(\g) \text{ such that $A\indep_{\g}\Oset(\g)\setminus \Oset^{\prime }\mid \Oset ^{\prime }$}.  
\end{equation*}
The intersection property of d-separation ensures that $\Omin(\g)$ is uniquely defined; see \citet[Lemma 7, Appendix]{rotnitzky2020efficient}.

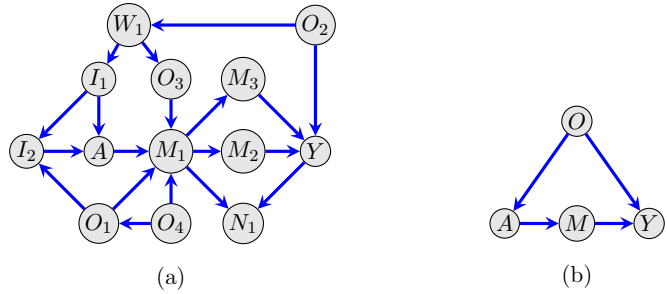
\begin{figure}[!htb]
\centering
\begin{tikzpicture}
\tikzset{rv/.style={circle,inner sep=1pt,fill=gray!20,draw,font=\sffamily}, 
redv/.style={circle,inner sep=1pt,fill=white,draw,dashed,font=\sffamily}, 
ov/.style={circle,inner sep=1pt,fill=gray!20,draw=red,thick,font=\sffamily}, 
sv/.style={circle,inner sep=1pt,fill=gray!20,draw,font=\sffamily,minimum size=1mm}, 
node distance=12mm, >=stealth, every node/.style={scale=0.8}}
\begin{scope}
\node[rv] (A) {$A$};
\node[rv, right of=A] (M1) {$M_1$};
\node[rv, right of=M1] (M2) {$M_2$};
\node[rv, right of=M2] (Y) {$Y$};
\node[rv, left of=A] (I2) {$I_2$};
\node[rv, below of=A] (W2) {$O_1$};
\node[rv, below of=M2] (N1) {$N_1$};
\node[rv, above of=A] (I1) {$I_1$};
\node[rv, above of=M1] (W3) {$O_3$};
\node[rv, above of=M2] (M3) {$M_3$};
\node[rv, below of=M1] (O4) {$O_4$};
\node[rv, above of=I1, xshift=5mm, yshift=-3mm] (W1) {$W_1$};
\node[rv, above of=Y, yshift=9mm] (W4) {$O_2$};
\node[below=2mm of O4] (lab1) {(a)};
\draw[->, very thick, color=blue] (I2) -- (A);
\draw[->, very thick, color=blue] (A) -- (M1);
\draw[->, very thick, color=blue] (M1) -- (M2);
\draw[->, very thick, color=blue] (M2) -- (Y);
\draw[->, very thick, color=blue] (W2) -- (I2);
\draw[->, very thick, color=blue] (W2) -- (M1);
\draw[->, very thick, color=blue] (M1) -- (N1);
\draw[->, very thick, color=blue] (Y) -- (N1);
\draw[->, very thick, color=blue] (I1) -- (I2);
\draw[->, very thick, color=blue] (I1) -- (A);
\draw[->, very thick, color=blue] (W1) -- (I1);
\draw[->, very thick, color=blue] (W1) -- (W3);
\draw[->, very thick, color=blue] (W3) -- (M1);
\draw[->, very thick, color=blue] (W4) -- (W1);
\draw[->, very thick, color=blue] (W4) -- (Y);
\draw[->, very thick, color=blue] (M1) -- (M3);
\draw[->, very thick, color=blue] (M3) -- (Y);
\draw[->, very thick, color=blue] (O4) -- (W2);
\draw[->, very thick, color=blue] (O4) -- (M1);
\end{scope} \begin{scope}[xshift=5cm]
\node[rv, right=3cm of N1] (A) {$A$};
\node[rv, right of=A] (M) {$M$};
\node[rv, right of=M] (Y) {$Y$};
\node[rv, above of=M, yshift=0.5cm] (O) {$O$};
\node[below=2mm of M] (lab2) {(b)};
\draw[->,very thick, color=blue] (A) to (M);
\draw[->,very thick, color=blue] (M) to (Y);
\draw[->,very thick, color=blue] (O) to (A);
\draw[->,very thick, color=blue] (O) to (Y);
\end{scope}
\end{tikzpicture}
\caption{(a) An illustration of the taxonomy of vertices. $A$ is the treatment and $Y$ is the outcome. Vertex $\Nset = \{N_1\}$ is non-ancestral to $Y$. Set $\Iset = \{I_1,I_2\}$ consists of indirected ancestors of $Y$, which are conditional instruments given $\{W_1, O_1\}$. We have $\Wset = \{W_1, O_1, O_2, O_3,O_4\}$, of which the subset $\Oset = \{O_1, O_2, O_3, O_4\}$ is the optimal adjustment set; further, $\Omin = \{O_1, O_2, O_3\}$. Finally, $\Mset = \{M_1, M_2,M_3,Y\}$ is the set of mediators. (b) An example with multiple identifying formulae: the g-formula \cref{eqs:front-back-door-g}, the back-door formula \cref{eqs:front-back-door-back} and the front-door formula \cref{eqs:front-back-door-front}.} 
\label{fig:taxonomy}
\end{figure}

\subsection{Causal graphical model and the g-formula} \label{sec:g-formula}
Throughout, we assume a causal agnostic graphical model \citep{spirtes2000causation,robins2010alternative} represented by a directed acyclic graph $\g$ on vertex set $\Vset$, where $A \in \Vset$ is a discrete treatment and $Y \in \Vset$ is the outcome of interest. We also impose \cref{assump:A-goes-to-Y} on $\g$. The causal model implies that the law $P$ of the factual variables $V$ belongs to the Bayesian network model $\model(\g, \Vset)$ defined in \cref{eqs:factorization}.

Under \cref{assump:positivity} introduced below, the causal graphical model further posits that when $A$ is intervened and set to level $a$, the density of the variables in the graph is 
\begin{equation} \label{eqs:g-dist}
p_{a}\left( v\right) \equiv J_{a}(v) \prod_{V_{j}\in \Vset \setminus \{A\}}p\{v_{j}\mid \Pa(v_{j},\g)\},  
\end{equation}
where $J_a(v)$ is the indicator of the $A$ component of $V$ being equal to $a$ when $V=v$.
The right-hand side of \cref{eqs:g-dist} is known as the g-formula \citep{robins1986new}, the manipulated distribution formula \citep{spirtes2000causation} or the truncated factorization formula \citep{pearl2009causality} in the literature. Our target of inference, the interventional mean, which we denote with $\E Y(a)$, is therefore given by
\begin{equation} \label{eqs:g-formula}
\Psi_{a}(P;\g) \equiv \sum_{y,\, v_{j}:V_{j}\in \Vset \setminus \{A,Y\}} y \prod_{j: V_j \in \Vset\setminus \{A\}}p\left\{v_{j} \mid \left. \Pa(v_{j},\g)\right\vert_{A=a} \right\},  
\end{equation}
if all components of $V$ are discrete; otherwise $\Psi_{a}(P;\g)$ is defined with the summation replaced by an integral with respect to the dominating measure $\mu$; see also \cref{eqs:factorization}. The symbol $\left. \Pa(v_j,\g)\right\vert_{A=a}$ indicates that if $A\in \Pa(V_j,\g)$, then the value taken by $A$ when $V_j=v_j$ is set to $a$. We refer to $\Psi_a(\cdot; \g): \model(\g, \Vset) \rightarrow \mathbb{R}$ as the g-functional. 

\medskip \begin{assumption}[Positivity] \label{assump:positivity}
There exists $\varepsilon>0$, which can depend on $P$, such that the conditional probability $P\{A = a \mid \Pa(A, \g)\} > \varepsilon $ $P$-almost surely. 
\end{assumption}
\smallskip \noindent By the local Markov property, this assumption implies $P(A=a\mid \Lset)>\varepsilon$ $P$-almost surely for every $\Lset \subset \Vset$ that is non-descendant to $A$.

\medskip For the rest of this paper, $\model_0(\Vset)$ denotes the set of all laws over $\Vset$ restricted only by the inequality in \cref{assump:positivity}. Accordingly, a Bayesian network $\model(\g;\Vset)$ should be understood as the intersection of the original definition \cref{eqs:factorization} with such $\model_0(\Vset)$. We impose \cref{assump:positivity} because otherwise the semiparametric variance bound for the g-functional is undefined. 

\medskip \begin{definition}[Identifying formula] \label{def:id-formula} 
Fix a model $\model(\Vset) \subseteq \model_0(\Vset)$ and a functional $\gamma(P): \model(\Vset) \rightarrow \mathbb{R}$. The functional $\chi(P): \model_0(\Vset) \rightarrow \mathbb{R}$ is an identifying formula for $\gamma(P)$ if $\chi(P) = \gamma(P)$ for every $P \in \model(\Vset)$. 
\end{definition} 

\medskip By the definition above, the natural extension $\Psi_a(P; \g): \model_0(\Vset) \rightarrow \mathbb{R}$ according to \cref{eqs:g-formula}, called the g-formula associated with graph $\g$, is an identifying formula for the g-functional.
However, because of conditional independences in a Bayesian network, one can typically derive more than one identifying formula.
As mentioned in \cref{sec:motivation}, the adjustment $\Psi_{a,L}^{\mathrm{ADJ}}(P; \g)$ given in \cref{eqs:ex-backdoor} based on a valid choice of $\Lset$ is also an identifying formula for $\Psi_a(P; \g)$. In particular, with discrete data, choosing $\Lset = \Oset(\g)$ for estimator $\Psi_{a,L}^{\mathrm{ADJ}}(\mathbb{P}_n)$ leads to the optimal adjustment, which achieves the smallest asymptotic variance among all valid choices of $\Lset$ \citep{rotnitzky2020efficient}; further, this choice is also optimal under the subclass of linear causal graphical models \citep{henckel2019graphical}. Here is another example of multiple identifying formulae.

\smallskip \begin{example} \label{ex:front-back-door}
Consider graph $\g$ in \cref{fig:taxonomy}(b). The g-functional associated with $\g$ is 
\begin{equation} \label{eqs:front-back-door-g}
\Psi_a(P; \g) = \sum_{y,m,o} y\,p(y \mid m, o)\, p(m \mid a)\, p(o).
\end{equation}
Under $\model(\g, \Vset)$, it agrees with the adjustment or back-door formula $\Psi^{\mathrm{ADJ}}_{a,O}(\cdot; \g): \model_0(\Vset) \rightarrow \mathbb{R}$,
\begin{equation} \label{eqs:front-back-door-back}
\Psi^{\mathrm{ADJ}}_{a,O}(P; \g) = \sum_{y,o} y \, p(y \mid a, o) \, p(o),
\end{equation}
and the front-door formula \citep{pearl1995causal} $\Psi^{\mathrm{FRONT}}_a(\cdot; \g): \model_0(\Vset) \rightarrow \mathbb{R}$, 
\begin{equation} \label{eqs:front-back-door-front}
\Psi^{\mathrm{FRONT}}_a(P; \g) = \sum_{y,m} y\,p(m \mid a) \sum_{a'} p(y \mid a', m)\,p(a').
\end{equation}
\end{example}

\medskip The notion of Markov equivalence is not directly applicable to our problem, as two Markov equivalent graphs may not admit the same identifying formula for the g-functional. 
This issue is fixed by the following refinement of Markov equivalence.

\begin{definition}[Causal Markov equivalence]\label{def:c-MEC} 
Two graphs $\g$ and $\g'$ are causal Markov equivalent with respect to the effect of $A$ on $Y$, denoted as  $\g \ceq \g'$, if $\g$ and $\g'$ are Markov equivalent and $\Psi_{a}(P;\g)=\Psi_{a}(P;\g')$ for all $P \in \model(\g, \Vset)$.
\end{definition}

\smallskip \citet{guo2020enumeration} showed that a causal Markov equivalence class can be represented by a maximally oriented partially directed acyclic graph and provided a polynomial-time algorithm to find the representation. In our context, where $|A|=|Y|=1$, the following is an alternative characterization.

\smallskip \begin{proposition} \label{prop:c-MEC-Oset} 
Let $\g$ and $\g^{\prime }$ be two directed acyclic graphs on vertex set $\Vset$, which contains the treatment $A$ and outcome $Y$. Suppose $\g$ and $\g'$ satisfy \cref{assump:A-goes-to-Y}. Graphs $\g$ and $\g'$ are causal Markov equivalent with respect to the effect of $A$ on $Y$ if and only if they are Markov equivalent and share the same optimal adjustment set defined in \cref{eqs:Oset}. 
\end{proposition}

\smallskip For example, graphs $\g$ and $\check{\g}$ in \cref{fig:example} are causal Markov equivalent.

\subsection{Efficient influence function, uninformative variables and
efficient identifying formulae} \label{sec:eif}

We now review the elements of semiparametric theory that are relevant to our derivations. An estimator $\widehat{\gamma}$ of a functional $\gamma(P) $ based on $n$ independent observations $\Vset^{(1)},\dots, \Vset^{(n)}$ drawn from $P$ is said to be asymptotically linear at $P$ if there exists a random variable $\gamma_{P}^{1}(V)$, called the influence function of $\widehat{\gamma}$ at $P$, such that $\E_{P} \gamma_{P}^{1}(\Vset)=0$, $\var_{P} \gamma_{P}^{1}(\Vset) <\infty $ and $n^{1/2}\left\{ \widehat{\gamma}-\gamma(P) \right\} =n^{-1/2}\sum_{i=1}^{n}\gamma_{P}^{1}(\Vset^{(i)}) + o_{p}(1)$ as $ n \rightarrow \infty$. For each asymptotically linear estimator $\widehat{\gamma}$, there exists a unique such $\gamma_{P}^{1}\left( {V}\right) $. It follows that $n^{1/2}\left\{ \widehat{\gamma }-\gamma \left( P\right) \right\} $ converges in distribution to a zero-mean normal distribution with variance $\var_{P} \gamma_{P}^{1}(\Vset)$.

Given a collection of probability laws $\model(\Vset)$ over $\Vset$, an estimator $\widehat{\gamma}$ of $\gamma(P)$ is said to be regular at $P$ if its convergence to $\gamma(P) $ is locally uniform at $P$ in $\model(\Vset)$. It is known that for a regular, i.e., pathwise-differentiable, functional $\gamma$, there exists a random variable, denoted by $\gamma_{P,\text{eff}}^{1}(\Vset)$ and called the efficient influence function of $ \gamma$ at $P$ with respect to $\model(\Vset)$, such that given any regular asymptotically linear estimator $\widehat{\gamma}$ of $\gamma$ with influence function $\gamma_{P}^{1}(\Vset)$, we have $\var_{P} \gamma_{P}^{1}(\Vset) \geq \var_{P}\gamma_{P,\text{eff}}^{1} (\Vset)$. If equality holds, then the estimator $\widehat{\gamma}$ is said to be locally semiparametric efficient at $P$ with respect to model $\model(\Vset)$. Further, it is said to be globally efficient if the equality holds for all $P$ in $\mathcal{M}(\Vset)$. 
When $\model(\Vset)$ is taken to be the nonparametric model $\model_0(\Vset)$, all regular asymptotically linear estimators have the same influence function, which therefore coincides with the efficient influence function with respect to $\model_0(\Vset)$. For ease of reference, we call it the nonparametric influence function and denote it by $\gamma_{P,\text{NP}}^{1}(\Vset)$. 
For more details, see \citet[Chapter 25]{van2000asymptotic}.

\medskip To define what it means for a variable to be uninformative, we need the next result.
For a law $P$ over $\Vset$ and $\Vset' \subseteq \Vset$, let $P(\Vset')$ denote the marginal law over $\Vset'$. Similarly, for model $\model(\Vset)$ or $\model(\g, \Vset)$, we use $\model(\Vset')$ or $\model(\g, \Vset')$ to denote the induced marginal model over $\Vset'$, i.e., $\model(\Vset') \equiv \{P(\Vset'): P \in \model(V)\}$ or $\model(\g, \Vset') \equiv \{P(\Vset'): P \in \model(\g, V)\}$; see also \cref{remark:marginal}. 

\smallskip \begin{lemma}[Proposition 17, \citealp{rotnitzky2020efficient}] \label{lem:eif-marg}
Let $\model(\Vset)$ be a semiparametric model for the law of a random vector $\Vset$. Suppose $\Vset'$ is a subvector of $\Vset$. Let $ \model(\Vset')$ be the induced marginal model over $\Vset'$.

Suppose $\gamma: \model(\Vset) \rightarrow \mathbb{R}$ is a regular functional with efficient influence function at $P$ equal to $\gamma_{P,\text{eff}}^{1}\left( {V}\right)$. 
Suppose there exists a regular functional $\chi: \model(\Vset') \rightarrow \mathbb{R}$ such that $\gamma(P) = \chi(P')$ for every $P \in \model(\Vset)$ and $P' \equiv P(\Vset')$.
Suppose, furthermore, that $\gamma_{P,\text{eff}}^{1}(\Vset)$ depends on $\Vset$ only through $\Vset'$. 
Let $\chi_{P^{\prime},\text{eff}}^{1}(\Vset')$ be the efficient influence function of $\chi(P')$ in model $\model(\Vset')$ at $P'$. 
Then, for every law $P\in \model(V)$ over $\Vset$ and its corresponding marginal law $P^{\prime }\in \model (V^{\prime })$ over $\Vset'$, $\gamma_{P,\text{eff}}^{1}(\Vset)$ and $\chi_{P',\text{eff}}^{1}(\Vset')$, as functions of $\Vset$ and $\Vset'$, respectively, are identical $P$-almost everywhere.
\end{lemma}

\smallskip This result tells us that to efficiently estimate $\gamma(P)$ under model $\model(V)$, we can discard the data on $V \setminus V'$ and recast the problem as one of efficiently estimating the functional $\chi(P')$ under model $\model(V')$. This leads us to make the following two definitions.

\smallskip \begin{definition}[Uninformative variables] \label{def:uninfo}
Given a model $\model(\Vset)$ for law $P$ over $\Vset$, we say that a subset of variables $\Uset \subseteq \Vset$ is uninformative for estimating a regular functional $\gamma(P) $ under $\model(\Vset)$ if $\Vset' = \Vset \setminus \Uset$ satisfies the assumptions of \cref{lem:eif-marg}.
\end{definition}

\smallskip \begin{definition}[Irreducible informative variables] \label{def:irr-info}
Let $\model(\Vset)$ be a model for law $P$ over variables $\Vset$. The set $\Vset^{\ast} \subseteq \Vset$ is said to be irreducible informative for estimating a regular functional $\gamma(P)$ under $\model(\Vset)$ if (i) $\Vset \setminus \Vset^{\ast}$ is uninformative and (ii) no proper superset of $\Vset \setminus \Vset^{\ast}$ is uninformative. 
\end{definition}

\smallskip \begin{lemma} \label{lem:info}
Suppose $\model(\Vset)$ is a model for law $P$ over variables $\Vset$, and let $\gamma: \model(\Vset) \rightarrow \mathbb{R}$ be a regular functional. Let $\gamma_{P,\text{eff}}^{1}(\Vset)$ be the corresponding efficient influence function. Suppose $\Vset^{\ast} \subseteq \Vset$ satisfies the following:
\begin{enumerate}[(i)]
\item $\gamma_{P,\text{eff}}^{1}(\Vset)$ depends on $\Vset$ only through $\Vset^{\ast}$ for every $P \in \model(\Vset)$;
\item there exists a functional $\chi: \model(\Vset^{\ast}) \rightarrow \mathbb{R}$ such that $\chi(P^{\ast}) = \gamma(P)$ for every $P \in \model(\Vset)$ and $P^{\ast} \equiv P(\Vset^{\ast})$;
\item for each $V_j \in \Vset^{\ast}$, there exists a nondegenerate law $P_j \in \model(\Vset)$ such that $\gamma_{P_j,\text{eff}}^{1}(\Vset)$ is not a constant function of $V_j$ with probability one.
\end{enumerate} Then $\Vset^{\ast}$ is the unique irreducible informative set. 
\end{lemma}

\medskip From \cref{sec:g-formula}, in the context of causal graphs, we see that typically there are more than one identifying formula for the g-functional. 
Our next two definitions, based on considerations of efficiency and informativeness, help us compare and choose between different identifying formulae. 

Let us first look at efficiency. 
As before, we let $\model_0(\Vset)$ be the nonparametric model over $\Vset$ and let $\model(\Vset)$ be a semiparametric submodel. 
Suppose $\gamma(P)$ and $\chi(P)$ are two identifying formulae, i.e., regular real-valued functionals defined on $\model_0(\Vset)$, such that they agree on $\model(\Vset)$. 
As such, they must have the same efficient influence function with respect to $\model(\Vset)$, i.e., $\gamma^{1}_{P,\text{eff}}(\Vset) = \chi^{1}_{P,\text{eff}}(\Vset)$ for every $P \in \model(\Vset)$.
Suppose that $\Vset$ is discrete and consider the plugin estimators $\gamma(\mathbb{P}_n)$ and $\chi(\mathbb{P}_n)$, where $\mathbb{P}_n$ is the empirical measure. 
Then, $\gamma(\mathbb{P}_n)$ and $\chi(\mathbb{P}_n)$ are regular asymptotically linear with influence functions equal to the nonparametric influence functions $\gamma^{1}_{P,\text{NP}}(\Vset)$ and $\chi^{1}_{P,\text{NP}}(\Vset)$ for every $P \in \model_0(\Vset)$. 
Suppose that $\gamma^{1}_{P,\text{NP}}(\Vset) = \gamma^{1}_{P,\text{eff}}(\Vset)$ for every $P \in \model(\Vset)$ but that, in contrast, $\chi^{1}_{P',\text{NP}}(\Vset) \neq \chi^{1}_{P',\text{eff}}(\Vset)$ for some $P' \in \model(\Vset)$. 
Then, in view of the concepts introduced at the beginning of this subsection, with respect to the semiparametric model $\model(\Vset)$, the estimator $\gamma(\mathbb{P}_n)$ is globally efficient, but $\chi(\mathbb{P}_n)$ is not. 
Then, for estimating functional $\gamma(P) = \chi(P)$ defined on model $\model(\Vset)$, we say $\gamma(P)$ is an efficient identifying formula, but $\chi(P)$ is an inefficient identifying formula. 
This gives us a concrete way of defining whether an identifying formula is efficient.
In below, we provide a definition for the general case where $\Vset$ need not be discrete.

\smallskip \begin{definition}[Efficient identifying formula] \label{def:eff-formula}
Consider a semiparametric model $\model(\Vset) \subseteq \model_0(\Vset)$ and a regular functional $\gamma: \model(\Vset) \rightarrow \mathbb{R}$. Let $\gamma^{1}_{P,\text{eff}}(\Vset)$ be its efficient influence function with respect to $\model(\Vset)$. An identifying formula $\chi: \model_0(\Vset) \rightarrow \mathbb{R}$ for the functional $\gamma$ is called efficient if $\chi^{1}_{P,\text{NP}}(\Vset) = \gamma^{1}_{P,\text{eff}}(\Vset)$ $P$-almost everywhere for every $P \in \model(\Vset)$. 
\end{definition}

\smallskip From \cref{eqs:factorization,eqs:g-formula}, when $\Vset$ is discrete, it is clear that the maximum likelihood estimator of $\Psi_a(P; \g)$ is simply the plugin estimator $\Psi_a(\mathbb{P}_n; \g)$. More generally, we have the following result for an arbitrary vector $\Vset$.
\begin{lemma} \label{lem:effic-of-g-formula}
For graph $\g$ satisfying \cref{assump:A-goes-to-Y}, the g-formula $\Psi_a(\cdot; \g): \model_0(\Vset) \rightarrow \mathbb{R}$ in \cref{eqs:g-formula} is an efficient identifying formula for the g-functional $\Psi_a(\cdot; \g): \model(\g, \Vset) \rightarrow \mathbb{R}$.
\end{lemma}

\smallskip As mentioned in \cref{sec:motivation}, more than one efficient identifying formula may exist for the same functional, such as the g-formulae associated with $\g^{\ast}$ and $\g$ in \cref{fig:example} for our motivating example. In this case, we argue that the g-formula associated with $\g^{\ast}$ should be preferred over that associated with $\g$, as the former requires measuring fewer variables than the latter. This motivates our next definition concerning informativeness.

\smallskip \begin{definition}[Irreducible identifying formula] \label{def:irred-formula}
An identifying formula $\chi: \model_0(\Vset) \rightarrow \mathbb{R}$ for a regular functional $\gamma: \model(\Vset) \rightarrow \mathbb{R}$ is called irreducible if there exists $\Vset^{\ast} \subseteq \Vset$, which is irreducible informative for estimating $\gamma(P)$ under $\model(\Vset)$, such that 
$P(\Vset^{\ast}) = P'(\Vset^{\ast})$ implies $\chi(P) = \chi(P')$ for every $P, P' \in \model_0(\Vset)$, i.e., $\chi(P)$ depends on $P$ only through $P(\Vset^{\ast})$.
\end{definition}

\smallskip In what follows, we will first characterize the irreducible informative set $\Vset^{\ast}$ and then construct the reduced graph $\g^{\ast}$ to represent the marginal model over $\Vset^{\ast}$. In particular, our general result would imply that the g-formula associated with $\g^{\ast}$ in \cref{fig:example} is an identifying formula that is both efficient and irreducible.

\section{Characterizing the uninformative variables} \label{sec:uninfo}
\subsection{Efficient influence function} \label{sec:sub-eif}

We now specialize the concepts and results in the preceding section to show that for estimating the g-functional $\Psi_a(P; \g)$ under the Bayesian network model $\model(\g, \Vset)$, there exists a unique set of irreducible informative variables, which we denote by $\Vset^{\ast} \equiv \Vset^{\ast}(\g)$ throughout. By \cref{lem:info}, this can be established if we can find $\Vset^{\ast} \subseteq \Vset$ such that 
(i) the efficient influence function $\Psi_{a,P,\text{eff}}^{1}(\Vset)$ depends on $V$ only through $\Vset^{\ast}$ for every $P \in \model(\g,\Vset)$;
(ii) $\Psi_{a}(P;\g)$ depends on $P \in \model(\g, \Vset)$ only through the $\Vset^{\ast}$ margin of $P$;
and (iii) for every $V_j \in \Vset^{\ast}$, there exists a nondegenerate law $P \in \model(\g,\Vset)$ such that $\Psi_{a,P,\text{eff} }^{1}(\Vset) $ depends nontrivially on $V_j$. 

Without loss of generality, here we focus on finding the informative variables for the g-functional, as opposed to the average treatment effects, which are contrasts or, more generally, linear combinations of g-functionals that correspond to different treatment levels. Indeed, as shown in \cref{lem:app-extension-effect} of the \supp, the set of irreducible informative variables for these effects is identical to $\Vset^{\ast}(\g)$. 

We will perform these tasks invoking an expression for $\Psi_{a,P,\text{eff}}^{1}(V;\g)$, which is  derived in \citet{rotnitzky2020efficient} and stated in the next lemma. Let $\I_a(A)$ be the indicator of $A$ being equal to $a$. Define $T_{a,P} \equiv \I_{a}(A)Y / P\left( A=a \mid \Omin \right)$ and $b_{a,P}(O) \equiv \E_{P}\left( Y \mid A=a,\Oset \right)$, where $\Oset \equiv \Oset(\g)$ and $\Omin \equiv \Omin(\g)$. 

\smallskip \begin{lemma}[Theorem 7, \citealp{rotnitzky2020efficient}]\label{lem:EIF} 
Let $\g$ be a directed acyclic graph on a vertex set $\Vset$ satisfying \cref{assump:A-goes-to-Y}. Suppose $P\in \model(\g,\Vset)$ and that $\Wset( \g)=\{W_{1},\dots ,W_{J}\}$ and $\Mset(\g)=\{M_{1},\dots ,M_{K-1},M_{K}\equiv Y\}$ are as defined in \cref{sec:sets}. Then the efficient influence function for estimating $\Psi_{a}(P;\g)$ with respect to model $\model(\g,\Vset)$ is given by 
\begin{multline*}
\Psi_{a,P,\text{eff}}^{1}(\Vset; \g) = \sum_{j=1}^{J}\big[ \E\left\{ b_{a,P}(\Oset) \mid W_{j},\Pa(W_{j}, \g)\right\} -\E\left\{ b_{a,P}(\Oset) \mid \Pa(W_{j}, \g) \right\} \big]  \\
 +\sum_{k=1}^{K} \big[ \E\left\{ T_{a,P} \mid M_{k},\Pa(M_{k}, \g)\right\} -
\E\left\{ T_{a,P}\mid \Pa(M_{k}, \g)\right\} \big].
\end{multline*}
\end{lemma}

In the rest of this section, we classify the uninformative variables into two types: redundant and non-redundant.
The redundant variables are those that can be identified from causal Markov equivalent graphs. 
In contrast, identifying the non-redundant, uninformative variables is less straightforward and sometimes counterintuitive. Nevertheless, we will develop a set of graphical criteria to characterize them both. The proofs for this section are given in \cref{sec:app-uninfo}.

\subsection{Redundant variables} \label{sec:redundant}
We start with the following result, which is immediate in view of \cref{eqs:g-formula} and \cref{lem:EIF}.

\begin{lemma} \label{lem:N-I-uninformative}
Given $\g$ satisfying \cref{assump:A-goes-to-Y}, $\Nset(\g) \dcup \Iset(\g)$ is uninformative for estimating $\Psi_a(P;\g)$ under $\model(\g, \Vset)$. 
\end{lemma}

By \cref{def:uninfo}, informativeness is a property defined with respect to a model and a functional. The notion of causal Markov equivalence leads us to the following definition. 

\begin{definition}[Redundant variables]
Given a graph $\g$ satisfying \cref{assump:A-goes-to-Y}, the set of redundant variables in $\g$ for estimating $\Psi_a(P; \g)$ under $\model(\g, \Vset)$ is 
\[ \bigcup_{\g' \ceq \g} \Nset(\g') \cup \Iset(\g'). \] 
\end{definition}

\begin{proposition} \label{prop:redundant} 
Given $\g$ satisfying \cref{assump:A-goes-to-Y}, the redundant variables are uninformative for estimating $\Psi_a(P;\g)$ under $\model(\g,V)$.
\end{proposition}

Revisiting our motivating example on graph $\g$ in \cref{fig:example}(a), the redundant variables are $\{I_1, W_1\}$, which can be summed out from the g-formula; see \cref{eqs:ex-g-prime}. They can also be identified from the causal Markov equivalent graph $\check{\g}$ shown in \cref{fig:example}(e). 

A surprising phenomenon in this example, as indicated earlier in \cref{sec:motivation}, is that $W_4$, despite being nonredundant, is actually uninformative for estimating $\Psi_a(P; \g)$ under the Bayesian network model represented by $\g$. 
To see this, by \cref{lem:EIF}, observe that $\Psi_{a,P,\text{eff}}^{1}(V;\g)$ could depend on $W_{4}$ only through the sum
\begin{equation*}
\begin{split}
& \quad \E\{b_{a,P}(O_1) \mid W_4, \Pa(W_4) \} + \E\{b_{a,P}(O_1) \mid O_1, \Pa(O_1) \} - \E\{b_{a,P}(O_1) \mid \Pa(O_1) \} \\
&= \E\{b_{a,P}(O_1) \mid W_4, W_2, W_3 \} + b_{a,P}(O_1) - \E\{b_{a,P}(O_1) \mid W_4\}.
\end{split}
\end{equation*}
However, model $\model(\g,\Vset)$ implies $O_{1}\indep W_{2},W_{3} \mid W_{4}$, so the sum reduces to $b_{a,P}(O_{1})$, which does not depend on $W_{4}$. In addition, under the model, $\Psi_{a}(P;\g)$ coincides with $\Psi_{a,O_1}^{\mathrm{ADJ}}(P;\g)$, which depends on $P$ only through the marginal law $P(A,Y,O_1)$. In view of \cref{def:uninfo,lem:eif-marg}, $\{I_1, W_1, W_4\}$ are uninformative. 
Those variables that vanish like $W_4$ are called nonredundant, uninformative variables. They are more subtle as they cannot be deduced from simple ancestral relations or causal Markov equivalence.
Next, we develop graphical results towards a complete characterization. 
\subsection{Graphical criteria} \label{sec:criteria}
In this subsection we will often omit $\g$ from the vertex sets introduced in \cref{sec:sets} to reduce clutter. 
First, we show that our search for uninformative variables can be limited to $(\Wset \setminus \Oset) \cup (\Mset \setminus \{Y\})$.

\medskip \begin{lemma} \label{lem:Oset-informative} 
Suppose that $\g$ is a directed acyclic graph on $\Vset$ satisfying \cref{assump:A-goes-to-Y}.
For any $\Uset \subseteq \Vset$ that is uninformative for estimating $\Psi_a(P; \g)$ under $\model(\g, \Vset)$, we have $\Uset \cap \{\{A, Y\} \cup \Oset(\g)\} = \emptyset$. 
\end{lemma} 

\medskip To proceed with our search for uninformative variables, it suffices to identify variables from $\Wset \setminus \Oset$ or $\Mset \setminus \{Y\}$ that vanish from the efficient influence function at every law in the model. This follows from \cref{def:uninfo} and \cref{lem:eif-marg} given that (i) $\Psi _{a}(P;\g)=$ $\Psi _{a,O}^{\mathrm{ADJ}}(P;\g)$ on $\model(\g,V)$ and (ii) $ \Psi _{a,O}^{\mathrm{ADJ}}(P;\g)$ depends on $P$ only through the marginal law of $\Oset \cup \{A,Y\}$.

Let us now identify uninformative variables in $\Wset \setminus \Oset$.  
Every $W_{j}\in \Wset\setminus \Oset$ satisfies $W_{j}\mapsto \Oset$, so $\Ch(W_{j})\cap \Wset \neq \emptyset $. Let us write $\Ch(W_{j})\cap \Wset=\{W_{j_{1}},\dots ,W_{j_{r}}\}$, indexed topologically for $j_{1}\leq \dots \leq j_{r}$ and $r \geq 1$, and define $ W_{j_{0}}\equiv W_{j}$. We observe that $\Psi_{a,P,\text{eff}}^{1}(V;\g)$ in \cref{lem:EIF} depends on $W_{j}$ only through 
\begin{equation} \label{eqs:depends-Wj}
\Gamma(W_j) \equiv \E \left\{ b_{a}(\Oset) \mid W_{j}, \Pa(W_{j}) \right\} + \sum_{t=1}^{r} \left[ \E \left\{ b_{a}(\Oset) \mid W_{j_t}, \Pa(W_{j_t}) \right\} - \E \left\{ b_{a}(\Oset) \mid  \Pa(W_{j_t}) \right\}  \right].
\end{equation}
To analyze $\Gamma(W_j)$, define $\plus_{j}$ as the smallest subset of $\Pa(W_{j})\cup \{W_{j}\}$ such that 
\begin{equation*}
\Pa(W_{j})\cup \{W_{j}\}\setminus \plus_{j}\indep_{\g}\Oset\mid \plus_{j},
\end{equation*}
and $\minus_{j}$ as the smallest subset of $\Pa(W_{j})$ such that 
\begin{equation*}
\Pa(W_{j})\setminus \minus_{j}\indep_{\g}\Oset\mid \minus_{j}.
\end{equation*}
Sets $\plus_{j}$ and $\minus_{j}$ are uniquely defined by the graphoid properties of d-separations. With these definitions and the corresponding conditional independences, \cref{eqs:depends-Wj} becomes 
\begin{equation} \label{eqs:depends-Wj-plus-minus}
\begin{split}
\Gamma(W_j) = \E \left\{ b_{a}(\Oset) \mid \plus_j \right\}  &+ \E \left\{ b_{a}(\Oset) \mid \plus_{j_1} \right\} + \dots +  \E \left\{ b_{a}(\Oset) \mid \plus_{j_{r-1}} \right\} +  \E \left\{ b_{a}(\Oset) \mid \plus_{j_r} \right\} \\
& -\E \left\{ b_{a}(\Oset) \mid \minus_{j_1} \right\} - \dots - \E \left\{ b_{a}(\Oset) \mid \minus_{j_{r-1}} \right\} - \E \left\{ b_{a}(\Oset) \mid \minus_{j_r} \right\}.
\end{split}
\end{equation}

The following lemma contains important properties of the sets $E_j^{+}$ and $E_j^{-}$.
\begin{lemma} \label{lem:plus-minus-W}
The following properties hold:
\begin{enumerate}[(i)]
\item $W_j \in \plus_j$;
\item if $r > 1$, then $W_j \in \plus_{j_t}$ for $t=1,\dots,r-1$;
\item $\minus_j = \Pa(W_j)$.
\end{enumerate}
\end{lemma}

The variable $W_j$ is uninformative if $\Gamma(W_j)$ does not depend on $ W_j$; for this to happen, plausibly, in \cref{eqs:depends-Wj-plus-minus} each $\minus$ term from the second line cancels exactly with one $\plus$ term from the first line, and the remaining term in the first line does not depend on $W_j$. By \cref{lem:plus-minus-W}(ii), the remaining term must be the last term in the first line, which should satisfy $W_j \notin \plus_{j_r}$. Now suppose that $\plus_{j_{r-1}}$ cancels with $\minus_{j_{t}}$ from the second line. Then, by \cref{lem:plus-minus-W}(i) and (ii), this implies $W_{j_{r-1}} \rightarrow W_{j_t}$, which requires $t=r$ to be compatible with the topological ordering. Continuing this argument, we see that $\plus_{j_{r-2}}$ cancels with $\minus_{j_{r-1}}$, and so forth. This is summarized as follows.

\begin{lemma} \label{lem:W-cancel-plus-minus} 
Under \cref{assump:A-goes-to-Y}, variable $W_j$ is uninformative if (i) $W_j \notin \plus_{j_r}$ and (ii) $\plus_{j_{t-1}} = \minus_{j_t}$ for $t=1,\dots,r$.
\end{lemma}

\noindent These conditions are further equivalent to the following graphical criterion.

\begin{lemma}[W-criterion] \label{lem:W-criterion}
Suppose $\g$ satisfies \cref{assump:A-goes-to-Y} and that $W_j \in \Wset \setminus \Oset$ and $\Ch(W_j) \cap \Wset = \{W_{j_1}, \dots, W_{j_r}\}$, indexed topologically for $r \geq 1$; define $W_{j_0} \equiv W_j$. Then the variable $W_j$ is uninformative if the following conditions are satisfied: \begin{enumerate}[(i)]
\item $W_j \indep_{\g} \Oset \mid \{W_{j_r}\} \cup \Pa(W_{j_r}) \setminus \{W_j\} $;
\item for $t=1,\dots,r$ one has
\begin{enumerate}[(a)]
\item $W_{j_{t-1}} \rightarrow W_{j_t}$;
\item $\Pa(W_{j_t}) \subseteq \Pa(W_{j_{t-1}}) \cup \{W_{j_{t-1}}\}$;
\item $\Pa(W_{j_{t-1}}) \setminus \Pa(W_{j_t}) \indep_{\g} \Oset \mid \Pa(W_{j_t})$. 
\end{enumerate}
\end{enumerate}
\end{lemma}

As an example, let us check that $W_{4}$ in \cref{fig:example}(a) satisfies
the W-criterion. Observe that $r=1$ and $W_{j_{r}}=O_1$. Condition (i) is trivial: recall that $W_{4}\indep_{\g}O_1\mid O_1$ is parsed as $W_{4}\indep_{\g}\emptyset \mid O_1$, which is true by our convention. For condition (ii), we check that
(a) $W_{4}\rightarrow O_1$, (b) $W_{4}\subset \{W_{2},W_{3},W_{4}\}$ and
(c) $W_{2},W_{3} \indep_{\g} O_1 \mid W_{4}$. In contrast, we see that $W_{2}$
and $W_{3}$ fail the W-criterion, in particular condition (ii)(b).

\medskip By a similar line of reasoning, we derive the corresponding criterion for the set of mediators. 

\begin{lemma}[M-criterion] \label{lem:M-criterion}
Suppose $\g$ satisfies \cref{assump:A-goes-to-Y} and that $M_i \in \Mset \setminus \{Y\}$ and $\Ch(M_i) \cap \Mset = \{M_{i_1}, \dots, M_{i_k}\}$, indexed topologically for $k \geq 1$; define $M_{i_0} \equiv M_i$. Then the variable $M_i$ is uninformative if the following conditions are satisfied: \begin{enumerate}[(i)]
\item $M_i \indep_{\g} \{A, Y\} \cup \Omin \mid \{M_{i_k}\} \cup \Pa(M_{i_k}) \setminus \{M_i\} $;
\item for $t=1,\dots,k$ one has
\begin{enumerate}[(a)]
\item $M_{i_{t-1}} \rightarrow M_{i_t}$;
\item $\Pa(M_{i_t}) \subseteq \Pa(M_{i_{t-1}}) \cup \{M_{i_{t-1}}\}$;
\item $\Pa(M_{i_{t-1}}) \setminus \Pa(M_{i_t}) \indep_{\g} \{A,Y\} \cup \Omin \mid \Pa(M_{i_t})$. 
\end{enumerate}
\end{enumerate}
\end{lemma}

We show the soundness of W- and M-criterion in \cref{sec:app-sound}. Our first main result shows that our graphical characterization is also complete.

\begin{theorem}[Graphical criteria for irreducible, informative variables] \label{thm:criteria}
Let $\g$ be a directed acyclic graph on a vertex set $\Vset$ that satisfies \cref{assump:A-goes-to-Y}. Suppose that $A \in \Vset$ is a discrete treatment and $Y \in \Vset$ is the outcome of interest. Then there exists a unique set of irreducible informative variables for estimating $\Psi _{a}(P;\g)$ under $ \model(\g,\Vset)$, 
\begin{multline*}
\Vset^{\ast }(\g) \equiv \{A,Y\}\cup \Oset\cup \{W_{j}\in \Wset\setminus \Oset:
\text{$W_{j}$ fails the W-criterion}\} \\
\cup \{M_{i}\in \Mset\setminus \{Y\}:\text{$M_{i}$ fails the M-criterion}\},
\end{multline*}
where $\Oset \equiv \Oset(\g)$, $\Wset \equiv \Wset(\g)$ and $\Mset \equiv \Mset(\g)$ are defined in \cref{sec:sets}.
\end{theorem}

To prove \cref{thm:criteria}, for each variable in $\Wset\setminus \Oset$ and $\Mset\setminus \{Y\}$ that fails the corresponding criterion, we show in \cref{sec:app-uninfo,sec:app-W-complete,sec:app-M-complete} that there exists a nondegenerate law $P\in \model(\g,\Vset)$ such that $ \Psi^{1}_{a,P,\text{eff}}(V;\g)$ depends nontrivially on the variable.

\section{Graph reduction and the efficient irreducible g-formula} \label{sec:graph}
\subsection{Marginal model} \label{sec:marginal}

The results of the preceding section imply that we do not
lose information by discarding the variables excluded from the set $\Vset ^{\ast }\equiv \Vset^{\ast }(\g)$ in \cref{thm:criteria}. 
In what follows, we will write $P^{\ast}$ for the marginal law $P(\Vset^{\ast})$. 
Also, recall from \cref{sec:eif} that $\model(\g; \Vset^{\ast})$ refers to the marginal model over $P^{\ast}$ induced by $P \in \model(\g, \Vset)$.
In this section, we will characterize the marginal model $\model(\g; \Vset^{\ast})$ and then re-express the g-functional as a functional of $P^{\ast}$ in $\model(\g; \Vset^{\ast})$.

Characterizing the marginal model is nontrivial, even when the state space of the variables that are marginalized over is unrestricted. In general, the margin of a Bayesian network can be a complicated statistical model subject to both equality and inequality constraints. The equalities consist of conditional independences and their generalizations known as the nested Markov properties; see \citet{shpitser2014introduction,evans2018margins}. The inequalities are related to Bell's inequalities \citep{gill2014statistics} and are often hard to characterize  \citep{pearl2013testability,bonet2013instrumentality}. Fortunately, we can avoid these complications because, as will be shown later, under our definition of Bayesian networks in \cref{sec:sets} where the state space of each variable is sufficiently large, the marginal model $\model(\g; \Vset^{\ast})$ is exactly a Bayesian network model represented by a certain directed acyclic graph $\g^{\ast}$ over vertices $\Vset^{\ast}$. Further, the g-formula associated with $\g^{\ast}$ immediately identifies the g-functional of $P$ as a functional of $P^{\ast}$. Finally, this formula is irreducible and efficient.  

The construction of $\g^{\ast}$ can be viewed as iteratively projecting out all the uninformative variables, such that each time a variable or a set of variables are projected out, the resulting graph represents the marginal model over the remaining variables. We will start by projecting out variables in $\Nset(\g)$ and $\Iset(\g)$ altogether. 

\subsection{Projecting out $\Nset(\g)$ and $\Iset(\g)$}
\begin{lemma}[Marginalizing over $\Nset(\g)$ and $\Iset(\g)$] \label{lem:marg-over-NI}
Let $\g$ be a directed acyclic graph on a vertex set $V$ satisfying \cref{assump:A-goes-to-Y}. Let $\Nset(\g)$ and $\Iset(\g)$ be defined as in \cref{sec:sets} and let $\Vset^0 \equiv \Vset \setminus \{\Nset(\g) \cup \Iset(\g)\}$. Let graph $\g^0$ be constructed from $\g$ as follows. First, for every $V_i, V_j \in \Vset^0$ such that $V_i \mapsto V_j$ through a causal path on which every non-endpoint vertex is in $\Iset(\g)$, add an edge $V_i \rightarrow V_j$ if the edge is not present. Next, remove vertices in $\Nset(\g) \cup \Iset(\g)$ and their associated edges. 
Call the resulting graph $\g^0$.
Then $\g^0$ is a directed acyclic graph over $\Vset^0$ and $\model(\g, \Vset^0) = \model(\g^0, \Vset^0)$.
\end{lemma}

\smallskip See \cref{sec:app-marg-over-NI} for a proof. The graph $\g^0$ is a reformulation of the graph produced by \citet[Algorithm 1]{rotnitzky2020efficient}. As an example, in \cref{fig:example}, projecting out $\Nset(\grev) \cup \Iset(\grev) = \{W_1, I_1\}$ from graph $\grev$ leads to graph $\g'$.

\subsection{Projecting out the remaining uninformative variables}
By exploiting the graphical structures in the W- and M-criterion and using the results on graphs for representing margins of Bayesian networks due to \citet{evans2018margins}, we show in \cref{sec:app-proj} that the remaining uninformative variables in $\Wset(\g) \cup \Mset(\g)$ can be projected out as well, one at a time. The projection is defined as follows. 
\begin{definition} \label{def:project-DAG}
Let $\g$ be a directed acyclic graph on a vertex set $\Vset$. For $V_i \in \Vset$, suppose that $\Ch(V_i, \g)$ is topologically ordered as $\pi = (V_{i_1}, \dots, V_{i_l})$ for $l \geq 1$ and let $V_{i_0} \equiv V_i$. Let $\g_{-V_i, \pi}$ be a graph on vertices $\Vset \setminus \{V_i\}$, formed by adding an edge $V_k \rightarrow V_{i_j}$ to $\g$ if the edge is not already present, for every $V_k \in \Pa(V_i, \g) \cup \{V_{i_0}, \dots, V_{i_{j-1}}\}$ and every $j=1,\dots, l$, and then removing $V_i$ and its associated edges.
\end{definition}
\smallskip In other words, all edges from $\Pa(V_i, \g)$ to $\Ch(V_i, \g)$ and all edges among $\Ch(V_i, \g)$ that are compatible with the topological ordering $\pi$ are saturated before $V_i$ is removed. In contrast to the latent projection of \citet{verma1991equivalence}, the projection defined above results in a directed acyclic graph; compare \cref{fig:example}(c) and (d). 

\medskip \begin{lemma} \label{lem:project-DAG}
Let $\g$ be a directed acyclic graph on vertices $\Vset$. Let $V_i \in \Vset$, whose children are topologically sorted as $\pi = (V_{i_1}, \dots, V_{i_l})$ for $l \geq 1$. 
Suppose that
\begin{equation} \label{eqs:projection-pa}
\Pa(V_{i_j}, \g) \subseteq \{V_{i_{j-1}}\} \cup \Pa(V_{i_{j-1}}, \g), \quad j=1,\dots, l-1,
\end{equation}
where $V_{i_0} \equiv V_i$. 
Then, $\g_{-V_i,\pi}$ is a directed acyclic graph on $\Vset \setminus \{V_i\}$ and $\model(\g, \Vset \setminus \{V_i\}) = \model(\g_{-V_i,\pi}, \Vset \setminus \{V_i\})$.
\end{lemma}

\smallskip \cref{lem:project-DAG} can be specialized to any uninformative vertex in $W$ or $M$ as follows. 

\medskip \begin{lemma} \label{lem:iterate-proj}
Let $\g$ be a directed acyclic graph on a vertex set $\Vset$. Suppose that $\g$ satisfies \cref{assump:A-goes-to-Y} and $\Nset(\g) = \Iset(\g) = \emptyset$. Consider a vertex $V_i \in \Vset \setminus \Vset^{\ast}(\g)$. If $V_i \in \Wset(\g)$, suppose $V_i \equiv W_i$ and let 
\begin{equation} \label{eqs:pi-W}
\pi = \begin{cases} (W_{i_1}, \dots, W_{i_l}), \quad &A \notin \Ch(W_i, \g) \\
(W_{i_1}, \dots, W_{i_l}, A), \quad & A \in \Ch(W_i, \g) 
\end{cases},
\end{equation}
where $\Ch(W_i, \g) \cap \Wset(\g) = \{W_{i_1}, \dots, W_{i_l}\}$ is uniquely topologically sorted. 
Otherwise, $V_i \equiv M_i$ for some $M_i \in \Mset(\g)$ and let 
\begin{equation} \label{eqs:pi-M}
\pi = (M_{i_1}, \dots, M_{i_l}) = \Ch(M_i, \g), 
\end{equation}
which is uniquely topologically sorted. Then
\[ \model(\g, \Vset \setminus \{V_i\}) = \model(\g_{-V_i, \pi}, \Vset \setminus \{V_i\}), \quad \Vset^{\ast}(\g_{-V_i, \pi}) = \Vset^{\ast}(\g). \]
\end{lemma}
In other words, by projecting out an uninformative variable $V_i \in \Wset \cup \Mset$ from a graph $\g$ whose $\Nset(\g)$ and $\Iset(\g)$ are empty, the resulting graph $\g_{-V_i, \pi}$ represents the marginal model over the remaining variables, and preserves the same set of irreducible informative variables given by \cref{thm:criteria}.

\subsection{Graph reduction algorithm and properties of the reduced graph}
The graph reduction procedure is presented in \cref{alg:graph}. In the algorithm each vertex is visited once. As checking any d-separation takes a polynomial time of $|V|$, the algorithm also finishes in a polynomial time of $|V|$. The algorithm is implemented in the R package \texttt{reduceDAG}, available from \texttt{https://github.com/richardkwo/reduceDAG}.

\begin{algorithm}[!htb]
\DontPrintSemicolon \SetNoFillComment \caption{Graph reduction algorithm} \label{alg:graph}
\KwIn{Graph $\g$ on vertex set $\Vset$ satisfying \cref{assump:A-goes-to-Y}}
\KwOut{Reduced graph $\g^{\ast}$ that represents $\model(\g, \Vset^{\ast})$}
$\Vset^{\ast} \gets \{A\} \cup \Wset(\g) \cup \Mset(\g)$\;
$\g^{\ast} \gets \g^{0}$ defined in  \cref{lem:marg-over-NI}\;
\For{$V_i \in \Vset^{\ast} \setminus \{\{A, Y\} \cup \Oset(\g)\}$} {
	\uIf{$V_i \in \Wset $ and $V_i$ satisfies the W-criterion in \cref{lem:W-criterion}} {
		$\Vset^{\ast} \gets \Vset^{\ast} \setminus \{V_i\}$\;
		$\g^{\ast} \gets \g^{\ast}_{-V_i,\pi}$ with $\pi$ defined in \cref{eqs:pi-W}\; 
	}
	\uElseIf{$V_i \in \Mset$ and $V_i$ satisfies the M-criterion in \cref{lem:M-criterion}} {
		$\Vset^{\ast} \gets \Vset^{\ast} \setminus \{V_i\}$\;
		$\g^{\ast} \gets \g^{\ast}_{-V_i,\pi}$ with $\pi$ defined in \cref{eqs:pi-M}\; 
	}
}
\Return{$\g^{\ast}$}\;
\end{algorithm}

The properties of the reduced graph are summarized by our next main result; see \cref{sec:app-reduction} for its proof.

\begin{theorem} \label{thm:algo}
Let $\g$ be a directed acyclic graph on a vertex set $\Vset$ that satisfies \cref{assump:A-goes-to-Y}. Suppose $A \in \Vset$ is a discrete treatment and $Y \in \Vset$ is the outcome of interest. Let $\g^{\ast}$ be the output of \cref{alg:graph} resulting from input $\g$. 
Let $\Vset^{\ast} \equiv \Vset^{\ast}(\g)$ be the set of irreducible informative variables given in \cref{thm:criteria}. Also, let $P^{\ast} \equiv P(\Vset^{\ast})$ and define $\Psi_a(P; \g^{\ast}) \equiv \Psi_a(P^{\ast}; \g^{\ast})$. 
The graph $\g^{\ast}$ satisfies the following properties:
\begin{enumerate}[(i)]
\item $\g^{\ast}$ is a directed acyclic graph on vertices $\Vset^{\ast}$;
\item $\g^{\ast}$ does not depend on the order in which vertices are visited in the for-loop of \cref{alg:graph};
\item $\model(\g, \Vset^{\ast}) = \model(\g^{\ast}, \Vset^{\ast})$; 
\item $\Psi_a(P; \g) = \Psi_a(P; \g^{\ast})$ for every $P \in \model(\g; \Vset)$;
\item for every $P \in \model(\g, \Vset)$, the efficient influence functions $\Psi_{a,P,\text{eff}}^{1}(\Vset; \g)$ and $\Psi_{a,P^{\ast},\text{eff}}^{1}(\Vset^{\ast}; \g^{\ast})$, as functions of $\Vset$ and $\Vset^{\ast}$, respectively, are identical $P$-almost everywhere;
\item the g-formula $\Psi_a(\cdot; \g^{\ast}): \model_0(\Vset) \rightarrow \mathbb{R}$ is an irreducible, efficient identifying formula for the g-functional defined on $\model(\g, \Vset)$. 
\end{enumerate}
\end{theorem}

\medskip \begin{corollary} \label{cor:asymp-equiv-g}
Suppose the conditions in \cref{thm:algo} are satisfied and that variables in $\Vset$ are discrete. Then under every $P \in \model(\g, \Vset)$, 
\[ n^{1/2} \left\{\Psi_a(\mathbb{P}^{\ast}_n; \g^{\ast}) - \Psi_a(\mathbb{P}_n; \g)\right\} = o_{p}(1) \] 
as $n \rightarrow \infty$, where $\mathbb{P}_n$ and $\mathbb{P}^{\ast}_n$ are respectively the empirical measures based on $n$ independent copies of $\Vset$ and $\Vset^{\ast}$.
\end{corollary}

In light of \cref{cor:asymp-equiv-g}, in \cref{sec:app-simu} we compare the two estimators for the example in \cref{fig:example} with simulations based on discrete data; their performances seem extremely close even for finite samples.

\section{Examples} \label{sec:ex}
To ease the notation, we omit the graph from vertex sets when it is clear from the context.

\addtocounter{example}{-1}
\begin{example}[continued]
By \cref{thm:criteria}, $\Vset^{\ast} = \Vset$ for \cref{fig:taxonomy}(b). Hence, the graph cannot be further reduced; g-formula \cref{eqs:front-back-door-g} is efficient, while \cref{eqs:front-back-door-back,eqs:front-back-door-front} are not.
\end{example}

\medskip \begin{example} \label{ex:strange}
Consider graph $\g_1$ in \cref{fig:strange}. Note that $\Oset_{\min} = \emptyset$. Variable $M$ is uninformative by checking against the M-criterion: (i) $M \indep_{\g} A, Y \mid A, Y, O$; (ii) (a) $M \rightarrow Y$, (b) $\Pa(Y) \subset \{A, O, M\}$ and (c) $O \indep_{\g} A, Y \mid A, M$. The graph $\g_1$ is reduced to $\g_1^{\ast}$, which prescribes an irreducible, efficient g-formula
\begin{equation} \label{eqs:g-strange}
\Psi_a(P; \g_1^{\ast}) = \sum_{o} \E(Y \mid A=a, o)\, p(o). 
\end{equation}
This result also follows from \citet[Theorem 19]{rotnitzky2020efficient}.

On the other hand, suppose we add edge $O \rightarrow A$ as in $\g_2$. Now, we have $\Oset_{\min}(\g_2) = \{O\}$ and $M$ fails the M-criterion. Hence, if $A$ is randomized conditionally on $O$, then \cref{eqs:g-strange} is still an identifying formula for the g-functional, but is no longer efficient. Since $\g_2 = \g_2^{\ast}$, g-formula $\Psi_a(P; \g_2)$ is irreducible and efficient.

Furthermore, suppose the edge between $A$ and $O$ is added in the reverse direction, as shown in $\g_3$, where $O$ is relabeled as $M'$. The variables $\{M, M'\}$ are uninformative by checking against the M-criterion or, alternatively, by recognizing that they are non-ancestors of $Y$ in a causal Markov equivalent graph $\g_3'$. In this case, an irreducible, efficient identifying formula is simply
\[ \Psi_a(P; \g_3^{\ast}) = \E(Y \mid A=a).  \]

\begin{figure}[!htb]
\centering
\begin{tikzpicture}
\tikzset{rv/.style={circle,inner sep=1pt,fill=gray!20,draw,font=\sffamily}, 
redv/.style={circle,inner sep=1pt,fill=white,draw,dashed,font=\sffamily}, 
sv/.style={circle,inner sep=1pt,fill=gray!20,draw,font=\sffamily,minimum size=1mm}, 
node distance=12mm, >=stealth, every node/.style={scale=0.8}}
\begin{scope}
\node[name=A, rv]{$A$};
\node[name=M, redv, right of=A]{$M$};
\node[name=Y, rv, right of=M]{$Y$};
\node[name=O, rv, above of=M]{$O$};
\node[below of=M, yshift=4mm]{(a) $\g_1$};
\draw[->, very thick, color=blue] (A) -- (M);
\draw[->, very thick, color=blue] (M) -- (Y);
\draw[->, very thick, color=blue, bend right] (A) to (Y);
\draw[->, very thick, color=blue] (O) -- (M);
\end{scope} \begin{scope}[xshift=3cm]
\node[name=A, rv]{$A$};
\node[name=Y, rv, right of=A]{$Y$};
\node[name=O, rv, above of=Y]{$O$};
\node[above of=A, xshift=-7mm, yshift=-4mm]{\Large $\Rightarrow$};
\node[below of=Y, xshift=-5mm, yshift=4mm]{(b) $\g_1^{\ast}$};
\draw[->, very thick, color=blue] (A) -- (Y);
\draw[->, very thick, color=blue] (O) -- (Y);
\end{scope} \begin{scope}[xshift=5cm]
\node[name=A, rv]{$A$};
\node[name=M, rv, right of=A]{$M$};
\node[name=Y, rv, right of=M]{$Y$};
\node[name=O, rv, above of=M]{$O$};
\node[below of=M, yshift=4mm]{(c) $\g_2 = \g_2^{\ast}$};
\draw[->, very thick, color=blue] (A) -- (M);
\draw[->, very thick, color=blue] (M) -- (Y);
\draw[->, very thick, color=blue, bend right] (A) to (Y);
\draw[->, very thick, color=blue] (O) -- (M);
\draw[->, very thick, color=blue] (O) -- (A);
\end{scope} \begin{scope}[xshift=8cm]
\node[name=A, rv]{$A$};
\node[name=M, redv, right of=A]{$M$};
\node[name=Y, rv, right of=M]{$Y$};
\node[name=O, redv, above of=M]{$M'$};
\node[below of=M, yshift=4mm]{(d) $\g_3$};
\draw[->, very thick, color=blue] (A) -- (M);
\draw[->, very thick, color=blue] (M) -- (Y);
\draw[->, very thick, color=blue, bend right] (A) to (Y);
\draw[->, very thick, color=blue] (O) -- (M);
\draw[->, very thick, color=blue] (A) -- (O);
\end{scope} \begin{scope}[xshift=10.7cm]
\node[name=A, rv]{$A$};
\node[name=M, redv, right of=A]{$M$};
\node[name=Y, rv, right of=M]{$Y$};
\node[name=O, redv, above of=M]{$M'$};
\node[below of=M, yshift=4mm]{(e) $\g_3'$};
\node[above of=A, xshift=-5mm, yshift=-4mm]{\Large $\ceq$};
\draw[->, very thick, color=blue] (A) -- (M);
\draw[->, very thick, color=blue] (Y) -- (M);
\draw[->, very thick, color=blue, bend right] (A) to (Y);
\draw[->, very thick, color=blue] (M) -- (O);
\draw[->, very thick, color=blue] (A) -- (O);
\end{scope}
\begin{scope}[xshift=13.5cm]
\node[name=A, rv]{$A$};
\node[name=Y, rv, right of=A]{$Y$};
\node[below of=A, yshift=4mm, xshift=5mm]{(f) $\g_3^{\ast}$};
\node[above of=A, xshift=-5mm, yshift=-4mm]{\Large $\Rightarrow$};
\draw[->, very thick, color=blue] (A) -- (Y);
\end{scope}
\end{tikzpicture}
\caption{Reduction of graphs $\g_1, \g_2, \g_3$ in \cref{ex:strange}.}
\label{fig:strange}
\end{figure}
\end{example}

\medskip \begin{example}[optimal adjustment] \label{ex:two-Os}
Consider the graphs in \cref{fig:two-Os}. Recall that the optimal adjustment estimator is the sample version of \cref{eqs:ex-backdoor} when $L = \Oset$. 
When the optimal adjustment estimator is efficient, such as under $\g_3$, it holds that $\Vset^{\ast}$ only consists of the optimal adjustment set, $A$ and $Y$. However, the reverse need not be true. Consider the graph $\g_1$ where $\Vset^{\ast}(\g_1) = \Oset(\g_1) \cup \{A,Y\}$, but the optimal adjustment estimator is inefficient because it does not exploit the independence between $O_1$ and $O_2$; compare with the g-formula associated with $\g_1$.
\begin{figure}[!htb]
\centering
\begin{tikzpicture}
\tikzset{rv/.style={circle,inner sep=1pt,fill=gray!20,draw,font=\sffamily}, 
redv/.style={circle,inner sep=1pt,fill=white,draw,dashed,font=\sffamily}, 
sv/.style={circle,inner sep=1pt,fill=gray!20,draw,font=\sffamily,minimum size=1mm}, 
node distance=12mm, >=stealth, every node/.style={scale=0.8}}
\begin{scope}
\node[name=A, rv]{$A$};
\node[name=M, right of=A]{};
\node[name=Y, rv, right of=M]{$Y$};
\node[name=O1, rv, above of=A]{$O_1$};
\node[name=O2, rv, above of=Y]{$O_2$};
\node[below of=M, yshift=4mm]{(a) $\g_1 = \g_1^{\ast}$};
\draw[->, very thick, color=blue] (A) -- (Y);
\draw[->, very thick, color=blue] (O1) -- (A);
\draw[->, very thick, color=blue] (O1) -- (Y);
\draw[->, very thick, color=blue] (O2) -- (A);
\draw[->, very thick, color=blue] (O2) -- (Y);
\end{scope} \begin{scope}[xshift=4cm]
\node[name=A, rv]{$A$};
\node[name=M, right of=A]{};
\node[name=Y, rv, right of=M]{$Y$};
\node[name=O1, rv, above of=A]{$O_1$};
\node[name=O2, rv, above of=Y]{$O_2$};
\node[name=W, rv, above of=M, yshift=7mm]{$W$};
\node[below of=M, yshift=4mm]{(b) $\g_2 = \g_2^{\ast}$};
\draw[->, very thick, color=blue] (A) -- (Y);
\draw[->, very thick, color=blue] (O1) -- (A);
\draw[->, very thick, color=blue] (O1) -- (Y);
\draw[->, very thick, color=blue] (O2) -- (A);
\draw[->, very thick, color=blue] (O2) -- (Y);
\draw[->, very thick, color=blue] (W) -- (O1);
\draw[->, very thick, color=blue] (W) -- (O2);
\end{scope} \begin{scope}[xshift=8cm]
\node[name=A, rv]{$A$};
\node[name=M, right of=A]{};
\node[name=Y, rv, right of=M]{$Y$};
\node[name=O1, rv, above of=A]{$O_1$};
\node[name=O2, rv, above of=Y]{$O_2$};
\node[name=W, redv, above of=M, yshift=7mm]{$W$};
\node[below of=M, yshift=4mm]{(c) $\g_3$};
\draw[->, very thick, color=blue] (A) -- (Y);
\draw[->, very thick, color=blue] (O1) -- (A);
\draw[->, very thick, color=blue] (O1) -- (Y);
\draw[->, very thick, color=blue] (O2) -- (A);
\draw[->, very thick, color=blue] (O2) -- (Y);
\draw[->, very thick, color=blue] (W) -- (O1);
\draw[->, very thick, color=blue] (W) -- (O2);
\draw[->, very thick, color=blue] (O1) -- (O2);
\end{scope} \begin{scope}[xshift=11.5cm]
\node[name=A, rv]{$A$};
\node[name=M, right of=A]{};
\node[name=Y, rv, right of=M]{$Y$};
\node[name=O1, rv, above of=A]{$O_1$};
\node[name=O2, rv, above of=Y]{$O_2$};
\node[below of=M, yshift=4mm]{(d) $\g_3^{\ast}$};
\draw[->, very thick, color=blue] (A) -- (Y);
\draw[->, very thick, color=blue] (O1) -- (A);
\draw[->, very thick, color=blue] (O1) -- (Y);
\draw[->, very thick, color=blue] (O2) -- (A);
\draw[->, very thick, color=blue] (O2) -- (Y);
\draw[->, very thick, color=blue] (O1) -- (O2);
\node[left= 2mm of O1]{\Large $\Rightarrow$};
\end{scope} \end{tikzpicture}
\caption{Reduction of graphs $\g_1, \g_2, \g_3$ in \cref{ex:two-Os}. The optimal adjustment estimator is inefficient for $\g_1$ even though $\Vset^{\ast}(\g_1) = \Oset(\g_1) \cup \{A,Y\}$.}
\label{fig:two-Os}
\end{figure}
\end{example}

\medskip \begin{example} \label{ex:three-Ms}
Consider graph $\g$ in \cref{fig:three-Ms}. By \cref{thm:criteria}, $A, Y, O_1$ and $O_2$ are included in $\Vset^{\ast}$. Note that $\Oset_{\min} = \{O_1\}$. By projecting out $I_1$, an indirect ancestor of $Y$, $\g$ is reduced to $\g^0$. Now let us check the M-criterion for $M_1$, $M_2$ and $M_3$.
First, $M_1$ fails the M-criterion because $M_1 \not\indep_{\g} A, Y, O_1 \mid Y, M_3$. Second, $M_2$ satisfies the criterion as it can be checked that (i) $M_2 \indep_{\g} A, Y, O_1 \mid M_1, M_3$; (ii)(a) $M_2 \rightarrow M_3$, (b) $\Pa(M_3) \subseteq \Pa(M_2) \cup \{M_2\}$ and (c) $\Pa(M_2) \setminus \Pa(M_3) = \emptyset$ so the corresponding d-separation trivially holds. Third, $M_3$ also satisfies the criterion: (i) $M_3 \indep_{\g} A, Y, O_1 \mid Y, M_1$; (ii) (a) $M_3 \rightarrow Y$, (b) $\Pa(Y) \subset \Pa(M_3) \cup \{M_3\}$ and (c) $M_2 \indep_{\g} A, Y, O_1 \mid M_1, M_3$. By further projecting out $M_2$ and $M_3$, we get $\g^{\ast}$. Consequently, an irreducible, efficient g-formula is
\[ \Psi_a(P; \g^{\ast}) = \sum_{m_1} \E(Y \mid m_1) \sum_{o_1, o_2} P(m_1 \mid A=a, o_1, o_2)\, p(o_1) \, p(o_2). \]

\begin{figure}[htb]
\centering
\begin{tikzpicture}
\tikzset{rv/.style={circle,inner sep=1pt,fill=gray!20,draw,font=\sffamily}, 
redv/.style={circle,inner sep=1pt,fill=white,draw,dashed,font=\sffamily}, 
sv/.style={circle,inner sep=1pt,fill=gray!20,draw,font=\sffamily,minimum size=1mm}, 
node distance=12mm, >=stealth, every node/.style={scale=0.75}}
\begin{scope}
\node[name=A, rv]{$A$};
\node[name=M1, rv, right of=A]{$M_1$};
\node[name=M2, redv, right of=M1]{$M_2$};
\node[name=M3, redv, right of=M2]{$M_3$};
\node[name=Y, rv, right of=M3]{$Y$};
\node[name=I, redv, above of=A, xshift=-5mm]{$I_1$};
\node[name=O1, rv, above of=A, xshift=5mm]{$O_1$};
\node[name=O2, rv, above of=A, xshift=15mm]{$O_2$};
\node[below of=M2, yshift=2mm]{(a) $\g$};
\draw[->,very thick, color=blue] (A) to (M1);
\draw[->,very thick, color=blue] (M1) to (M2);
\draw[->,very thick, color=blue] (M2) to (M3);
\draw[->,very thick, color=blue] (M3) to (Y);
\draw[->,very thick, color=blue, bend right] (M1) to (Y);
\draw[->,very thick, color=blue, bend left] (M1) to (M3);
\draw[->,very thick, color=blue] (I) to (A);
\draw[->,very thick, color=blue] (O1) to (A);
\draw[->,very thick, color=blue] (O1) to (M1);
\draw[->,very thick, color=blue] (O2) to (M1);
\end{scope} \begin{scope}[xshift=4.2cm]
\node[name=A, rv]{$A$};
\node[name=M1, rv, right of=A]{$M_1$};
\node[name=M2, redv, right of=M1]{$M_2$};
\node[name=M3, redv, right of=M2]{$M_3$};
\node[name=Y, rv, right of=M3]{$Y$};
\node[name=O1, rv, above of=A, xshift=5mm]{$O_1$};
\node[name=O2, rv, above of=A, xshift=15mm]{$O_2$};
\node[left=4mm of O1]{\Large $\Rightarrow$};
\node[below of=M2, yshift=2mm]{(b) $\g^0$};
\draw[->,very thick, color=blue] (A) to (M1);
\draw[->,very thick, color=blue] (M1) to (M2);
\draw[->,very thick, color=blue] (M2) to (M3);
\draw[->,very thick, color=blue] (M3) to (Y);
\draw[->,very thick, color=blue, bend right] (M1) to (Y);
\draw[->,very thick, color=blue, bend left] (M1) to (M3);
\draw[->,very thick, color=blue] (O1) to (A);
\draw[->,very thick, color=blue] (O1) to (M1);
\draw[->,very thick, color=blue] (O2) to (M1);
\end{scope} \begin{scope}[xshift=8.4cm]
\node[name=A, rv]{$A$};
\node[name=M1, rv, right of=A]{$M_1$};
\node[name=M3, redv, right of=M1]{$M_3$};
\node[name=Y, rv, right of=M3]{$Y$};
\node[name=O1, rv, above of=A, xshift=5mm]{$O_1$};
\node[name=O2, rv, above of=A, xshift=15mm]{$O_2$};
\node[left=4mm of O1]{\Large $\Rightarrow$};
\node[below of=M1, yshift=2mm, xshift=2mm]{(c) $\g^1$};
\draw[->,very thick, color=blue] (A) to (M1);
\draw[->,very thick, color=blue] (M1) to (M3);
\draw[->,very thick, color=blue] (M3) to (Y);
\draw[->,very thick, color=blue, bend right] (M1) to (Y);
\draw[->,very thick, color=blue] (O1) to (A);
\draw[->,very thick, color=blue] (O1) to (M1);
\draw[->,very thick, color=blue] (O2) to (M1);
\end{scope} \begin{scope}[xshift=12cm]
\node[name=A, rv]{$A$};
\node[name=M1, rv, right of=A]{$M_1$};
\node[name=Y, rv, right of=M1]{$Y$};
\node[name=O1, rv, above of=A, xshift=5mm]{$O_1$};
\node[name=O2, rv, above of=A, xshift=15mm]{$O_2$};
\node[left=4mm of O1]{\Large $\Rightarrow$};
\node[below of=M1, yshift=2mm]{(d) $\g^{\ast}$};
\draw[->,very thick, color=blue] (A) to (M1);
\draw[->,very thick, color=blue] (M1) to (Y);
\draw[->,very thick, color=blue] (O1) to (A);
\draw[->,very thick, color=blue] (O1) to (M1);
\draw[->,very thick, color=blue] (O2) to (M1);
\end{scope}
\end{tikzpicture}
\caption{Graph reduction for \cref{ex:three-Ms}, where $\Vset \setminus \Vset^{\ast} = \{I_1, M_2, M_3\}$.}
\label{fig:three-Ms}
\end{figure}
\end{example}

\begin{example} \label{ex:long}
Let $\g$ be the graph drawn as \cref{fig:long}(a), for which $\Oset = \{O_1, O_2, O_3\}$. Again, variable $I_1$ is an indirect ancestor of $Y$ and hence uninformative. It can be checked that, variables $W_3, W_4, W_5$ fail the W-criterion, in particular its condition (i). It can also be checked that variables $W_1, W_2, W_6$ satisfy the W-criterion. For example, for $W_2$ observe that: (i) $W_2 \indep_{\g} O_1, O_2, O_3 \mid O_1, W_5$; (ii)(a) $W_2 \rightarrow O_1$, (b) $\Pa(O_1) \subset \Pa(W_2) \cup \{W_2\}$ and (c) $W_3, W_4 \indep_{\g} O_1, O_2, O_3 \mid W_2, W_5$. By iteratively projecting out $I, W_1, W_2$ and $W_6$, graph $\g$ is reduced to $\g^{\ast}$, from which we can derive an irreducible, efficient g-formula
\begin{multline*}
\Psi_a(P; \g^{\ast}) = \sum_{o_1, o_2, o_3} \E(Y \mid A=a, o_1, o_2, o_3) p(o_3) \\
  \times \sum_{w_3, w_4} p(w_3) p(w_4) \sum_{w_5} p(o_1 \mid w_3, w_4, w_5) p(o_2 \mid w_5) p(w_5).
\end{multline*}

\begin{figure}[htb]
\centering
\begin{tikzpicture}
\tikzset{rv/.style={circle,inner sep=1pt,fill=gray!20,draw,font=\sffamily}, 
redv/.style={circle,inner sep=1pt,fill=white,draw,dashed,font=\sffamily}, 
sv/.style={circle,inner sep=1pt,fill=gray!20,draw,font=\sffamily,minimum size=1mm}, 
node distance=12mm, >=stealth, every node/.style={scale=0.7}}
\begin{scope}
\node[name=A, rv]{$A$};
\node[name=Y, rv, right of=A, xshift=5mm]{$Y$};
\node[name=I, redv, left of=A]{$I_1$};
\node[name=W1, redv, above of=I]{$W_1$};
\node[name=O3, rv, left of=I]{$O_3$};
\node[name=O1, rv, above of=Y, xshift=-5mm]{$O_1$};
\node[name=W2, redv, above of=A]{$W_2$};
\node[name=W3, rv, above of=W2, xshift=-5mm]{$W_3$};
\node[name=W4, rv, above of=W2, xshift=5mm]{$W_4$};
\node[name=O2, rv, above of=Y, xshift=5mm]{$O_2$};
\node[name=W5, rv, above of=Y, yshift=10mm]{$W_5$};
\node[name=W6, redv, right of=W5, xshift=2mm]{$W_6$};
\node[below of=A, yshift=0mm]{(a) $\g$};
\draw[->,very thick, color=blue] (A) to (Y);
\draw[->,very thick, color=blue] (I) to (A);
\draw[->,very thick, color=blue] (W1) to (A);
\draw[->,very thick, color=blue] (W1) to (I);
\draw[->,very thick, color=blue] (W1) to (O3);
\draw[->,very thick, color=blue, bend right] (O3) to (Y);
\draw[->,very thick, color=blue] (O1) to (Y);
\draw[->,very thick, color=blue] (O1) to (A);
\draw[->,very thick, color=blue] (W2) to (A);
\draw[->,very thick, color=blue] (W2) to (O1);
\draw[->,very thick, color=blue] (W3) to (W2);
\draw[->,very thick, color=blue] (W4) to (W2);
\draw[->,very thick, color=blue] (O2) to (Y);
\draw[->,very thick, color=blue] (W5) to (O2);
\draw[->,very thick, color=blue] (W5) to (O1);
\draw[->,very thick, color=blue] (W5) to (W2);
\draw[->,very thick, color=blue] (W5) to (W6);
\draw[->,very thick, color=blue] (W6) to (O2);
\end{scope} \begin{scope}[xshift=4cm]
\node[name=A, rv]{$A$};
\node[name=Y, rv, right of=A, xshift=5mm]{$Y$};
\node[name=W1, redv, left of=A]{$W_1$};
\node[name=O3, rv, left of=W1]{$O_3$};
\node[name=O1, rv, above of=Y, xshift=-5mm]{$O_1$};
\node[name=W2, redv, above of=A]{$W_2$};
\node[name=W3, rv, above of=W2, xshift=-5mm]{$W_3$};
\node[name=W4, rv, above of=W2, xshift=5mm]{$W_4$};
\node[name=O2, rv, above of=Y, xshift=5mm]{$O_2$};
\node[name=W5, rv, above of=Y, yshift=10mm]{$W_5$};
\node[name=W6, redv, right of=W5, xshift=2mm]{$W_6$};
\node[below of=A, yshift=0mm]{(b) $\g^0$};
\node[left=8mm of W2]{\Large $\Rightarrow$};
\draw[->,very thick, color=blue] (A) to (Y);
\draw[->,very thick, color=blue] (W1) to (A);
\draw[->,very thick, color=blue] (W1) to (O3);
\draw[->,very thick, color=blue, bend right] (O3) to (Y);
\draw[->,very thick, color=blue] (O1) to (Y);
\draw[->,very thick, color=blue] (O1) to (A);
\draw[->,very thick, color=blue] (W2) to (A);
\draw[->,very thick, color=blue] (W2) to (O1);
\draw[->,very thick, color=blue] (W3) to (W2);
\draw[->,very thick, color=blue] (W4) to (W2);
\draw[->,very thick, color=blue] (O2) to (Y);
\draw[->,very thick, color=blue] (W5) to (O2);
\draw[->,very thick, color=blue] (W5) to (O1);
\draw[->,very thick, color=blue] (W5) to (W2);
\draw[->,very thick, color=blue] (W5) to (W6);
\draw[->,very thick, color=blue] (W6) to (O2);
\end{scope} \begin{scope}[xshift=8cm]
\node[name=A, rv]{$A$};
\node[name=Y, rv, right of=A, xshift=5mm]{$Y$};
\node[name=O3, rv, left of=A, xshift=-3mm]{$O_3$};
\node[name=O1, rv, above of=Y, xshift=-5mm]{$O_1$};
\node[name=W2, redv, above of=A]{$W_2$};
\node[name=W3, rv, above of=W2, xshift=-5mm]{$W_3$};
\node[name=W4, rv, above of=W2, xshift=5mm]{$W_4$};
\node[name=O2, rv, above of=Y, xshift=5mm]{$O_2$};
\node[name=W5, rv, above of=Y, yshift=10mm]{$W_5$};
\node[name=W6, redv, right of=W5, xshift=2mm]{$W_6$};
\node[below of=A, yshift=0mm]{(c) $\g^1$};
\node[left=8mm of W2]{\Large $\Rightarrow$};
\draw[->,very thick, color=blue] (A) to (Y);
\draw[->,very thick, color=blue] (O3) to (A);
\draw[->,very thick, color=blue, bend right] (O3) to (Y);
\draw[->,very thick, color=blue] (O1) to (Y);
\draw[->,very thick, color=blue] (O1) to (A);
\draw[->,very thick, color=blue] (W2) to (A);
\draw[->,very thick, color=blue] (W2) to (O1);
\draw[->,very thick, color=blue] (W3) to (W2);
\draw[->,very thick, color=blue] (W4) to (W2);
\draw[->,very thick, color=blue] (O2) to (Y);
\draw[->,very thick, color=blue] (W5) to (O2);
\draw[->,very thick, color=blue] (W5) to (O1);
\draw[->,very thick, color=blue] (W5) to (W2);
\draw[->,very thick, color=blue] (W5) to (W6);
\draw[->,very thick, color=blue] (W6) to (O2);
\end{scope} \begin{scope}[xshift=2cm, yshift=-3cm]
\node[name=A, rv]{$A$};
\node[name=Y, rv, right of=A, xshift=3mm]{$Y$};
\node[name=O3, rv, left of=A, xshift=-3mm]{$O_3$};
\node[name=O1, rv, above of=Y, xshift=-5mm]{$O_1$};
\node[name=W3, rv, above of=A, xshift=-8mm, yshift=2mm]{$W_3$};
\node[name=W4, rv, above of=A, xshift=1mm, yshift=8mm]{$W_4$};
\node[name=O2, rv, above of=Y, xshift=5mm]{$O_2$};
\node[name=W5, rv, above of=Y, yshift=10mm]{$W_5$};
\node[name=W6, redv, right of=W5, xshift=2mm]{$W_6$};
\node[below of=A, yshift=0mm]{(d) $\g^2$};
\node[left=3mm of W3]{\Large $\Rightarrow$};
\draw[->,very thick, color=blue] (A) to (Y);
\draw[->,very thick, color=blue] (O3) to (A);
\draw[->,very thick, color=blue] (W3) to (A);
\draw[->,very thick, color=blue] (W4) to (A);
\draw[->,very thick, color=blue, bend right] (O3) to (Y);
\draw[->,very thick, color=blue] (O1) to (Y);
\draw[->,very thick, color=blue] (O1) to (A);
\draw[->,very thick, color=blue] (O2) to (Y);
\draw[->,very thick, color=blue] (W5) to (O2);
\draw[->,very thick, color=blue] (W5) to (O1);
\draw[->,very thick, color=blue] (W3) to (O1);
\draw[->,very thick, color=blue] (W4) to (O1);
\draw[->,very thick, color=blue] (W5) to (W6);
\draw[->,very thick, color=blue] (W6) to (O2);
\draw[->,very thick, color=blue, bend left, out=45] (W5) to (A);
\end{scope} \begin{scope}[xshift=6cm, yshift=-3cm]
\node[name=A, rv]{$A$};
\node[name=Y, rv, right of=A, xshift=3mm]{$Y$};
\node[name=O3, rv, left of=A, xshift=-3mm]{$O_3$};
\node[name=O1, rv, above of=Y, xshift=-5mm]{$O_1$};
\node[name=W3, rv, above of=A, xshift=-8mm, yshift=2mm]{$W_3$};
\node[name=W4, rv, above of=A, xshift=1mm, yshift=8mm]{$W_4$};
\node[name=O2, rv, above of=Y, xshift=5mm]{$O_2$};
\node[name=W5, rv, above of=Y, yshift=10mm]{$W_5$};
\node[below of=A, yshift=0mm]{(e) $\g^{\ast}$};
\node[left=3mm of W3]{\Large $\Rightarrow$};
\draw[->,very thick, color=blue] (A) to (Y);
\draw[->,very thick, color=blue] (O3) to (A);
\draw[->,very thick, color=blue] (W3) to (A);
\draw[->,very thick, color=blue] (W4) to (A);
\draw[->,very thick, color=blue, bend right] (O3) to (Y);
\draw[->,very thick, color=blue] (O1) to (Y);
\draw[->,very thick, color=blue] (O1) to (A);
\draw[->,very thick, color=blue] (O2) to (Y);
\draw[->,very thick, color=blue] (W5) to (O2);
\draw[->,very thick, color=blue] (W5) to (O1);
\draw[->,very thick, color=blue] (W3) to (O1);
\draw[->,very thick, color=blue] (W4) to (O1);
\draw[->,very thick, color=blue, bend left, out=45] (W5) to (A);
\end{scope}
\end{tikzpicture}
\caption{Graph reduction for \cref{ex:long}, where $\Vset \setminus \Vset^{\ast} = \{I_1, W_1, W_2, W_6\}$.}
\label{fig:long}
\end{figure}

\end{example}

\section{Concluding remarks} \label{sec:conclusion}
When all variables in the graph are discrete, an asymptotically efficient estimator based on the set of irreducible informative variables is readily available as $\Psi_{a}(\mathbb{P}_{n}^{\ast};\g^{\ast}) $. Unfortunately, when not all components of $\Vset^{\ast}$ are discrete, the plug-in estimator $\Psi_{a}(\widehat{P^{\ast}}; \g^{\ast})$ for $\widehat{P^{\ast}} \in \model(\g^{\ast},\Vset^{\ast})$ based on smooth nonparametric estimators of the conditional densities $\left\{p\{v_j|\pa(v_j, \g^{\ast})\}: V_j \in \Vset^{\ast} \right\}$ will generally fail to even be root-$n$-consistent. This is because $\Psi_{a}(\widehat{P^{\ast}}; \g^{\ast})$ will typically inherit the bias and thus the rate of convergence of the nonparametric density estimators. The one-step estimator $\widehat{\Psi}_{a}=\Psi_{a}(\widehat{P^{\ast }},\g^{\ast }) + \mathbb{P}_{n} \left\{ \Psi^{1}_{a,\widehat{P^{\ast}},\text{eff}}(\Vset)\right\} $ corrects the bias, and under smoothness or complexity assumptions on the conditional densities it converges at the root-$n$ rate and is asymptotically efficient. However, the calculation of $\Psi^{1}_{a,\widehat{P^{\ast}},\text{eff}}(\Vset)$ will typically require evaluating  complicated integrals involved in the computation of each $\mathbb{E}_{\widehat{P^{\ast}}}\left\{ b_{a, \widehat{P^\ast}}(\Oset) \mid W_{j}, \Pa(W_{j}, \g^{\ast}) \right\} $ and each $\mathbb{E}_{\widehat{P^{\ast}}}\left\{ T_{a, \widehat{P^\ast}} \mid M_{k}, \Pa(M_{k}, \g^{\ast}) \right\} $; see \cref{lem:EIF}. Further work exploring methods that facilitate these calculations is warranted. 

In this article we have considered estimating the mean of an outcome under an intervention that sets a point exposure to a fixed value in the entire population. This is just one out of the many functionals of interest in causal inference. We hope this work sparks interest in the characterization of informative irreducible variables for other functionals. In particular, we are currently studying the extension of the present work to interventions that set the treatment to a value that depends on covariates, i.e., the so-called dynamic treatment regimes. Extensions to time-dependent interventions in graphs with time-dependent confounding is also of interest, but appears to be more difficult because an optimal time-dependent adjustment set does not exist \citep{rotnitzky2020efficient}. Other functionals of interest include the pure direct effect and the treatment effect on the treated.

\section*{Acknowledgement}
The authors thank Thomas Richardson and James Robins for valuable comments and discussions, as well as the referees and the associate editor for helpful suggestions. Part of this work was done while the authors were visiting the Simons Institute for the Theory of Computing.
Rotnitzky is partially supported by the U.S. National Institutes of Health and is also affiliated with CONICET, Argentina. 

\bibliographystyle{plainnat}

\newpage
\appendix

The \supps are organized as follows. 
\Cref{sec:app-DAG} provides additional background on Bayesian networks and graphical models. 
\Cref{sec:app-simu} contains simulation results based on the motivating example. In \cref{sec:app-graph}, we prove graphical results \cref{lem:osetandw,prop:c-MEC-Oset}. In \cref{sec:app-reduction}, we prove the results for the graph reduction procedure including \cref{thm:algo}. In \cref{sec:app-eff-formula}, we establish that the g-formula is an efficient identifying formula for the g-functional. In \cref{sec:app-uninfo}, we prove results for graphically characterizing the irreducible informative set $\Vset^{\ast}$, culminating in the proof of \cref{thm:criteria}. To prove \cref{thm:criteria}, we construct laws in the model such that the efficient influence function non-trivially depends on the variables should the variable fail the corresponding W- or M-criterion; these constructions are detailed in \cref{sec:app-W-complete,sec:app-M-complete}. Further, these constructions are based on the certain graphical configurations that must exist should the W- or M-criterion fail, which are provided in \cref{sec:app-aug-graph}.  

Throughout the \supps, we will often omit the graph from the vertex sets introduced in \cref{sec:sets} when it is clear from the context.

\section{Bayesian network and directed acyclic graph} \label{sec:app-DAG}
\subsection{Bayesian network on a large state space} \label{sec:app-state-space}
For every random variable $V_j \in \Vset$, its state space $\X_j$ is defined in \cref{eqs:state-space}, which allows $V_j$ to be potentially continuous or discrete or a mixed type of both. 
Any set $A \subseteq \X_j$ can be decomposed as $A = A_{C} \dcup A_{D}$ for a continuous part $A_C \equiv A \cap \mathbb{R}^{d_j}$ and a discrete part $A_D \equiv A \cap \mathbb{W}$. Let $\F_j$ be the $\sigma$-algebra generated by unions of Borel sets on $\mathbb{R}^{d_j}$ and on $\mathbb{W}$. Define measure $\mu_j(A) \equiv \text{Leb}(A_C) +  |A_D|$ for $A \in \F_j$. Finally, the measurable space for random vector $\Vset$ is $(\X, \F)$, where $\X \equiv \bigtimes_{j: V_j \in \Vset} \X_j$ and $\F \equiv \bigtimes_{j: V_j \in \Vset} \F_j$. 

\subsection{Notations for graphical models} \label{sec:app-graph-notation}
Symbol $V_{i}\rightarrow V_{j}$ or $V_j \leftarrow V_i$ denotes a directed edge from $V_{i}$ to $V_{j}$. In such case, we say that $V_{i}$ is a parent of $V_{j}$, $V_{j}$ is a child of $V_{i}$ and denoted with $V_i \in \Pa(V_j, \g)$, $V_j \in \Ch(V_i, \g)$. We say $V_{i}$ and $V_{j}$ are adjacent if $V_i \rightarrow V_j$ or $V_i \leftarrow V_j$. A path $p=\langle V_{1},\dots ,V_{k}\rangle $ with $k\geq 2$ is a sequence of distinct vertices, such that $V_{i}$ and $V_{i+1}$ are adjacent in the graph for $i=1,\dots ,k-1$. When $p$ is of the form $V_{1}\rightarrow \dots \rightarrow V_{k}$, we say that $p$ is a directed path. In a configuration $ V_{i}\rightarrow V_{k}\leftarrow V_{j}$, $V_{k}$ is said to be a collider; further, if $V_{i}$ and $V_{j}$ are non-adjacent, then $V_{k}$ is an unshielded collider. We say $V_{i}$ is an ancestor of $V_{j}$, or equivalently $V_{j}$ is a descendant of $V_{i}$, if either $V_{i}=V_{j}$ or there is a directed path from $V_{i}$ to $V_{j}$. This is denoted as $V_{i}\mapsto V_{j}$. Symbol $V_i \not\mapsto V_j$ means $V_i$ is not an ancestor of $V_j$. The set of ancestors of $V_{j}$ with respect to graph $\g$ is denoted by $\An (V_{j},\g)$, and similarly, the set of descendants of $V_{i}$ is denoted by $ \De(V_{i},\g)$; by definition, we have $V_{j}\in \An(V_{j},\g)$ and $ V_{i}\in \De(V_{i},\g)$. The definitions of relational sets extend disjunctively to a set of vertices $\Cset$, e.g., $\Pa(\Cset,\g)\equiv \bigcup_{V_{j}\in \Cset}\Pa(V_{j},\g)$, $\Ch(\Cset,\g)\equiv \bigcup_{V_{j}\in \Cset}\Ch(V_{j},\g)$, $\An(\Cset,\g)\equiv \bigcup_{V_{j}\in \Cset}\An(V_{j},\g)$, $\De(\Cset,\g)\equiv \bigcup_{V_{j}\in \Cset}\De (V_{j},\g)$, etc. Moreover, a set $\Cset$ is called ancestral if $V_{j}\in \Cset$ implies $\An(V_{j},\g)\subseteq \Cset$. Vertices $V_{1},\dots ,V_{k}$ are topologically ordered if $V_i \mapsto V_j$ for $i \neq j$ implies $i < j$.

\section{Simulations} \label{sec:app-simu}
We report simulation results based on the motivating example in \cref{fig:example}. For simplicity, we get rid of $I_1$ and $W_1$ and fix $\g'$ as the original graph, according to which data is generated. 
We consider discrete data generating mechanisms according to $\g'$. Let $A, O_1, Y$ be binary valued, i.e. taking value from $\{0,1\}$. Suppose $W_4$ takes value from $\{0,\dots,k-1\}$; while both $W_2$ and $W_3$ take value from $\{0,\dots,m-1\}$. 

We specify the data generating mechanism as follows. First, we draw
\[ W_2, W_3 \sim \unif(\{0, \dots, m-1\}) \]
independently. Then, we reserve $\{0,\dots,4\}$ in the support of $W_4$ as a special set of values and draw
\[ W_4 \mid W_2, W_3 \sim \begin{cases} \unif(\{0, \dots, 4\}), &\quad W_2 = W_3 \\
Q, & \quad W_2 \neq W_3 \end{cases}, \]
where 
\[ Q(w_4) \propto \begin{cases} \left[1 + \exp\{-(w_4+1) / 5\} \right]^{-1}, \quad & w_4 \notin \{0,\dots,4\} \\
0, \quad & \text{otherwise} \end{cases}.\]
Further, $O$'s distribution depends on whether $W_4$ takes a special value:
\[ P(O=1 \mid W_4 = w_4) = \begin{cases} 0.99, \quad & w_4 \in \{0,\dots,4\} \\
0.01, \quad & \text{otherwise} \end{cases}. \]
We use these configurations to introduce the interaction between $W_2$ and $W_3$, so that the marginal independence between $W_2$ and $W_3$ must be utilized to minimize the variance. 
Finally, we draw $A$ and $Y$ according to 
\[ P(A=1 \mid W_4=w_4) = \left[1 + \exp\{2 -(w_4+1) / 5\} \right]^{-1} \]
and 
\[ P(Y=1 \mid A=a, O_1=o_1) = \left[1 + \exp\{-(o_1 - 1/2)(9a + 5)\} \right]^{-1}. \]

We consider three estimators for $\Psi_{1}(P; \g)$:
\begin{enumerate}
\item Plugin g-formula $\Psi_{1}(\mathbb{P}_n; \g')$ based on the original graph $\g'$.
\item Plugin g-formula $\Psi_{1}(\mathbb{P}_n; \g^{\ast})$ based on the reduced graph $\g^{\ast}$, which does not use $W_4$.
\item Optimal adjustment $\Psi_{1,O_1}^{\text{ADJ}}(\mathbb{P}_n; \g') = \sum_{o_1} \mathbb{P}_n(Y=1 \mid A=1, O_1 = o_1) \mathbb{P}_n(o_1)$, which does not use $W_2, W_3, W_4$.   
\end{enumerate}

We perform simulations in two settings: (a) $m=5$, $k=50$ and (b) $m=50$, $k=10$. The results are reported in \cref{tab:simu}. We select different sample sizes in each setting; the smallest sample size is chosen such that all levels of $(W_2, W_4)$ appear in the data. In both settings, while $\Psi_{1}(\mathbb{P}_n; \g')$ and $\Psi_{1}(\mathbb{P}_n; \g^{\ast})$ achieve a significantly smaller variance compared to $\Psi_{1,O_1}^{\text{ADJ}}(\mathbb{P}_n; \g')$, there is no discernible difference between their performances. 

\begin{table}[!hbtp]
\begin{center}
\caption{Variance of the estimator multiplied by sample size $n$. The number of replications is large enough such that the standard error is within the next significant digit.} 
\label{tab:simu}
\begin{tabular}{rlccc}
\toprule[1.2pt]
& $n$ & $\Psi_{1,O_1}^{\text{ADJ}}(\mathbb{P}_n; \g')$ & $\Psi_{1}(\mathbb{P}_n; \g')$ & $\Psi_{1}(\mathbb{P}_n; \g^{\ast})$ \\ 
\midrule
\textbf{(a)} &&&&\\ 
1&200&0.163&0.013&0.013\\ 
2&500 &0.164&0.012&0.012\\ 
3&1000 &0.164&0.012&0.012\\ 
4&10000 &0.165&0.012&0.012\\ 
5&25000 &0.163&0.012&0.012\\ 
6&50000 &0.168&0.012&0.012\\ 
7&100000 &0.161&0.011&0.011\\ 
\textbf{(b)}&&&&\\ 
8&25000 &0.031&0.012&0.013\\ 
9&50000 &0.031&0.012&0.013\\ 
10&100000 &0.031&0.012&0.012\\ 
\bottomrule[1.2pt]
\end{tabular}
\end{center}
\end{table}

\section{Various graphical results} \label{sec:app-graph}

\subsection{Proof of \cref{lem:osetandw}}
\begin{proof}
First, we show that $\Wset \subseteq \An(\Oset)$. Fix any $W_j \in \Wset$, we want to show that $W_j \mapsto O$. By definition in \cref{eqs:Wset}, there exists a causal path $p: W_j \rightarrow Z_1 \rightarrow \dots \rightarrow Z_k = Y$ such that $p$ does not contain $A$. By definition of set $\Mset$, $Z_i \in \Mset$ implies $Z_{i+1} \in \Mset$ for $i=1,\dots,k-1$. Let $l$ be the smallest index such that $Z_l \in \Mset$. If $l=1$, then $W_j \in \Oset$ by definition in \cref{eqs:Oset}; otherwise, $Z_{l-1} \in \Oset$ by definition as $Z_{l-1} \rightarrow Z_l \in \Mset$ but $Z_{l-1} \notin \De(\Mset)$, $Z_{l-1} \neq A$. In either case, $W_j \mapsto \Oset$. 

Now we show that $\An(\Oset) \subseteq \Wset$. Pick any $V_i$ such that $V_i \mapsto O_k$ for $O_k \in \Oset$. Showing that $V_i \in \Wset$ boils down to showing (i) $V_i \mapsto Y$ not through $A$ and (ii) $V_i \notin \De(A)$. Note (i) is follows from $V_i \mapsto O_k \mapsto Y$, which need not go through $A$. To see (ii), suppose $V_i \in \De (A)$. Then, it follows that $O_k \in \De(A)$, which implies $O_k \in \Mset$, contradicting the definition of $\Oset$. 
\end{proof}

\subsection{Proof of \cref{prop:c-MEC-Oset}}

\begin{proof}
Suppose $\g$ and $\g'$ are causal Markov equivalent. By \cref{def:c-MEC}, $\g$ and $\g'$ are then Markov equivalent and $\Psi_{a}(P;\g)=\Psi_{a}(P;\g')$ for all $P \in \model(\g, \Vset)$. Given that $\g$ and $\g'$ satisfy \cref{assump:A-goes-to-Y}, by Theorem 3 of \cite{guo2020enumeration}, $\g$ and $\g'$ are then represented by a maximally oriented partially directed acyclic graph $\tilde{\g}$, given which the total causal effect of $A$ on $Y$ is identified. Furthermore, since $|A| = |Y| = 1$, by Corollary 2 of \cite{guo2020enumeration}, there exists an adjustment set with respect to the effect of $A$ on $Y$ in $\tilde{\g}$. Then by Lemma E.7 of \cite{henckel2019graphical}, the optimal adjustment set with respect to the effect of $A$ on $Y$ is the same in $\tilde{\g}$, $\g$, and $\g'$.

Conversely, suppose that $\g$ and $\g'$ are Markov equivalent and that the optimal adjustment set with respect to the effect of $A$ on $Y$ is the same in $\g$ and $\g'$. Then using the fact that $\Psi_{a}(P;\g)= \Psi_{a,\Oset}^{\text{ADJ}}(P;\g)$  for all $P \in \model(\g, \Vset)$ and $\Psi_{a,\Oset}^{\text{ADJ}}(P;\g')= \Psi_{a}(P;\g')$ for all $P \in \model(\g', \Vset)$, and that $\g$ and $\g'$ are Markov equivalent, we have, $\Psi_{a}(P;\g)=\Psi_{a}(P;\g')$ for all $P \in \model(\g, \Vset)$. 
\end{proof}

\section{Graph reduction} \label{sec:app-reduction}

\subsection{Proof of \cref{lem:marg-over-NI}} \label{sec:app-marg-over-NI}
Latent projection is introduced by \citet{verma1991equivalence} to represent the marginal model of a directed acyclic graph. Recall that for $L \subseteq \Vset$, $\model(\g, L)$ denotes the marginal model of $L$ induced by $\model(\g, \Vset)$.
\begin{definition}[Latent projection] \label{def:latent-proj}
Let $\g$ be a directed acyclic graph on vertices $\Vset \,\dot{\cup}\, \Uset$, where $\Uset$ is the set of latent variables. The latent projection of $\g$ on $\Vset$, denoted by $\g(\Vset)$, is a mixed graph on $\Vset$ with directed and bidirected edges:
\begin{enumerate}
\item $V_i \rightarrow V_j$ if there is a directed path $V_i \mapsto V_j$ in $\g$ on which every non-endpoint vertex is in $\Uset$,
\item $V_i \leftrightarrow V_j$ if there exists a path of the form $V_i \leftarrow \dots \rightarrow V_j$ in $\g$ on which every non-endpoint vertex is a non-collider and contained in $\Uset$. 
\end{enumerate}
\end{definition}

For graph $\g$ on vertices $\Vset$ and $\Vset' \subseteq \Vset$, let $\g_{\Vset'}$ be the subgraph induced by $\Vset'$. 
\begin{lemma}[Proposition 3.22, \citealp{lauritzen1996graphical}] \label{lem:ancestral-marg}
Let $\g$ be a directed acyclic graph on vertices $\Vset$. Let $\Vset'$ be an ancestral subset of $\Vset$. Then, $\model(\g, \Vset') = \model(\g_{\Vset'}, \Vset')$.
\end{lemma}

The following definition and two lemmas are due to \citet{evans2016graphs}.
\begin{definition}[Exogenized graph] \label{def:exo-graph}
Let $\g$ be a directed acyclic graph containing vertex $U_i$. The exogenized graph $\mathfrak{r}(\g, U_i)$ is a directed acyclic graph transformed from $\g$ as follows. For each $V_j \in \Pa(U_i)$ and $V_k \in \Ch(U_i)$, add edge $V_j \rightarrow V_k$ if the edge is not already present. Then, remove all edges between $\Pa(U_i)$ and $U_i$. 
\end{definition}

\begin{lemma} \label{lem:exo-graph}
Let $\g$ be a directed acyclic graph on vertices $\Vset \cup \{U_i\}$. Let $\g' = \mathfrak{r}(\g, U_i)$. Then, $\model(\g, \Vset) = \model(\g', \Vset)$. 
\end{lemma}

\begin{lemma} \label{lem:rm-single-child}
Let $\g$ be a directed acyclic graph on vertices $\Vset \cup \{U_i\}$, such that $U_i$ has no parent and at most one child. Then, $\model(\g, \Vset) = \model(\g_{-U_i}, \Vset)$, where graph $\g_{-U_i}$ denotes the graph from removing vertex $U_i$. 
\end{lemma}

\medskip 
\begin{proof}[of \cref{lem:marg-over-NI}]
It is easy to see that
\[ \g^0 = \g_{\Vset \setminus \Nset(\g)}(\Vset^0), \]
namely the graph obtained from first removing vertices $\Nset(\g)$ and then projecting out $\Iset(\g)$ with latent projection. 

First, let $\tilde{\Vset}^0 \equiv \Vset \setminus \N(\g)$ and $\tilde{\g}^0 \equiv \g_{\Vset \setminus \Nset(\g)}$. Because $\tilde{\Vset}^0$ is ancestral, by \cref{lem:ancestral-marg}, we have
\[ \model(\g,\tilde{\Vset}^0) = \model(\tilde{\g}^0, \tilde{\Vset}^0). \]

Now we iteratively project out vertices in $\Iset(\g)$. Suppose $\Iset(\g) = \{I_1, \dots, I_L\}$ is topologically indexed. By definition of $\Iset$ and the fact that every vertex an ancestor of $Y$ in $\tilde{\g}^0$, it must hold that $\Ch(I_L, \tilde{\g}^0) = \{A\}$. Let $\tilde{\g}_{e}^{1} = \mathfrak{r}(\tilde{\g}^0, I_L)$ be the graph from exogenizing $I_L$. In $\tilde{\g}_{e}^{1}$, $A$ is still the only child of $I_L$. Also, let $\tilde{\g}^{1}$ be the graph by removing $I_L$ from $\tilde{\g}_{e}^{1}$. Then, with $\tilde{V}^1 \equiv \tilde{V}^0 \setminus \{I_L\}$, by \cref{lem:exo-graph,lem:rm-single-child}, we have
\[ \model(\tilde{\g}^0, \tilde{V}^1) = \model(\tilde{\g}_{e}^1, \tilde{V}^1) = \model(\tilde{\g}^1, \tilde{V}^1).\]
Comparing \cref{def:latent-proj,def:exo-graph}, it is easy to see that $\tilde{\g}^1 = \tilde{\g}^0(\tilde{V}^1)$ and $\tilde{\g}^1$ is a directed acyclic graph. Now $I_{L-1}$ becomes the vertex with sole child $A$ in $\tilde{\g}^1$. Continuing this operation for $I_{L-1}, \dots, I_1$, we arrive at graph $\tilde{\g}^{L}$ on vertices $\tilde{V}^L \equiv \tilde{\Vset}^0 \setminus \Iset = \Vset \setminus \{\Nset(\g) \cup \Iset(\g)\} = \Vset^0$. Because the marginal model is taken iteratively, we have
\[ \model(\g, \Vset^0) = \model(\tilde{\g}^{L}, \Vset^0). \]
Further, because $\tilde{\g}^{0} = \g_{\tilde{V}^0} = \g(\tilde{V}^0), \tilde{\g}^{1} = \tilde{\g}^{0}(\tilde{V}^1), \dots, \tilde{\g}^{L} = \tilde{\g}^{L-1}(\tilde{V}^L)$ with iterative latent projections and every projection remains a directed acyclic graph, we have 
\[ \tilde{\g}^L = \g(\Vset^0) = \g_{\Vset \setminus \Nset(\g)}(\Vset^0) =  \g^0\]
by commutativity of latent projection; see \citet[Theorem 1]{evans2016graphs}. It also follows that $\tilde{\g}^L$ is a directed acyclic graph over vertices $\Vset^0$. 
\end{proof}

\subsection{Projecting out uninformative variables in $\Wset(\g) \cup \Mset(\g)$} \label{sec:app-proj}
In the following, we use the mDAG representation to prove \cref{lem:project-DAG}. Marginal directed acyclic graphs, or mDAGs, are a class of hyper-graphs introduced by \citet{evans2016graphs} to represent the marginal model of directed acyclic graphs. As opposed to the latent projection of \citet{verma1991equivalence} that introduces bidirected edges, mDAGs use hyper-edges to signify latent variables that confound two or more observed variables.  An mDAG $\mathcal{H} = (V, \mathcal{E}, \mathcal{B})$ is a hyper-graph on vertices $V$ with directed edges $\mathcal{E}$ and hyper-edges $\mathcal{B}$, where $\mathcal{B}$ is an abstract simplicial complex. Elements of $\mathcal{B}$ are called bidirected faces. A face is called trivial if it is a  singleton set. The inclusion maximal elements of $\mathcal{B}$ are called facets. A directed acyclic graph is an mDAG with trivial bidirected facets. 

It can be shown that if two projections lead to the same mDAG, they the marginal models must be the same. For an mDAG $\mathcal{H}$, let $\model_m(\mathcal{H})$ denote the marginal model represented by $\mathcal{H}$. More concretely, $\model_m(\mathcal{H})$ can be taken as the marginal model of the canonical directed acyclic graph associated with $\mathcal{H}$, which replaces every non-trivial facet with an exogenous latent variable. 

\begin{lemma}[Proposition 5, \citealp{evans2016graphs}] \label{lem:mDAG}
Let $\mathcal{H}$ be an mDAG containing a bidirected facet $B = C \dcup D$ such that 
\begin{enumerate}[(i)]
\item every bidirected face containing any $C_i \in C$ is a subset of $B$; and
\item $\Pa(C, \mathcal{H}) \subseteq \Pa(D_j, \mathcal{H})$ for every $D_j \in D$.
\end{enumerate}
Let $\mathcal{H}'$ be the mDAG defined from $\g$ by removing facet $B$ and replacing it with $C$ and $D$, and adding edges $C_i \rightarrow D_j$ for each $C_i \in C$ and $D_j \in D$, if the edge is not already present. Then, 
\[ \model_m(\mathcal{H}) = \model_m(\mathcal{H}'). \]
\end{lemma}

\medskip 
\begin{proof}[of \cref{lem:project-DAG}]
Let $\mathcal{H}^1$ be the mDAG from projecting out $V_i$; see \citet{evans2016graphs} for details. Graph $\mathcal{H}^1$ is a graph on $\Vset$ consisting of both directed edges and a bidirected facet, which represents the marginal model of $\g$ when $V_i$ is marginalized over, as denoted by
\[ \model_m(\mathcal{H}^1) = \model(\g, \Vset \setminus \{V_i\}). \]
By construction, we have $V_k \rightarrow V_{i_j}$ in $\mathcal{H}^1$ for every $V_k \in \Pa(V_i, \g)$ and $j=1,\dots,l$. Set $B^1 = \{V_{i_1}, \dots, V_{i_l}\}$ is the only bidirected facet in $\mathcal{H}^1$. Partition $B^1$ into $C^1 = \{V_{i_1}\}$ and $D^1 = \{V_{i_2}, \dots, V_{i_l}\}$. We observe that (i) every bidirected face that contains $V_{i_1} \in C^1$ is a subset of $B^1$, which follows trivially from $B^1$ being the only facet. We also observe that (ii) $\Pa(C^1, \mathcal{H}^1) = \Pa(V_{i_1}, \mathcal{H}^1) \subseteq \Pa(D_j, \mathcal{H}^1)$ for every $D_j \in D^1$. Statement (ii) follows from 
\[ \Pa(V_{i_1}, \mathcal{H}^1) = \Pa(V_i, \g), \]
which holds because (a) $\Pa(V_i, \g) \subseteq \Pa(V_{i_1}, \mathcal{H}^1)$ by construction of $\mathcal{H}^1$ and (b) $\Pa(V_i, \g) \supseteq \Pa(V_{i_1},\mathcal{H}^1)$ by \cref{eqs:projection-pa} when $j=1$. By \cref{lem:mDAG}, we have
\[ \model_m(\mathcal{H}^1) = \model_m(\mathcal{H}^2), \]
where in $\mathcal{H}^2$, an mDAG on the same set of vertices, facet $B^1$ is replaced by facet $D^1$ and edges $\{V_{i_1} \rightarrow D_j: D_j \in D^1\}$ are added, if not already present. 

In graph $\mathcal{H}^2$, $B^2 = \{V_{i_2}, \dots, V_{i_l}\}$ is the only bidirected facet
 and can partitioned into $C^2 = \{V_{i_2}\}$ and $D^2 = \{V_{i_3}, \dots, V_{i_l}\}$. We claim that 
\[ \Pa(V_{i_2}, \mathcal{H}^2) = \{V_{i_1}\} \cup \Pa(V_i, \g),\]
which follows from (a) $\{V_{i_1}\} \cup \Pa(V_i, \g) \subseteq \Pa(V_{i_2}, \mathcal{H}^2)$ by construction of $\mathcal{H}^2$, and (b) $\{V_{i_1}\} \cup \Pa(V_i, \g) \supseteq \Pa(V_{i_2}, \mathcal{H}^2)$ by \cref{eqs:projection-pa} and construction of $\mathcal{H}^2$. Therefore, we again have (i) every bidirected face containing $V_{i_2} \in C^2$ is a subset of $B^2$, and (ii) $\Pa(C^2, \mathcal{H}^2) \subseteq \Pa(D_i, \mathcal{H}^2)$ for every $D_j \in D^2$. Applying \cref{lem:mDAG} again, we have
\[ \model_m(\mathcal{H}^2) = \model_m(\mathcal{H}^3),\]
where $\mathcal{H}^3$ is the mDAG formed by replacing facet $B^2$ by facet $D^2$ and adding edges $\{V_{i_2} \rightarrow D_j: D_j \in D^2\}$ if not already present. 

Iterating this process, we get a sequence of mDAGs $\mathcal{H}^1, \mathcal{H}^2, \dots, \mathcal{H}^l$; see \cref{fig:project-DAG} for an example. The last graph $\mathcal{H}^l$ contains no non-trivial bidirected facet and it is easy to see that $\mathcal{H}^l$ is a directed acyclic graph on vertices $\Vset \setminus \{V_i\}$. Further, graph $\mathcal{H}^l$ is identical to $\g_{-V_i,\pi}$ described in the lemma. The proof is complete upon noting
\[ \model(\g, \Vset \setminus \{V_i\}) = \model_m(\mathcal{H}^1) = \model_m(\mathcal{H}^2) =\dots = \model_m(\mathcal{H}^l) = \model(\mathcal{H}^l, \Vset \setminus \{V_i\}). \]
\end{proof}

\begin{figure}[!htb]
\centering
\begin{tikzpicture}
\tikzset{rv/.style={circle,inner sep=1pt,fill=gray!20,draw,font=\sffamily}, 
redv/.style={circle,inner sep=1pt,fill=white,draw,dashed,font=\sffamily}, 
sv/.style={circle,inner sep=1pt,fill=gray!20,draw,font=\sffamily,minimum size=1mm},
uv/.style={circle, minimum size=2mm, inner sep=0mm, fill=red}, 
node distance=12mm, >=stealth, every node/.style={scale=0.6}}
\begin{scope}
\node[name=Wj, rv]{$W_j$};
\node[name=Wl, rv, above of=Wj, xshift=-6mm]{$W_l$};
\node[name=Wk, rv, above of=Wj, xshift=6mm]{$W_k$};
\node[name=Wj1, rv, below of=Wj, xshift=-6mm]{$W_{j_1}$};
\node[name=Wj2, rv, below of=Wj, xshift=6mm]{$W_{j_2}$};
\node[name=Wj3, rv, right of=Wj2]{$W_{j_3}$};
\node[name=A, rv, right of=Wj3, xshift=2mm]{$A$};
\node[name=U, below of=Wj2, xshift=2mm]{};
\node[below of=U, yshift=6mm] {(a) $\g$};
\draw[->, very thick, color=blue] (Wl) -- (Wj);
\draw[->, very thick, color=blue] (Wk) -- (Wj);
\draw[->, very thick, color=blue] (Wj) -- (Wj1);
\draw[->, very thick, color=blue] (Wj) -- (Wj2);
\draw[->, very thick, color=blue] (Wj) -- (Wj3);
\draw[->, very thick, color=blue] (Wj) to (A);
\draw[->, very thick, color=blue] (Wj1) -- (Wj2);
\draw[->, very thick, color=blue] (Wj2) -- (Wj3);
\draw[->, very thick, color=blue, bend right] (Wl) to (Wj1);
\end{scope} \begin{scope}[xshift=3cm]
\node[name=Wj]{};
\node[name=Wl, rv, above of=Wj, xshift=-6mm]{$W_l$};
\node[name=Wk, rv, above of=Wj, xshift=6mm]{$W_k$};
\node[name=Wj1, rv, below of=Wj, xshift=-6mm]{$W_{j_1}$};
\node[name=Wj2, rv, below of=Wj, xshift=6mm]{$W_{j_2}$};
\node[name=Wj3, rv, right of=Wj2]{$W_{j_3}$};
\node[name=A, rv, right of=Wj3, xshift=2mm]{$A$};
\node[name=U, uv, below of=Wj2, xshift=6mm]{};
\node[left of=Wj, xshift=-1mm] {\Large $\Rightarrow$};
\node[below of=U, yshift=6mm] {(b) $\mathcal{H}^1$};
\draw[->, very thick, color=blue] (Wj1) -- (Wj2);
\draw[->, very thick, color=blue] (Wj2) -- (Wj3);
\draw[->, very thick, color=blue] (Wl) to (Wj1);
\draw[->, very thick, color=blue] (Wk) to (Wj1);
\draw[->, very thick, color=blue] (Wl) to (Wj2);
\draw[->, very thick, color=blue] (Wk) to (Wj2);
\draw[->, very thick, color=blue] (Wl) to (Wj3);
\draw[->, very thick, color=blue] (Wk) to (Wj3);
\draw[->, very thick, color=blue] (Wl) to (A);
\draw[->, very thick, color=blue] (Wk) to (A);
\draw[->, very thick, color=red] (U) to (Wj1);
\draw[->, very thick, color=red] (U) to (Wj2);
\draw[->, very thick, color=red] (U) to (Wj3);
\draw[->, very thick, color=red] (U) to (A);
\end{scope} \begin{scope}[xshift=6cm]
\node[name=Wj]{};
\node[name=Wl, rv, above of=Wj, xshift=-6mm]{$W_l$};
\node[name=Wk, rv, above of=Wj, xshift=6mm]{$W_k$};
\node[name=Wj1, rv, below of=Wj, xshift=-6mm]{$W_{j_1}$};
\node[name=Wj2, rv, below of=Wj, xshift=6mm]{$W_{j_2}$};
\node[name=Wj3, rv, right of=Wj2]{$W_{j_3}$};
\node[name=A, rv, right of=Wj3, xshift=2mm]{$A$};
\node[name=U, uv, below of=Wj3, xshift=0mm]{};
\node[left of=Wj, xshift=-1mm] {\Large $\Rightarrow$};
\node[below of=U, yshift=6mm, xshift=-5mm] {(c) $\mathcal{H}^2$};
\draw[->, very thick, color=blue] (Wj1) -- (Wj2);
\draw[->, very thick, color=blue] (Wj2) -- (Wj3);
\draw[->, very thick, color=blue] (Wl) to (Wj1);
\draw[->, very thick, color=blue] (Wk) to (Wj1);
\draw[->, very thick, color=blue] (Wl) to (Wj2);
\draw[->, very thick, color=blue] (Wk) to (Wj2);
\draw[->, very thick, color=blue] (Wl) to (Wj3);
\draw[->, very thick, color=blue] (Wk) to (Wj3);
\draw[->, very thick, color=blue] (Wl) to (A);
\draw[->, very thick, color=blue] (Wk) to (A);
\draw[->, very thick, color=blue, bend right] (Wj1) to (Wj3);
\draw[->, very thick, color=blue, bend right] (Wj1) to (A);
\draw[->, very thick, color=red] (U) to (Wj2);
\draw[->, very thick, color=red] (U) to (Wj3);
\draw[->, very thick, color=red] (U) to (A);
\end{scope} \begin{scope}[xshift=9cm]
\node[name=Wj]{};
\node[name=Wl, rv, above of=Wj, xshift=-6mm]{$W_l$};
\node[name=Wk, rv, above of=Wj, xshift=6mm]{$W_k$};
\node[name=Wj1, rv, below of=Wj, xshift=-6mm]{$W_{j_1}$};
\node[name=Wj2, rv, below of=Wj, xshift=6mm]{$W_{j_2}$};
\node[name=Wj3, rv, right of=Wj2]{$W_{j_3}$};
\node[name=A, rv, right of=Wj3, xshift=2mm]{$A$};
\node[name=U, uv, below of=Wj3, xshift=7mm]{};
\node[left of=Wj, xshift=-1mm] {\Large $\Rightarrow$};
\node[below of=U, yshift=6mm, xshift=-10mm] {(d) $\mathcal{H}^3$};
\draw[->, very thick, color=blue] (Wj1) -- (Wj2);
\draw[->, very thick, color=blue] (Wj2) -- (Wj3);
\draw[->, very thick, color=blue] (Wl) to (Wj1);
\draw[->, very thick, color=blue] (Wk) to (Wj1);
\draw[->, very thick, color=blue] (Wl) to (Wj2);
\draw[->, very thick, color=blue] (Wk) to (Wj2);
\draw[->, very thick, color=blue] (Wl) to (Wj3);
\draw[->, very thick, color=blue] (Wk) to (Wj3);
\draw[->, very thick, color=blue] (Wl) to (A);
\draw[->, very thick, color=blue] (Wk) to (A);
\draw[->, very thick, color=blue, bend right] (Wj1) to (Wj3);
\draw[->, very thick, color=blue, bend right] (Wj1) to (A);
\draw[->, very thick, color=blue, bend right] (Wj2) to (A);
\draw[->, very thick, color=red] (U) to (Wj3);
\draw[->, very thick, color=red] (U) to (A);
\end{scope} \begin{scope}[xshift=12cm]
\node[name=Wj]{};
\node[name=Wl, rv, above of=Wj, xshift=-6mm]{$W_l$};
\node[name=Wk, rv, above of=Wj, xshift=6mm]{$W_k$};
\node[name=Wj1, rv, below of=Wj, xshift=-6mm]{$W_{j_1}$};
\node[name=Wj2, rv, below of=Wj, xshift=6mm]{$W_{j_2}$};
\node[name=Wj3, rv, right of=Wj2]{$W_{j_3}$};
\node[name=A, rv, right of=Wj3, xshift=2mm]{$A$};
\node[name=U, below of=Wj3, xshift=0mm]{};
\node[left of=Wj, xshift=-1mm] {\Large $\Rightarrow$};
\node[below of=U, yshift=6mm, xshift=-5mm] {(e) $\mathcal{H}^4 = \g_{-W_j,\pi}$};
\draw[->, very thick, color=blue] (Wj1) -- (Wj2);
\draw[->, very thick, color=blue] (Wj2) -- (Wj3);
\draw[->, very thick, color=blue] (Wl) to (Wj1);
\draw[->, very thick, color=blue] (Wk) to (Wj1);
\draw[->, very thick, color=blue] (Wl) to (Wj2);
\draw[->, very thick, color=blue] (Wk) to (Wj2);
\draw[->, very thick, color=blue] (Wl) to (Wj3);
\draw[->, very thick, color=blue] (Wk) to (Wj3);
\draw[->, very thick, color=blue] (Wl) to (A);
\draw[->, very thick, color=blue] (Wk) to (A);
\draw[->, very thick, color=blue, bend right] (Wj1) to (Wj3);
\draw[->, very thick, color=blue, bend right] (Wj1) to (A);
\draw[->, very thick, color=blue, bend right] (Wj2) to (A);
\draw[->, very thick, color=blue] (Wj3) to (A);
\end{scope} \end{tikzpicture}
\caption{Application of \cref{lem:project-DAG} to $V_i = W_j$ with $\pi = (W_{j_1}, W_{j_2},W_{j_3},A)$.}
\label{fig:project-DAG}
\end{figure}
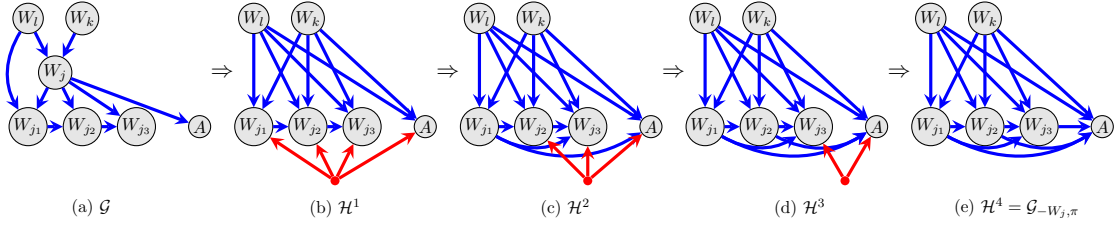

\begin{proof}[of \cref{lem:iterate-proj}]
Since $V_i \in \Vset \setminus \Vset^{\ast}(\g)$ and $\Nset(\g) = \Iset(\g) = \emptyset$, by \cref{thm:criteria}, we have either (1) $V_i \equiv W_i \in \Wset(\g) \setminus \Oset(\g)$ and $W_i$ satisfies the W-criterion or (2) $V_i \equiv M_i \in \Mset(\g) \setminus \{Y\}$ and $M_i$ satisfies the M-criterion. We now check that \cref{lem:project-DAG} can be applied.

In Case (1), we have $\Wset(\g) \cup \{A\} \supset \Ch(W_i, \g)$, for which $\pi$ is a topological ordering by W-criterion's (ii)(a) and the fact that $A \not\mapsto \Wset$. Further, the implied ordering of $\Ch(W_i, \g) \cap \Wset(\g)$ is unique. Condition \cref{eqs:projection-pa} is implied by W-criterion's (ii)(b). 

In Case (2), we have $\Mset(\g) \supset \Ch(M_i, \g)$, for which $\pi$ is the unique topological ordering by M-criterion's (ii)(a). Condition \cref{eqs:projection-pa} is implied by M-criterion's (ii)(b). 

By \cref{lem:project-DAG}, we have $\model(\g, \Vset \setminus \{V_i\}) = \model(\g_{-V_i, \pi}, \Vset \setminus \{V_i\})$. Further, for $P \in \model(\g, \Vset)$ under \cref{assump:positivity}, $\Psi_a(P; \g)$ depends on $P$ only through $P(\Vset \setminus \{V_i\})$ because $\Psi_a(P; \g) = \E[\E\{Y \mid A=a, \Oset(\g)\}]$ and $V_i \notin \{A, Y\} \cup \Oset(\g)$. 
By \cref{lem:eif-marg} and \cref{def:irr-info}, we have $\Vset^{\ast}(\g) = \Vset^{\ast}(\g_{-V_i,\pi})$. 
\end{proof}

\subsection{Commutativity of $\g_{-V_i, \pi}$ projection}
The following result shows that projection $\g_{-V_i, \pi}$ given by \cref{def:project-DAG}, when applied to uninformative variables in $\Wset(\g)  \cup \Mset(\g)$, is commutative. 

\smallskip \begin{lemma}\label{lem:order-indep-algo}
Suppose $\g$ is a directed acyclic graph on vertex set $\Vset$ that satisfies \cref{assump:A-goes-to-Y}. Suppose $A \in \Vset$ is the treatment and $Y \in \Vset$ is the outcome. Let $V_i, V_j $ be two distinct vertices in $\{\Wset(\g) \cup \Mset(\g)\} \setminus (\Oset(\g) \cup \{Y\})$ such that $V_i$ and $V_j$ satisfy \cref{lem:W-criterion} or \cref{lem:M-criterion}.
Let $\pi_i$, $\pi_j$ be defined according to \cref{eqs:pi-W} or \cref{eqs:pi-M}, depending which criterion is applicable. It holds that 
\[ (\g_{-V_i,\pi_i})_{-V_j, \pi_j} \equiv (\g_{-V_j,\pi_j})_{-V_i, \pi_i}. \]
\end{lemma}

\begin{proof}
For simplicity let $\g_1 \equiv (\g_{-V_i,\pi_i})_{-V_j, \pi_j} $ and let $\g_2 \equiv  (\g_{-V_j,\pi_j})_{-V_i, \pi_i}$.
By construction of $\g_1, \g_2$ and \cref{lem:marg-over-NI,lem:project-DAG}, both $\g_1$ and $\g_2$ are directed acyclic graphs on $\Vset \setminus \{V_1, V_2\}$. 
Hence, for \cref{lem:order-indep-algo} to hold, we only need to show that the set of edges in $\g_1$ and $\g_2$ are identical. 

 The only edges that differ between $\g_1$ and $\g_2$ as compared to $\g$ involve vertices $V_i, V_j$, $\pa(V_i, \g), \pa(V_j, \g)$, $\pi_{i}$ and $\pi_j$.   Let $E^1$ be the set of edges in $\g_1$ and $E^2$ be the set of edges in $\g_2$. 
Without loss of generality, we will suppose that $V_i$ precedes $V_j$ in the topological ordering of $\Vset$ in $\g$. 
If  $V_i \to V_j$ is not in $\g$, then 
\[ \pa(V_i, \g_{-V_j, \pi_j}) =\pa(V_i, \g),\quad \Ch(V_i, \g_{-V_j, \pi_j}) = \Ch(V_i, \g), \]
and 
\[ \pa(V_j, \g_{-V_i, \pi_i}) =\pa(V_j, \g), \quad \Ch(V_j, \g_{-V_i, \pi_i}) =\Ch(V_j, \g).\]
Therefore, by construction of $\g_1$ and $\g_2$, $E^1= E^2$.

For the rest of the proof, we suppose that  $V_i \to V_j$ is in $\g$. Then $V_i, V_j$ are both in $\Wset(\g)$, or $V_i, V_j$ are both in $\Mset$. If $V_i, V_j \in \Wset(\g)$, let $Q \equiv \Wset(\g)$, otherwise, let $Q \equiv \Mset(\g)$. Let $\pa(V_i, \g) \cap Q = \{V_i^1, \dots , V_j^{l_1}\}$ and $\pa(V_j, \g) \cap Q =  \{V_i^1, \dots , V_j^{l_2}\}$, such that $(V_i^1, \dots , V_j^{l_1})$ and  $(V_i^1, \dots , V_j^{l_2})$ are  topologically ordered in $\g$. Additionally, suppose that $\pi_i = (V_{i_{1}}, \dots , V_{i_{r_1}})$ and $\pi_j = (V_{j_{1}}, \dots , V_{j_{r_2}})$. 

Let $l(i) \in \{1, \dots, l_2\}$ such that $V_i = V_j^{l(i)}$ and let $r(j) \in \{1, \dots r_1\}$ such that $V_j = V_{i_{r(j)}}$. Below, we will tackle the most general case, that is $1 \neq l(i) \neq l_2$ and $1 \neq r(j) \neq r_1$. The proof for special cases when $l(i) \in \{1, l_2\}$ or $r(j) \in \{1, r_1\}$ follows the same logic and drops a few of the arguments below.

Let $E$ be the set of edges in $\g$ and let $E'$ be the  set 
\begin{align*}
E' =  E \setminus \Big( & \big\{V_{i}^{l} \to V_{i_{r}}:  1\le l \le l_1, 0 \le r \le r_1 \big\}  \cup \big\{V_{i_{r'}} \to V_{i_{r''}}: 0 \le r' <r'' \le r_1 \big\} \\
&\cup \big\{V_{j}^{l} \to V_{j_{r}}: 1\le l \le l_2, 0 \le r \le r_2\big\} \cup  \big\{V_{j_{r'}} \to V_{j_{r''}}: 0 \le r' <r'' \le r_2   \big\} \Big).
\end{align*}
Then $E' \subset E^1$ and $E' \subset E^2$. Additionally, the edges $E'$ are also in $\g_{-V_j, \pi_j}$ and $\g_{-V_i, \pi_i}$ as well as in $\g_1$ and $\g_2$. In order to specify the edges in $E^1 \setminus E'$ and $E^2 \setminus E'$ we need to consider edges in $\g_{-V_j, \pi_j}$ and $\g_{-V_i, \pi_i}$.
Hence, consider the parents and children of $V_i$ and $V_j$ in   $\g_{-V_j, \pi_j}$ and $\g_{-V_i, \pi_i}$:
\begin{align}
&\pa(V_i, \g_{-V_j, \pi_j}) \cap Q =\pa(V_i, \g) \cap Q = \{V_{i_1},\dots, V_{i_{r_1}}\}, \label{eq:pavi}\\
&\Ch(V_i, \g_{-V_j, \pi_j}) \cap (Q \cup \{A\}) =  \{V_{i_1}, \dots , V_{i_{r^{j}-1}},  V_{i_{r^{j}+1}}, \dots, V_{i_{r_1}}\} \cup \{V_{j_1}, \dots, V_{j_{r_2}}\},\label{eq:chvi1}\\
&\pa(V_j, \g_{-V_i, \pi_i}) \cap Q = \{V_j^{1}, \dots, V_j^{l(i) -1}, V_j^{l(i) + 1}, \dots, V_j^{l_2}\} \cup \{V_{i}^1, \dots, V_{i}^{l_1}\} \cup  \{V_{i_1}, \dots, V_{i_{r(j) -1}}\},\label{eq:pavj1}\\
&\Ch(V_j, \g_{-V_i, \pi_i})  \cap (Q \cup \{A\}) = \{V_{j_1}, \dots ,V_{j_{r_2}}\} \cup \{V_{i_{r(j) +1}}, \dots, V_{i_{r_1}}\}. \label{eq:chvj1}
\end{align}

Since $V_i \in Q$ and $V_j \in \Ch(V_i,\g) \cap Q$ and since $V_i$ satisfies \cref{lem:W-criterion} or \cref{lem:M-criterion}, it holds that 
\begin{align}\label{eq:prop-pa}
\pa(V_j,\g) \cap Q \subseteq (\pa(V_i, \g) \cap Q) \cup \{V_i, V_{i_1}, \dots, V_{i_{r(j)-1}}\}.
\end{align}
Additionally, by \cref{eq:pavj1} 
\begin{align*}
\pa(V_j, \g_{-V_i, \pi_i}) \cap Q = \left[(\pa(V_j,\g) \cap Q) \setminus\{V_i\} \right] \cup   (\pa(V_i, \g) \cap Q) \cup \{V_i, V_{i_1}, \dots, V_{i_{r(j) -1}}\},
\end{align*}
and \cref{eq:pavj1} can be rewritten as 
\begin{align}\label{eq:pavj}
&\pa(V_j, \g_{-V_i, \pi_i}) \cap Q = \{V_{i}^1, \dots, V_{i}^{l_1}, V_{i_1}, \dots, V_{i_{r(j) -1}}\},
\end{align}
where all the vertices are listed in a topological ordering consistent with $\g$ and $\g_{-V_i, \pi_i}$.

Next we consider how to simplify and topologically order in $\g$ the set $\{V_{j_1}, \dots ,V_{j_{r_2}}\} \cup \{V_{i_{r(j) +1}}, \dots, V_{i_{r_1}}\}$ which is contained in both $\Ch(V_j, \g_{-V_i, \pi_i}) \cap (Q \cup \{ A \})$ and $\Ch(V_i, \g_{-V_j, \pi_j}) \cap (Q \cup \{A\})$; see \cref{eq:chvi1,eq:chvj1}.

Let 
\[ \{V_{{ij}_1},  \dots V_{{ij}_{r_3}}\} =  (\{V_{j_1}, \dots ,V_{j_{r_2}}\} \cap Q) \setminus \{V_{i_{r(j) +1}}, \dots, V_{i_{r_1}}\} \]
be a vertex set such that $(V_{{ij}_1},  \dots V_{{ij}_{r_3}} )$ is topologically ordered in $\g$.  Since the topological ordering of $ (V_{j_1}, \dots ,V_{j_{r_2}}) $ is unique by \cref{lem:iterate-proj}, the ordering  $(V_{{ij}_1},  \dots V_{{ij}_{r_3}} )$  does not conflict with  the ordering $(V_{j_1}, \dots ,V_{j_{r_2}})$. Additionally, by construction of $\g_{-V_i,\pi_i}$ and $\g_{-V_j,\pi_j}$ from $\g$, $(V_{{ij}_1},  \dots V_{{ij}_{r_3}} )$ is also  topologically ordered in $\g_{-V_i, \pi_i}$ and $\g_{-V_j, \pi_j}$.

The proof is split into two cases.

\begin{enumerate}[(a)]
\item $V_{i_{r_1}} \neq A \neq V_{j_{r_2}}$.

\medskip We have $V_{i_{r_1}} \neq A$, $A \notin \{V_{{ij}_1},  \dots V_{{ij}_{r_3}} \}$ and both $(V_{i_{r(j) +1}}, \dots, V_{i_{r_1}})$ and  $(V_{{ij}_1},  \dots V_{{ij}_{r_3}})$ are  topologically ordered in $\g$, $\g_{-V_i, \pi_i}$, and $\g_{-V_j, \pi_j}$. Hence, for $(V_{i_{r(j) +1}}, \dots, V_{i_{r_1}},V_{{ij}_1},  \dots V_{{ij}_{r_3}})$ to be topologically ordered in $\g$, as well as in $\g_{-V_i, \pi_i}$, and $\g_{-V_j, \pi_j}$, it is enough to prove that a vertex $B_1 \in \{V_{{ij}_1},  \dots V_{{ij}_{r_3}} \} = ( \{V_{j_1}, \dots ,V_{j_{r_2}}\} \cap Q) \setminus \{V_{i_{r(j) +1}}, \dots, V_{i_{r_1}}\}$ cannot be an ancestor of a vertex $B_2 \in \{V_{i_{r(j) +1}}, \dots, V_{i_{r_1}}\} \cap Q$  in $\g$.

Suppose for a contradiction that such a pair of vertices exists in $\g$ and  choose $B_1$ and $B_2$ to be a pair of such vertices with a shortest causal path $p$ from $B_1$ to $ B_2$ in $\g$. By choice of $B_1$ and $B_2$, no other vertex on $p$ is in  $\{V_{i_{r(j) +1}}, \dots, V_{i_{r_1}}\} \cap Q$. Since $B_2 \in \Ch(V_i, \g) \cap Q$, by \cref{lem:W-criterion,lem:M-criterion}, $\pa(B_2, \g) \subseteq \pa(V_i, \g) \cup \{V_i, V_{i_1}, \dots , V_{i_{r(j)}}\}$. This, in turn, implies that $B_1$ is an ancestor of a vertex in $\pa(V_i, \g) \cup \{V_i, V_{i_1}, \dots , V_{i_{r(j)}}\}$ through $p$, which together with $B_1 \in \Ch(V_j, \g) \cap Q$ and $V_i \to V_j$ in $\g$, implies that a directed cycle is present in $\g$, which is a contradiction. 

Since $A \neq V_{i_{r_1}}$ and $A \neq V_{j_{r_2}}$, we also have that the vertex sets on the right-hand-side of equations \eqref{eq:chvi} and \eqref{eq:chvj} below are listed in a topological ordering consistent with $\g$, $\g_{-V_i, \pi_i}$,  and $\g_{-V_j, \pi_j}$.

\begin{align}
&\Ch(V_i, \g_{-V_j, \pi_j}) \cap (Q \cup \{A\}) =  \{V_{i_1}, \dots , V_{i_{r(j)-1}}, V_{i_{r(j) +1}}, \dots, V_{i_{r_1}},V_{{ij}_1},  \dots V_{{ij}_{r_3}}\},\label{eq:chvi}\\
&\Ch(V_j, \g_{-V_i, \pi_i})  \cap (Q \cup \{A\}) = \{V_{i_{r(j) +1}}, \dots, V_{i_{r_1}},V_{{ij}_1},  \dots V_{{ij}_{r_3}}\}. \label{eq:chvj}
\end{align}

\medskip Now consider the set of edges $E^1$, which can be decomposed as $E^1 = E' \cup E^1_i \cup E^2_j$. Edges $E^1_i$ given by 
\begin{multline} \label{eq:e1i}
E^{1}_i  = \big\{V_{i}^{l} \to V_{i_{r}}:  1\le l \le l_1, 1 \le r \le r_1, r \neq r(j) \big\} \\
\cup  \big\{V_{i_{r'}} \to V_{i_{r''}}: 1 \le r' <r'' \le r_1, r' \neq r(j) \neq r'' \big\},
\end{multline}
are the edges in $\g_1$ that are added by transforming $\g$ into $\g_{-V_i, \pi_i}$. Meanwhile, edges $E^2_j$ given by 
\begin{multline} \label{eq:e2j}
E^{2}_j = \big\{V_{i}^{l} \to V_{{ij}_{r}}: 1\le l \le l_1, 1 \le r \le r_3\big\} \\
\cup  \big\{V_{{i}_{r'}} \to V_{{ij}_{r''}}: 1\le r' < r_1, 1 \le r'' \le r_3, r' \neq r(j) \big\} \\
\cup \big\{V_{{ij}_{r'}} \to V_{{ij}_{r''}}: 1 \le r' < r'' \le r_3 \big\}.
\end{multline}
are the edges added to $\g_1$, by transforming $\g_{-V_i, \pi_i}$ into $\g_1 = (\g_{-V_i,\pi_i})_{-V_j, \pi_j}$. Equation \cref{eq:e2j} uses \cref{eq:pavj,eq:chvj}. 

By \cref{eq:e1i,eq:e2j}, we have that $E^{1}_i \cup E^{2}_j = E^1 \setminus E'$ is equal to
\begin{multline}\label{eq:all1}
E^{1}_i   \cup E^{2}_j =\big\{V_{i}^{l} \to V_{i_{r}}:  1\le l \le l_1, 1 \le r \le r_1, r \neq r(j) \big\}  \cup \big\{V_{i}^{l} \to V_{{ij}_{r}}: 1\le l \le l_1, 1 \le r \le r_3\big\} \\
\cup  \big\{V_{i_{r'}} \to V_{i_{r''}}: 1 \le r' <r'' \le r_1, r' \neq r(j) \neq r'' \big\} \\
\cup  \big\{V_{{i}_{r'}} \to V_{{ij}_{r''}}: 1\le r' < r_1, 1 \le r'' \le r_3, r' \neq r(j) \big\} \\
\cup \big\{V_{{ij}_{r'}} \to V_{{ij}_{r''}}: 1 \le r' < r'' \le r_3 \big\}.
\end{multline}

\medskip Next, consider the set of edges $E^2$, which can be decomposed as $E^2 = E' \cup E^1_j \cup E^2_i$. Edges $E^1_j$ given by 
\begin{equation} \label{eq:e1j}
E^{1}_j = \big\{V_{j}^{l} \to V_{j_{r}}: 1\le l \le l_2, 1 \le r \le r_2, l \neq l(i) \big\} \cup \big\{V_{j_{r'}} \to V_{j_{r''}}: 1 \le r' <r'' \le r_1   \big\},
\end{equation}
are the edges in $\g_2$ that are added by transforming $\g$ into $\g_{-V_j, \pi_j}$. Meanwhile, edges $E^2_i$ given by 
\begin{multline} \label{eq:e2i}
E^{2}_i =  \big\{V_{i}^{l} \to V_{i_{r}}:  1\le l \le l_1, 1 \le r \le r_1, r \neq r(j) \big\} \cup \big\{V_{i}^{l} \to V_{{ij}_{r}}:  1\le l \le l_1, 1 \le r \le r_3 \big\} \\
\cup \big\{V_{i_{r'}} \to V_{{i}_{r''}}: 1 \le r' < r'' \le r_1, r' \neq r(j)  \big\} \\
\cup \big\{V_{i_{r'}} \to V_{{ij}_{r''}}: 1 \le r' \le r_1, 1 \le r'' \le r_3, r' \neq r(j)  \big\} \\
\cup \big\{V_{{ij}_{r'}} \to V_{{ij}_{r''}}: 1 \le r' < r'' \le r_3 \big\},
\end{multline}
are the edges added to $\g_2$, by transforming $\g_{-V_j, \pi_j}$ into $\g_2 = (\g_{-V_j,\pi_j})_{-V_i, \pi_i}$. Equation \cref{eq:e2i} uses \cref{eq:pavi,eq:chvi}.

Based on \cref{eq:e1j,eq:e2i}, we have that $E^{1}_j \cup E^{2}_i = E^2 \setminus E'$ is equal to
\begin{multline} \label{eq:all2}
E^{1}_j \cup E^{2}_i  =  \big\{V_{j}^{l} \to V_{j_{r}}: 1\le l \le l_2, 1 \le r \le r_2, l \neq l(i) \big\} \cup \big\{V_{j_{r'}} \to V_{j_{r''}}: 1 \le r' <r'' \le r_1   \big\} \\
\cup\big\{V_{i}^{l} \to V_{i_{r}}:  1\le l \le l_1, 1 \le r \le r_1, r \neq r(j) \big\} \cup \big\{V_{i}^{l} \to V_{{ij}_{r}}:  1\le l \le l_1, 1 \le r \le r_3 \big\} \\
\cup  \big\{V_{i_{r'}} \to V_{{i}_{r''}}: 1 \le r' < r'' \le r_1, r' \neq r(j) \neq r'' \big\} \\
\cup  \big\{V_{i_{r'}} \to V_{{ij}_{r''}}: 1 \le r' \le r_1, 1 \le r'' \le r_3, r' \neq r(j)  \big\}\\
\cup \big\{V_{{ij}_{r'}} \to V_{{ij}_{r''}}: 1 \le r' < r'' \le r_3 \big\}.
\end{multline}
By \cref{eq:prop-pa}, we have $\{V_j^{1}, \dots, V_j^{l_2}\} \subseteq \{V_i^{1}, \dots , V_i^{l_1}\} \cup \{V_i, V_{i_1}, \dots, V_{i_{r(j)-1}}\}$. Using this property, together with the fact that 
\[\{V_{{ij}_1}, \dots , V_{{ij}_{r_3}}\} =  \left[\{V_{j_1}, \dots ,V_{j_{r_2}}\} \cap (Q \cup \{A\})\right] \setminus \{V_{i_{r(j) +1}}, \dots, V_{i_{r_1}}\},\]
equation \cref{eq:all2} simplifies to
\begin{multline} \label{eq:all22}
E^{1}_j \cup E^{2}_i  =   \big\{V_{i}^{l} \to V_{i_{r}}:  1\le l \le l_1, 1 \le r \le r_1, r \neq r(j) \big\} \cup \big\{V_{i}^{l} \to V_{{ij}_{r}}:  1\le l \le l_1, 1 \le r \le r_3 \big\}  \\
\cup \big\{V_{i_{r'}} \to V_{{i}_{r''}}: 1 \le r' < r'' \le r_1, r' \neq r(j) \neq r''  \big\} \\
\cup \big\{V_{i_{r'}} \to V_{{ij}_{r''}}: 1 \le r' \le r_1, 1 \le r'' \le r_3, r' \neq r(j) \big\}\\
\cup   \big\{V_{{ij}_{r'}} \to V_{{ij}_{r''}}: 1 \le r' < r'' \le r_3 \big\}.
\end{multline}
By \cref{eq:all1,eq:all22}, we have $E^{1}_j \cup E^{2}_i=E^{1}_i \cup E^{2}_j $ and therefore $E_1 = E_2$.

\medskip \item $A = V_{i_{r_1}}$ or  $ A = V_{j_{r_2}}$.
In this case, the same argument as above can be used to show that $(V_{i_{r(j) +1}}, \dots, V_{i_{r_1-1}},V_{{ij}_1},  \dots V_{{ij}_{r_3}})$
 is topologically ordered in $\g$, $\g_{-V_i, \pi_i}$ and $\g_{-V_j, \pi_j}$. Then, using the 
 fact that $V_{i_{r_1}} = A$ or $V_{j_{r_2}} = A$ we know that $V_i, V_j \in \Wset(\g)$. 
 Hence, we know $(V_{i_{r(j) +1}}, \dots, V_{i_{r_1 -1}},V_{{ij}_1},  \dots V_{{ij}_{r_3}}, A)$ is a 
 topological ordering in $\g$. Since  $V_i \in \Wset(\g_{-V_j, \pi_j})$ and $V_j \in \Wset(\g_{-V_i, \pi_i})$, by \cref{def:project-DAG,lem:iterate-proj}, $(V_{i_{r(j) +1}}, \dots, V_{i_{r_1 -1}},V_{{ij}_1},  \dots V_{{ij}_{r_3}}, A)$ is also a topological ordering in $\g_{-V_i, \pi_i}$  and $\g_{-V_j, \pi_j}$.

Then let $V_{{ij}_{r_3 +1}} \equiv A$. Equations \cref{eq:chvi,eq:chvj} in this case become
\begin{align*}
&\Ch(V_i, \g_{-V_j, \pi_j}) \cap (Q \cup \{A\}) =  \{V_{i_1}, \dots , V_{i_{r(j)-1}}, V_{i_{r(j) +1}}, \dots, V_{i_{r_1}},V_{{ij}_1},  \dots V_{{ij}_{r_3 +1}}\},\\
&\Ch(V_j, \g_{-V_i, \pi_i})  \cap (Q \cup \{A\}) = \{V_{i_{r(j) +1}}, \dots, V_{i_{r_1}},V_{{ij}_1},  \dots V_{{ij}_{r_3+1}}\}.
\end{align*}
The rest of the argument follows in exactly the same fashion as above, except that we replace $r_3$ with $r_3 +1$ in the definition of  $E_2^{j}$ and $E_2^i$.

\end{enumerate}
\end{proof}

\subsection{Proof of \cref{thm:algo}} \label{sec:app-proof-thm-algo}
\begin{proof}
\begin{enumerate}[(i)]
\item By construction, graph $\g^{\ast}$ has vertex set $\Vset^{\ast}$. By \cref{lem:marg-over-NI,lem:project-DAG}, $\g^{\ast}$ is a directed acyclic graph. 
\item This is a consequence of \cref{lem:order-indep-algo}.
\item Let graph $\g^{0}$ be graph $\g^{\ast}$ after the executing the second line of \cref{alg:graph}. By \cref{lem:marg-over-NI}, for $\Vset^0 \equiv \Vset \setminus (\Nset(\g) \cup \Iset(\g))$, we have 
\begin{equation} \label{eqs:proof-thm-algo}
\model(\g, \Vset^0) = \model(\g^0, \Vset^0) 
\end{equation}
If $\Vset^0 = \Vset^{\ast}(\g)$, then the algorithm returns $\g^{\ast} = \g^0$ and there is nothing more to prove. Otherwise, suppose $V_i$ is the next uninformative variable that the algorithm visits. By \cref{lem:eif-marg} and the fact that $\Vset^{\ast}(\cdot)$ is the irreducible informative set according to \cref{thm:criteria}, it is easy to see that 
\[ \Vset^{\ast}(\g) = \Vset^{\ast}(\g^0). \]
It then follows that $V_i \in \Vset^{0} \setminus \Vset^{\ast}(\g^0)$. By \cref{lem:iterate-proj}, with $\g^1 \equiv \g_{-v, \pi}$ and $\Vset^{1} \equiv \Vset^{0} \setminus \{V_i\}$, we have
\[ \model(\g^1, \Vset^1) = \model(\g^0, \Vset^1) \stackrel{(a)}{=} \model(\g, \Vset^1) \quad \text{and} \quad \Vset^{\ast}(\g^1) = \Vset^{\ast}(\g^0) = \Vset^{\ast}(\g) \]
where equality (a) follows from \cref{eqs:proof-thm-algo}. 
The proof is completed by iterating this argument for the remaining uninformative variables. 
\item Since $\{A,Y\} \cup \Oset(\g) \subseteq \Vset^{\ast}$, for every $P \in \model(\g, \Vset)$, we have
\[ \Psi_a(P; \g) \stackrel{(b)}{=} \E[\E\{Y \mid A=a, \Oset(\g)\}] \stackrel{(c)}{=} \Psi_a(P^{\ast}; \g^{\ast}), \]
where (b) holds because $\Oset(\g)$ is an adjustment set in graph $\g$. Note that (c) also holds because $P^{\ast} \in \model(\g^{\ast}, \Vset^{\ast})$ by property (iii) and $\Oset(\g^{\ast}) = \Oset(\g)$ is an adjustment set in graph $\g^{\ast}$. By $\Psi_a(P; \g^{\ast}) \equiv \Psi_a(P^{\ast}; \g^{\ast})$, the proof is complete. 

\item This follows from applying \cref{lem:eif-marg} with $V' = \Vset^{\ast}$, for which the conditions are fulfilled by \cref{thm:criteria} and property (iii).

\item By property (iv), $\Psi_{a}(\cdot; \g^{\ast}): \model_0(\Vset) \rightarrow \mathbb{R}$ is an identifying formula for the g-functional $\Psi_a(P; \g)$ defined on $\model(\g, \Vset)$. Now we show that the identifying formula is efficient and irreducible. 

\smallskip First, we show that it is efficient. For every $P \in \model_0(\Vset)$, recall that $\Psi_a(P; \g^{\ast}) \equiv \Psi_a(P^{\ast}; \g^{\ast})$. For any $P \in \model(\g, \Vset)$, we have the corresponding $P^{\ast} \in \model(\g^{\ast}, \Vset^{\ast})$ and 
\[\Psi_{a,P,\text{NP}}^{1}(\Vset; \g^{\ast}) = \Psi_{a,P^{\ast},\text{NP}}^{1}(\Vset^{\ast}; \g^{\ast}) \stackrel{(a)}{=} \Psi_{a,P^{\ast},\text{eff}}^{1}(\Vset^{\ast}; \g^{\ast}) \stackrel{(b)}{=} \Psi_{a,P,\text{eff}}^{1}(\Vset; \g),  \text{ $P$-almost everywhere},\]
where equality (a) follows from applying \cref{lem:effic-of-g-formula} to graph $\g^{\ast}$, equality (b) follows from property (v). Then, by \cref{def:eff-formula}, $\Psi_{a}(\cdot; \g^{\ast}): \model_0(\Vset) \rightarrow \mathbb{R}$ is an efficient identifying formula. 

\smallskip Then, to see that it is irreducible, observe that $\Psi_a(P; \g^{\ast}) \equiv \Psi_a(P^{\ast}; \g^{\ast})$, where $P^{\ast}$ is the $\Vset^{\ast}$ margin of $P$. Set $\Vset^{\ast}$ is irreducible informative by \cref{thm:criteria}.

\end{enumerate}
\end{proof}

\subsection{Proof of \cref{cor:asymp-equiv-g}}
\begin{proof}
Estimators $\Psi_a(\mathbb{P}_n; \g^{\ast})$ and $\Psi_a(\mathbb{P}_n; \g)$ are regular asymptotically linear. We have
\begin{equation*}
\begin{split}
\Psi_a(\mathbb{P}_n; \g^{\ast}) - \Psi_a(P; \g)  &= \frac{1}{n} \sum_{i=1}^n \Psi_{a,P,\text{NP}}^{1}(\Vset^i; \g^{\ast}) + o_p(n^{-1/2}), \\
\Psi_a(\mathbb{P}_n; \g) - \Psi_a(P; \g)  &= \frac{1}{n} \sum_{i=1}^n \Psi_{a,P,\text{NP}}^{1}(\Vset^i; \g) + o_p(n^{-1/2}), \\
\end{split}
\end{equation*}
where $V^i$ is the $i$-th observation of $\Vset \sim P$. Under $P \in \model(\g; \Vset)$, by \cref{lem:effic-of-g-formula} and \cref{thm:algo}(vi), we have $\Psi_{a,P,\text{NP}}^{1}(\Vset; \g^{\ast}) = \Psi_{a,P,\text{NP}}^{1}(\Vset; \g) = \Psi_{a,P,\text{eff}}^{1}(\Vset; \g)$. The result then follows. 
\end{proof}

\section{Efficient identifying formula} \label{sec:app-eff-formula}
For elements of semiparametric efficiency theory, see \cref{sec:eif} and references therein.

\subsection{Proof of \cref{lem:effic-of-g-formula}}
\begin{proof}[of \cref{lem:effic-of-g-formula}]
It is a consequence of the following general result.
\end{proof}

\smallskip \begin{lemma} \label{lem:app-effic-of-g-general}
Given a directed acyclic graph $\mathcal{G}$ with vertex $\{ V_{1},\dots,V_{J}\} $. Suppose the formula $\varphi: \model_0(\Vset) \rightarrow \mathbb{R}$ is a regular functional on $ \model_{0}(\Vset)$ such that 
\[ P, P' \in \model_0(\Vset): p(v_{j} \mid \Pa(v_j,\g)) = p'(v_{j} \mid \Pa(v_j,\g)), \quad j=1,\dots,J, \]
implies $\varphi(P) =\varphi(P')$. 
Then, we have $\varphi _{P,\text{NP}}^{1}(\Vset) =\varphi _{P,\text{eff}}^{1}(\Vset) $ holds $P$-almost everywhere for every $P\in \model(\g, \Vset)$, where $\varphi_{P,\text{eff}}^{1}(\Vset)$ is the efficient influence function of $\varphi(P)$ at $P$ with respect to $\model(\g, \Vset)$.
\end{lemma}

\smallskip \begin{proof}
Let $p_{j}$ denote the conditional density of $V_{j}$ given $\Pa(V_{j},\g)$.
Because by assumption $\varphi \left( P\right) $ depends on $P$ only through $p_{1},\dots,p_{J}$, we can write 
\[ \varphi \left( P\right) =\nu ( p_{1},\dots,p_{J}). \]
Consider a regular submodel $P_{t}$ for $t \in [0, \varepsilon]$ for $\varepsilon > 0$ with $P_{t=0}=P$. Denote the score at $t=0$ with $S$. Also, let the $p_{j,t}$ be the conditional density of $V_{j}$ given $\Pa(V_j, \g)$ associated with $P_t$. We have
\begin{equation*}
\begin{split}
\left. \frac{d}{dt}\varphi \left( P_{t}\right) \right\vert _{t=0} &=\sum_{j=1}^{J}\left. \frac{d}{dt}\nu \left(p_{1},\dots,p_{j,t},\dots,p_{J}\right) \right\vert _{t=0} \\
&=\sum_{j=1}^{J} \E_{P}\left[ \nu _{j,P}^{1}\left( V\right) S\right] = \E_{P}\left[ \left\{ \sum_{j=1}^{J}\nu _{j,P}^{1}\left( V\right) \right\} S\right].
\end{split}
\end{equation*}
Thus, we have
\[ \varphi _{P,\text{NP}}^{1}(\Vset) = \sum_{j=1}^{J}\nu_{j,P}^{1}\left( V\right). \]
To prove the result it then suffices to show that for $P\in \mathcal{M}\left( \mathcal{G},V\right) ,$ $\sum_{j=1}^{J}\nu _{j,P}^{1}\left( V\right) $ is an element of the tangent space of $\mathcal{M }\left( \mathcal{G},V\right) $ at $P$, i.e., $\Lambda \left( P\right) =\oplus _{j=1}^{J}\Lambda _{j}\left( P\right) $ with orthogonal subspaces 
\[ \Lambda _{j}\left( P\right) =\left\{ Q_j \equiv q_{j}(V_{j},\Pa(V_{j},\g)): \E_{P}\left( Q_{j} \mid \Pa(V_{j},\g) \right)=0 \right\}, \quad j=1,\dots,J. \]
We will show this by proving 
\[ \nu_{j,P}^{1}\left( V\right) \in \Lambda_{j}( P), \quad j=1,\dots,J. \]

To do so, for $W_{j} \equiv V \setminus \left[ \left\{ V_{j}\right\} \cup \Pa( V_{j},\g) \right]$, 
decompose
\[ p_{t}(V) = p_{t}(\Pa(V_{j}, \g)) \, p_{t}(V_{j} \mid \Pa(V_{j}, \g)) \, p_{t} (W_{j} \mid V_j, \Pa(V_{j}, \g)). \]
Then, we can write 
\[ S=s_{1}(\Pa(V_{j}, \g))) + s_{2}( V_{j} \mid \Pa(V_{j}, \g)) + s_{3}( W_{j} \mid V_j, \Pa(V_{j}, \g)),\]
where $s_{1}\left( \Pa(V_{j}, \g) \right) $ is the score of the submodel 
\begin{equation} \label{eqs:app-s1} 
t \mapsto p_{t}(\Pa(V_{j}, \g)) \, p(V_{j} \mid \Pa(V_{j}, \g)) \, p (W_{j} \mid V_j, \Pa(V_{j}, \g)), 
\end{equation}
$s_{2}\left(V_{j} \mid \Pa(V_{j}, \g) \right)$ is the score of the submodel 
\begin{equation}  \label{eqs:app-s2} 
t \mapsto p(\Pa(V_{j}, \g)) \, p_t(V_{j} \mid \Pa(V_{j}, \g)) \, p (W_{j} \mid V_j, \Pa(V_{j}, \g)), \end{equation}
and $s_{3}\left(W_j \mid V_{j}, \Pa(V_{j}, \g) \right)$ is the score of the submodel 
\begin{equation} \label{eqs:app-s3} 
t \mapsto p(\Pa(V_{j}, \g)) \, p(V_{j} \mid \Pa(V_{j}, \g)) \, p_t (W_{j} \mid V_j, \Pa(V_{j}, \g)).
\end{equation}

But since $\nu(p_{1},\dots,p_{j,t},\dots,p_{J})$ remains constant under the submodels \cref{eqs:app-s2,eqs:app-s3}, we have 
\[ \E_{P}\left\{ \nu_{j}^{1}(V)s_{1}\left( \Pa(V_{j}, \g) \right) \right\} =  \E_{P}\left\{ \nu_{j}^{1}(V) s_{3}\left(W_j \mid V_{j}, \Pa(V_{j}, \g) \right)  \right\} =0. \]
Furthermore, $s_{1}$ and $s_{3}$ are uncorrelated under $P$ with the elements of $\Lambda
_{j}\left( P\right) .$ So, it also holds that  
\[ \E_{P}\left\{ \nu_{j}^{1}(V)s_{1}\left( \Pa(V_{j}, \g) \right) \right\} = \E_{P}\left\{ \Pi\left[ \nu_{j}^{1}(V) \mid \Lambda_j(P) \right] s_{1}\left( \Pa(V_{j}, \g) \right) \right\} = 0\]
and 
\[ \E_{P}\left\{ \nu_{j}^{1}(V) s_{3}\left(W_j \mid V_{j}, \Pa(V_{j}, \g) \right)  \right\} = \E_{P}\left\{ \Pi\left[\nu_{j}^{1}(V) \mid \Lambda_j(P) \right]s_{3}\left(W_j \mid V_{j}, \Pa(V_{j}, \g) \right)  \right\}  = 0,\]
where $\Pi \left[ \cdot \mid \Lambda _{j}( P)\right]$ is the projection operator in $L_{2}(P) $ onto $\Lambda_{j}(P)$. In addition, since $s_{2}\left(V_{j} \mid \Pa(V_{j}, \g) \right) \in \Lambda_j(P)$, we also have 
\[ \E_{P}\left\{ \nu_{j}^{1}(V) s_{2}\left(V_{j} \mid \Pa(V_{j}, \g) \right) \right\} = \E_{P}\left\{  \Pi\left[\nu_{j}^{1}(V) \mid \Lambda_j(P) \right]s_{2}\left(V_{j} \mid \Pa(V_{j}, \g) \right) \right\}. \]
Therefore, we have
\begin{eqnarray*}
\E_{P}\left\{ \nu _{j}^{1}\left( V\right) S\right\}  &=&\E_{P}\left\{ \nu
_{j}^{1}\left( V\right) s_{1}\left( \Pa(V_{j}, \g) \right) \right\} + \E_{P}\left\{ \nu_{j}^{1}\left( V\right) s_{2}\left(V_{j} \mid \Pa(V_{j}, \g) \right) \right\}  \\
&&+ \E_{P}\left\{ \nu _{j}^{1}\left( V\right) s_{3}\left(W_j \mid V_{j}, \Pa(V_{j}, \g) \right)\right\}  \\
&=&\E_{P}\left\{ \Pi \left[ \left. \nu_{j}^{1}\left( V\right) \right\vert
\Lambda _{j}\left( P\right) \right] s_{1}\left( \Pa(V_{j}, \g) \right) \right\} +\E_{P}\left\{ \Pi \left[ \left. \nu
_{j}^{1}\left( V\right) \right\vert \Lambda _{j}\left( P\right) \right]
s_{2}\left(V_{j} \mid \Pa(V_{j}, \g) \right) \right\}
\\
&&+\E_{P}\left\{ \Pi \left[ \left. \nu _{j}^{1}\left( V\right) \right\vert
\Lambda _{j}\left( P\right) \right] s_{3}\left(W_j \mid V_{j}, \Pa(V_{j}, \g) \right) \right\}  \\
&=&\E_{P}\left\{ \Pi \left[ \left. \nu _{j}^{1}\left( V\right) \right\vert
\Lambda _{j}\left( P\right) \right] S\right\},
\end{eqnarray*}
and hence
\[\E_{P}\left( \left\{ \nu _{j}^{1}\left( V\right) -\Pi \left[ \left. \nu_{j}^{1}\left( V\right) \right\vert \Lambda _{j}\left( P\right) \right] \right\} S\right) = 0.\]
The above holds for all scores $S$ in the unrestricted model $\model_0(\Vset)$, i.e., for all mean-zero functions in $L_2(P)$. In particular, choosing $S = \nu _{j}^{1}\left( V\right) -\Pi \left[ \left. \nu_{j}^{1}\left( V\right) \right\vert \Lambda _{j}\left( P\right) \right]$, we have
\[\E_{P} \left\{ \nu _{j}^{1}\left( V\right) -\Pi \left[ \left. \nu_{j}^{1}\left( V\right) \right\vert \Lambda _{j}\left( P\right) \right] \right\}^2 = 0, \]
from which we deduce $\nu_{j,P}^{1}\left( V\right) \in \Lambda_{j}( P)$ for $j=1,\dots,J$. This concludes the proof.
\end{proof}

\section{Informative variables and their characterization} \label{sec:app-uninfo}
\subsection{Extension to average treatment effects} \label{sec:app-extension-effect}
\begin{lemma} \label{lem:app-extension-effect}
Let $\g$ be a directed acyclic graph on vertex set $\Vset$ that satisfies \cref{assump:A-goes-to-Y}. Suppose $A \in \Vset$ is a discrete treatment and $Y \in \Vset$ is the outcome of interest. Let $(a_1, \dots, a_J)$ be $J$ distinct treatment levels and let $(c_1, \dots, c_J)$ be non-zero constants. Consider the functional
\[ \Psi_{c}(P; \g) \equiv \sum_{j=1}^J c_j \Psi_{a_j}(P; \g), \]
where $\Psi_{a_j}(P; \g)$ is the g-functional for treatment level $a_j$. Then, $\Vset^{\ast}(\g)$ given by \cref{thm:criteria} is the unique set of irreducible informative variables for estimating $\Psi_c(P; \g)$ under $\model(\g, \Vset)$. 
\end{lemma}
\begin{proof}
By \cref{lem:EIF} and linearity of influence functions, for $P \in \model(P; \g)$, the efficient influence function for $\Psi_{c}(P; \g)$ with respect to $\model(\g, \Vset)$ is given by 
\begin{multline*}
\Psi_{c,P,\text{eff}}^{1}(\Vset; \g) = \sum_{j=1}^{J}\big[ \E\left\{ b_{c,P}(\Oset) \mid W_{j},\Pa(W_{j}, \g)\right\} -\E\left\{ b_{c,P}(\Oset) \mid \Pa(W_{j}, \g) \right\} \big]  \\
 +\sum_{k=1}^{K} \big[ \E\left\{ T_{c,P} \mid M_{k},\Pa(M_{k}, \g)\right\} -
\E\left\{ T_{c,P}\mid \Pa(M_{k}, \g)\right\} \big],
\end{multline*}
where 
\[ b_{c,P} \equiv \sum_{j=1}^J c_j b_{a_j,P}, \quad T_{c,P} \equiv \sum_{j=1}^J c_j T_{a_j,P}. \]
Inspecting the proof of \cref{thm:criteria}, we see that conditions (i) and (ii) still apply. Condition (iii) can be established by choosing the appropriate laws in the model similarly. 
\end{proof}

\subsection{Proof of \cref{lem:info}}
\begin{proof}
First, we show $\Vset^{\ast}$ is irreducible informative. Condition (i) of \cref{def:irr-info} is fulfilled by (i) and (ii). We claim condition (ii) is implied by (iii). Suppose not and there must exist $V_j \in \Vset^{\ast}$ such that $(\Vset \setminus \Vset^{\ast}) \cup \{V_j\}$ is uninformative, which by \cref{def:uninfo} implies that $\gamma_{P,\text{eff}}^{1}(\Vset)$ is a function of $\Vset^{\ast} \setminus \{V_j\}$, contradicting (iii). 

Now we show $\Vset^{\ast}$ is the only set that is irreducible informative. Suppose set $\Vset^{\ast\ast}$ is also irreducible informative. By (iii), we must have $\Vset^{\ast} \subseteq \Vset^{\ast \ast}$. Further, this inclusion cannot be strict by applying condition (ii) of \cref{def:irr-info} to $\Vset^{\ast \ast}$. 
\end{proof}

\subsection{Proof of \cref{lem:N-I-uninformative}}
\begin{proof}
First, according vertex sets defined in \cref{sec:sets}, for any $V_i \in \Nset(\g) \cup \Iset(\g)$, it holds that $V_i \notin \Pa(\Wset(\g) \cup \Mset(\g), \g)$. Then, by \cref{lem:EIF}, $\Psi_{a,P,\text{eff}}^{1}(\Vset; \g)$ does not depend on $\Nset(\g) \cup \Iset(\g)$ with probability one for every $P \in \model(\g, \Vset)$. Second, for $P \in \model(\g, \Vset)$, observe that $\Psi_a(P; \g) = \Psi_{a,\Oset}^{\text{ADJ}}(P; \g)$, which implies that $\Psi_a(P; \g)$ depends on $P$ only through $P(\Vset \setminus \Nset(\g) \setminus \Iset(\g))$. By \cref{lem:eif-marg,def:uninfo}, it follows that $\Nset(\g) \cup \Iset(\g)$ is uninformative for estimating $\Psi_a(P; \g)$ under $\model(\g, \Vset)$. 
\end{proof}

\subsection{Proof of \cref{prop:redundant}}
\begin{proof} 
Set $U$ in \cref{def:uninfo} depends on $\g$ only through $\model( \g,\Vset)$ and $\Psi (P;\g):\model(\g,\Vset)\rightarrow \mathbb{R}$. The result then follows from \cref{lem:N-I-uninformative} and \cref{def:c-MEC}. 
\end{proof}

\subsection{Proof of \cref{lem:Oset-informative}} \label{sec:proof-lem:Oset-informative}
\begin{proof} 
\begin{figure}[!htb]
\centering
\begin{tikzpicture}
\tikzset{rv/.style={circle,inner sep=1pt,fill=gray!20,draw,font=\sffamily}, 
sv/.style={circle,inner sep=0pt,fill=gray!20,draw,font=\sffamily,minimum size=3mm}, 
node distance=12mm, >=stealth}
\node[rv] (Ml) {$M_L$};
\node[rv, left of=Ml] (Oi) {$O_i$};
\node[rv, above of=Oi] (U) {$U$};
\node[right of=Ml] (e1) {$\dots$};
\node[rv, right of=e1] (Y) {$Y$};
\node[left of=U] (e2) {$\dots$};
\node[rv, left of=e2] (A) {$A$};
\draw[->, very thick, color=blue] (Oi) -- (Ml);
\draw[->, very thick, color=red] (Ml) -- (e1);
\draw[->, very thick, color=red] (e1) -- (Y);
\draw[->, very thick, color=brown] (U) -- (Ml);
\draw[->, very thick, color=brown] (A) -- (e2);
\draw[->, very thick, color=brown] (e2) -- (U);
\end{tikzpicture}
\caption{Proof that $O_i$ is informative: path $p$ ({\color{brown} $\bm{-}$}), path $q$ ({\color{red} $\bm{-}$})}
\label{fig:proof-nec-O}
\end{figure}
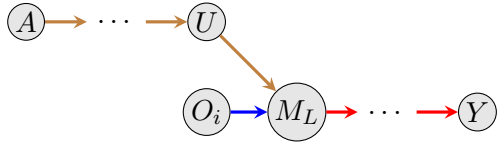 

From \cref{lem:EIF}, it is easy to see that $A$ and $Y$ cannot be uninformative. Next, we show that every $O_i \in \Oset$ cannot be contained in any uninformative set.

Fix $O_i \in \Oset(\g)$. We label the children of $O_i$: $\Ch(O_i, \g) \cap \Wset(\g)$ are numbered topologically as $\{W_1, \dots, W_K\}$ for $K \geq 0$, and $\Ch(O_i, \g) \cap \Mset(\g)$ are numbered topologically as $\{M_1, \dots, M_L\}$ for $L \geq 1$. Let $q$ be the shortest causal path from $M_L$ to $Y$, with $M_L \equiv Y$ as a special case. Let $p$ be the shortest causal path from $A$ to $M_L$. Let $U$ be the vertex that precedes $M_L$ on $p$, with $U \equiv A$ as a special case. Without loss of generality, we take $a=1$ and replace $\I_{a}(A)$ with $A$. To show $O_i$ is not contained by any uninformative variable set, it suffices to show that 
\begin{multline*}
\Gamma(O_i) \equiv \E[b(\Oset) \mid O_i, \Pa(O_i, \g)] \\
+ \sum_{k \leq K} \left\{\E[b(\Oset) \mid W_k, \Pa(W_k,\g)] - \E[b(\Oset) \mid \Pa(W_k,\g)] \right\} \\
+ \sum_{l \leq L} \left\{\E[ AY/P(A=1 \mid \Omin) \mid M_l, \Pa(M_l,\g)] - \E[AY/P(A=1 \mid \Omin) \mid \Pa(M_l,\g)] \right\},
\end{multline*}
the part of $\Psi_{a,P,\text{eff}}^{1}(\Vset; \g)$ that could depend on $O_i$, indeed non-trivially depends on $O_i$ under some degenerate law that is Markov to $\g$. 

For this purpose, we shall choose $P$ that is Markov to a subgraph $\g'$ of $\g$.
Let $\g'$ be a subgraph of $\g$ that consists of the same set of vertices but only includes edges $O_i \rightarrow M_L$ and those on $p$, $q$; see \cref{fig:proof-nec-O}. Let $P$ be chosen that is Markov to $\g'$ such that the following hold almost surely:
\begin{enumerate}[(1)]
\item $A = \dots = U$ along path $p$;
\item $M_L = \dots = Y$ along path $q$;
\item $\E[M_L \mid O_i, U=1] = b(O_i)$ for some function $b$;
\item $P(A=1 \mid \Omin) = c$ for some constant $c \in (0,1)$.
\end{enumerate}
Note that it follows that $b(\Oset) = b(O_i)$ almost surely under $P$.

Let us compute $\Gamma(O_i)$ term by term. First, it is easy to see
\[ \E[b(\Oset) \mid O_i, \Pa(O_i, \g)] = b(O_i). \]
Then, we have
\begin{equation*}
\E[b(\Oset) \mid W_k, \Pa(W_k,\g)] - \E[b(\Oset) \mid \Pa(W_k,\g)] = 0
\end{equation*}
because $b(\Oset) = b(O_i)$ and $O_i \in \Pa(W_k, \g)$. For any $l < L$, we claim
\begin{multline*}
\E[ AY/P(A=1 \mid \Omin) \mid M_l, \Pa(M_l,\g)] - \E[AY/P(A=1 \mid \Omin) \mid \Pa(M_l,\g)] \\
= c^{-1}\left\{\E[ AY \mid M_l, \Pa(M_l,\g)] - \E[AY \mid \Pa(M_l,\g)] \right\} = 0
\end{multline*}
because $A, Y$ are non-descendants of $M_l$ on $\g'$. In particular, $M_l$ cannot be on $q$ by topological ordering. For $l=L$, we have
\begin{equation*}
\begin{split}
&\quad \E[ AY/P(A=1 \mid \Omin) \mid M_L, \Pa(M_L,\g)] - \E[AY/P(A=1 \mid \Omin) \mid \Pa(M_L,\g)] \\
&= c^{-1}\left\{\E[ A M_L \mid M_L, \Pa(M_L,\g)] - \E[A M_L \mid \Pa(M_L,\g)] \right\} \\
&= c^{-1}\left\{A M_L - A \E[ M_L \mid \Pa(M_L,\g)] \right\} \\
&= c^{-1}A (M_L - \E[ M_L \mid O_i, U]) \\
&= c^{-1}A (M_L - \E[ M_L \mid O_i, A]) \\
&= c^{-1}A (M_L - \E[ M_L \mid O_i, A=1]) = c^{-1}A (M_L - b(O_i)).
\end{split}
\end{equation*}
Finally, we arrive at
\[ \Gamma(O_i) = (1 - A/c)\, b(O_i) + A M_L / c,\]
which depends on $O_i$ through $b(O_i)$. To remedy the fact that the chosen $P$ is degenerate, consider a sequence of non-degenerate laws $P_n \in \model(\g', \Vset)$ that weakly converges to $P$. For large enough $n$, $\Gamma(O_i)$ depends on $O_i$ under $P_n$. 
\end{proof}

\subsection{Soundness of W-criterion and M-criterion} \label{sec:app-sound}
\begin{proof}[of \cref{lem:plus-minus-W}]
For (i), note that the causal path from $W_{j}$ to $\Oset$ cannot be blocked by any subset of $\Pa(W_{j})$. For (ii), note that the causal path $ W_{j}\rightarrow W_{j_{r}}\rightarrow \dots \rightarrow O$ cannot be blocked by any subset of $\{W_{j_{t}}\} \cup \Pa(W_{j_{r}})\setminus \{W_{j}\}$ for $t<r$ by topological ordering. Statement (iii) holds trivially if $\Pa (W_{j})=\emptyset $. When $\Pa(W_{j})\neq \emptyset $, (iii) also holds because for every $Z\in \Pa(W_{j})$, the causal path $Z\rightarrow W_{j}\rightarrow \dots \rightarrow \Oset$ is not blocked by any subset of $ \Pa(W_{j})\setminus \{Z\}$. 
\end{proof}

\medskip \begin{proof}[of \cref{lem:W-cancel-plus-minus}]
For $P \in \model(\g, \Vset)$, $\Psi_{a,P,\text{eff}}^{1}(\Vset; \g)$ depends on $W_j$ only through $\Gamma(W_j)$ given by \cref{eqs:depends-Wj-plus-minus}. Hence, under (i) and (ii), $\Psi_{a,P,\text{eff}}^{1}(\Vset; \g)$ is does not depend on $W_j$ for every $P \in \model(\g, \Vset)$. This shows that $\{W_j\}$ is uninformative for estimating the g-functional under $\model(\g, \Vset)$; see the beginning of \cref{sec:criteria}.
\end{proof}

\medskip \begin{lemma} \label{lem:app-W-plus-minus-equiv}
Conditions (i) and (ii) in \cref{lem:W-cancel-plus-minus} are equivalent to the conditions (i) and (ii) in \cref{lem:W-criterion}.
\end{lemma}
\begin{proof}
We first show that W-criterion's (i)  $\iff$ (i) of \cref{lem:W-cancel-plus-minus}. For ``$\Rightarrow$'', by definition of $\plus_{j_r}$, we know $\plus_{j_r} \subseteq  \{W_{j_r}\} \cup \Pa(W_{j_r}) \setminus \{W_j\}$ and hence $W_j \notin \plus_{j_r}$. For ``$\Leftarrow$'', again by definition, we have $\{W_{j_r}\} \cup \Pa(W_{j_r}) \setminus \plus_{j_r} \indep_{\g} \Oset \mid \plus_{j_r}$. Since $W_j \rightarrow W_{j_r}$, the result then follows from the weak union property of d-separation.

Noting \cref{lem:plus-minus-W}(iii), we shall show that W-criterion's (ii) above $\iff$ $\plus_{j_{t-1}} = \Pa(W_{j_t})$ for $t=1,\dots,r$. We start with the ``$\Rightarrow$'' direction. Statements (a)--(c) imply that $\Pa(W_{j_t}) \subseteq \Pa(W_{j_{t-1}}) \cup \{W_{j_{t-1}}\}$ and $\left[\Pa(W_{j_{t-1}}) \cup \{W_{j_{t-1}}\}\right] \setminus \Pa(W_{j_t}) \indep_{\g} \Oset \mid \Pa(W_{j_t})$. Then, by definition of $\plus_{j_{t-1}}$, we know $\plus_{j_{t-1}} \subseteq \Pa(W_{j_t})$. We continue to show that $\Pa(W_{j_t}) \subseteq \plus_{j_{t-1}}$. By (a) and $W_{j_{t-1}} \in \plus_{j_{t-1}}$ according to \cref{lem:plus-minus-W}(i), it remains to be shown that for every $Z \in \Pa(W_{j_t}) \setminus \{W_{j_{t-1}}\}$, we have $Z \in \plus_{j_{t-1}}$. This is true because, by topological ordering, the causal path $Z \rightarrow W_{j_t} \rightarrow \dots \rightarrow \Oset$ cannot be blocked by any subset of $\{W_{j_{t-1}}\} \cup \Pa(W_{j_{t-1}}) \setminus \{Z\}$. The ``$\Leftarrow$'' direction is immediate from the definition of $\plus_{j_{t-1}}$ and \cref{lem:plus-minus-W}(i). 
\end{proof}

\medskip \begin{proof}[of \cref{lem:W-criterion}]
This follows from \cref{lem:app-W-plus-minus-equiv}.
\end{proof}

\medskip \begin{proof}[of \cref{lem:M-criterion}]
By the reasoning at the beginning of \cref{sec:criteria}, it suffices to show that for every $P \in \model(\g, \Vset)$, $\Psi_{a,P,\text{eff}}^{1}(\Vset; \g)$ does not depend on $M_i$. By \cref{lem:EIF}, $\Psi_{a,P,\text{eff}}^{1}(\Vset; \g)$ can depend on $M_i$ only through
\begin{equation*}
\begin{split}
\Gamma(M_i) &\equiv \E\left\{ T_{a,P} \mid M_{i},\Pa(M_{i})\right\} + \sum_{l=1}^{k} \big[ \E\left\{ T_{a,P} \mid M_{i_l},\Pa(M_{i_l})\right\} - \E\left\{ T_{a,P}\mid \Pa(M_{i_l})\right\} \big] \\
&= \sum_{t=1}^{k} \big[ \E\left\{ T_{a,P} \mid M_{i_{t-1}},\Pa(M_{i_{t-1}}) \right\} - \E\left\{ T_{a,P} \mid \Pa(M_{i_{t}}) \right\} \big ] + \E\left\{ T_{a,P} \mid M_{i_k},\Pa(M_{i_k})\right\},
\end{split}
\end{equation*}
where $M_{i_0} \equiv M_i$, $T_{a,P} \equiv \I_{a}(A)Y / P\left( A=a \mid \Omin \right)$ is a function of $\{A, Y\} \cup \Omin$. 
It then suffices to show that (I) $\E\left\{ T_{a,P} \mid M_{i_k},\Pa(M_{i_k})\right\}$ does not depend on $M_i$ and (II) $\E\left\{ T_{a,P} \mid M_{i_{t-1}},\Pa(M_{i_{t-1}}) \right\}$ cancels with $\E\left\{ T_{a,P} \mid \Pa(M_{i_{t}}) \right\} $ for $t=1,\dots,k$.

To see (I), by condition (i) of the lemma, we have 
\[ \E\left\{ T_{a,P} \mid M_{i_k},\Pa(M_{i_k})\right\} = \E\left[ T_{a,P} \mid M_{i_k},\Pa(M_{i_k}) \setminus \{M_i\}\right], \]
which does not depend on $M_i$. 

To see (II), note that for $t=1,\dots,k$, 
\[ \E\left\{ T_{a,P} \mid M_{i_{t-1}},\Pa(M_{i_{t-1}}) \right\} = \E\left[ T_{a,P} \mid \Pa(M_{i_{t}}), \Pa(M_{i_{t-1}}) \setminus \Pa(M_{i_{t}}) \right] = \E\left\{ T_{a,P} \mid \Pa(M_{i_{t}}) \right\}, \]
where the first equality follows from condition (ii)(a) and (ii)(b), the second equality follows from condition (ii)(c).
\end{proof}

\subsection{Proof of \cref{thm:criteria}}
\begin{proof}
We prove the result by showing that $\Vset^{\ast}(\g)$ given in the theorem fulfills conditions (i)--(iii) in \cref{lem:info}. 
\begin{enumerate}[(i)]
\item We claim that for every $P \in \model(\g, \Vset)$, with probability one, $\Psi_{a,P,\text{eff}}^{1}(\Vset; \g)$ given by \cref{lem:EIF} depends on $\Vset$ only through $\Vset^{\ast}(\g)$. This is true because every variable in 
\begin{multline*}
\Vset \setminus \Vset^{\ast}(\g) = \Nset(\g) \cup \Iset(\g) \\
\cup \{W_{j}\in \Wset\setminus \Oset: \text{$W_{j}$ satisfies the W-criterion}\} \\
\cup \{M_{i}\in \Mset\setminus \{Y\}:\text{$M_{i}$ satisfies the M-criterion}\}
\end{multline*}
has been shown to vanish from $\Psi_{a,P,\text{eff}}^{1}(\Vset; \g)$ by \cref{lem:N-I-uninformative,lem:W-criterion,lem:M-criterion}.

\item Since $\{A, Y\} \cup \Oset \subseteq \Vset^{\ast}(\g)$, consider functional $\Psi_{a, \Oset}^{\text{ADJ}}(P^{\ast}; \g^{\ast})$ that agrees with $\Psi_{a}(P; \g)$ for $P \in \model(\g, \Vset)$. 

\item We shall show that, for every $V_j \in \Vset^{\ast}$, there exists some $P_j \in \model(\g, \Vset)$ such that $\Psi_{a,P,\text{eff}}^{1}(\Vset; \g)$ depends on $V_j$ non-trivially. 
\begin{enumerate}[(a)]
\item For $V_j \in \{A, Y\} \cup \Oset$, this is shown in the proof of \cref{lem:Oset-informative}; see \cref{sec:proof-lem:Oset-informative}.
\item For $V_j \in \Wset\setminus \Oset$ that fails the W-criterion, it is shown by \cref{lem:W-complete}; see \cref{sec:app-W-complete}.
\item For $V_j \in \Mset \setminus \{Y\}$ that fails the M-criterion, it is shown by \cref{lem:M-complete}; see \cref{sec:app-M-complete}.
\end{enumerate}
\end{enumerate}
\end{proof}

\section{Completeness proof of W-criterion} \label{sec:app-W-complete}
This section provides the following supporting result for the proof of \cref{thm:criteria}.  

\medskip \begin{lemma} \label{lem:W-complete}
Under the assumptions of \cref{thm:criteria}, suppose that variable $W_j \in \Wset(\g) \setminus \Oset(\g)$ fails the W-criterion in \cref{lem:W-criterion}. Then, there exists a non-degenerate law $P \in \model(\g, \Vset)$, under which $\Psi_{a,P,\text{eff}}^{1}(\Vset; \g)$ non-trivially depends on $W_j$.
\end{lemma}

\smallskip \begin{proof}
By \cref{lem:cond1-w-graph}, $W_j$ must be of one of the following cases. 
\begin{enumerate}
\item[(W-a)] $W_j$ has two children $W_{j_r}$ and $W_{j_k}$ that are not adjacent.
\item[(W-b)] $W_{j} \to W_{j_k} \leftarrow W_i$ with $W_j$ not adjacent to $W_i$. 
\item[(W-c)] $W_{i} \to  W_{j} \to W_{j_k}$, $i \in \{1, \dots j\}$, $W_i \in \Pa(W_j,\g) \setminus \Pa(W_{j_k},\g) $, and there is a path $p = \langle W_i, \dots , O_1 \rangle$, $O_1 \in \opt(\g)\setminus \Pa(W_{j_k},\g)$  that is d-connecting given $\Pa(W_{j_k},\g)$. If $W_i \in \opt(\g)$, then $W_i \equiv O_1$ and $|p|=0$.
\end{enumerate}
By splitting into these cases, the result is shown to hold by \cref{lem:W-a,lem:W-b,lem:W-c}.
\end{proof}

\medskip To show that $\Psi_{a,P,\text{eff}}^{1}(\Vset; \g)$ non-trivially depends on $W_j$, by \cref{lem:EIF}, it is suffices to show under $P$, 
\begin{equation} \label{eqs:ast-W}
\Gamma(W_j) \equiv \E \left\{ b(\Oset) \mid W_{j}, \Pa(W_{j}) \right\} + \sum_{t=1}^{r} \left[ \E \left\{ b(\Oset) \mid W_{j_t}, \Pa(W_{j_t}) \right\} - \E \left\{ b(\Oset) \mid  \Pa(W_{j_t}) \right\}  \right]
\end{equation}
non-trivially depends on $W_j$, where $\Ch(W_{j})\cap \Wset=\{W_{j_{1}},\dots ,W_{j_{r}}\}$ for $r \geq 1$. For convenience, in this section we suppose $A$ is binary and choose $a=1$. We write $b(\Oset) \equiv b_{1}(\Oset)$ for short.

\medskip For the proofs, we will use the property of an inducing path. 

\begin{definition}[inducing path] \label{def:inducing-path}
Path $p = \langle A, \dots, B \rangle$ between $A$ and $B$ is called an inducing path with respect to set $C$ for $A, B \notin C$, if (i) no non-collider on $p$ is in $C$ and (ii) every collider on $p$ is ancestral to $\{A, B\}$. An edge between $A$ and $B$ is a trivial inducing path.
\end{definition}

\begin{lemma}[Lemma 1, \citealp{richardson2003markov}] \label{lem:inducing-path}
If there exists an inducing path between $A$ and $B$ with respect to $C$, then $A$ and $B$ are d-connected given $C$. 
\end{lemma}

\medskip In what follows, we will typically choose $P$ that is Markov to a subgraph $\g'$ of $\g$, and hence also Markov to $\g$. Note that in \cref{eqs:ast-W}, $\Pa(\cdot)$ is defined with respect to the original graph $\g$ instead of the subgraph  $\g'$.
We will use symbol `$\stackrel{\g}{=}$' or `$\stackrel{\g'}{=}$' to signify an equality that follows from d-separations on $\g$ or $\g'$. Besides, for two paths $p = \langle V_i, \dots, V_j \rangle $ and $q = \langle V_j, \dots, V_k \rangle$, notation $p \oplus q$ denotes the path formed by concatenating $p$ and $q$. 

\subsection{Case (W-a)}
\begin{figure}[!htb]
\centering
\begin{tikzpicture}
\tikzset{rv/.style={circle,inner sep=1pt,fill=gray!20,draw,font=\sffamily}, 
redv/.style={circle,inner sep=1pt,fill=white,draw,dashed,font=\sffamily}, 
ov/.style={circle,inner sep=1pt,fill=gray!20,draw=red,thick,font=\sffamily}, 
sv/.style={circle,inner sep=1pt,fill=gray!20,draw,font=\sffamily,minimum size=1mm}, 
node distance=12mm, >=stealth, every node/.style={scale=0.8}}
\begin{scope} \node[name=Wj, rv]{$W_j$};
\node[name=Wr, rv, below of=Wj, xshift=-17mm, yshift=-5mm]{$W_{j_r}$};
\node[name=Wk, rv, right of=Wr]{$W_{j_k}$};
\node[name=Wl, rv, right of=Wk, xshift=5mm]{$W_{j_l}$};
\node[name=Wi, rv, above of=Wr, yshift=5mm]{$W_i$};
\node[name=O1, rv, below of=Wr]{$O_1$};
\node[name=O2, rv, below of=Wk]{$O_2$};
\node[name=M, rv, below of=O1]{$M$};
\node[name=Y, rv, below of=O2]{$Y$};
\node[name=A, rv, left of=M]{$A$};
\node[below=3mm of M, xshift=10mm]{(W-a)};
\draw[->,very thick,>=stealth,color=blue] (Wj) to (Wr);
\draw[->,very thick,>=stealth,color=blue] (Wi) to (Wj);
\draw[->,very thick,>=stealth,color=blue, bend right] (Wi) to (A);
\draw[->,dashed,very thick,>=stealth,color=blue] (Wi) to (Wr);
\draw[->,very thick,>=stealth,color=blue] (Wj) to (Wk);
\draw[->,very thick,>=stealth,color=blue] (Wj) to (Wl);
\draw[->,very thick,>=stealth,color=blue] (Wr) to (O1);
\draw[->,very thick,>=stealth,color=blue] (Wk) to (O2);
\draw[->,dashed,very thick,>=stealth,color=blue] (Wl) to (Wk);
\draw[->,very thick,>=stealth,color=blue] (O2) to (O1);
\draw[->,very thick,>=stealth,color=blue] (Wl) to (O2);
\draw[->,very thick,>=stealth,color=blue] (O1) to (A);
\draw[->,very thick,>=stealth,color=blue] (O1) to (M);
\draw[->,very thick,>=stealth,color=blue] (O1) to (Y);
\draw[->,very thick,>=stealth,color=blue] (A) to (M);
\draw[->,very thick,>=stealth,color=blue] (M) to (Y);
\draw[->,very thick,>=stealth,color=blue, bend right] (A) to (Y);
\end{scope}
\begin{scope}[xshift=4cm] \node[rv, yshift=-30mm] (Wjk) {$W_{j_k}$};
\node[rv, left of=Wjk, yshift=-8mm] (Wi) {$W_i$};
\node[rv, left of=Wjk, yshift=8mm] (Wj) {$W_j$};
\node[right of=Wjk] (e1) {$\dots$};
\node[rv, right of=e1] (O1) {$O_1$};
\node[right of=O1] (e2) {$\dots$};
\node[rv, right of=e2] (Y) {$Y$};
\draw[->, very thick, color=blue] (Wi) -- (Wjk);
\draw[->, very thick, color=blue] (Wj) -- (Wjk);
\draw[->, very thick, color=blue] (Wjk) -- (e1);
\draw[->, very thick, color=blue] (e1) -- (O1);
\draw[->, very thick, color=blue] (O1) -- (e2);
\draw[->, very thick, color=blue] (e2) -- (Y);
\node[below of=O1, xshift=-2mm, yshift=-9mm] {(W-b)};
\end{scope}
\end{tikzpicture}
\caption{Examples of (W-a) and (W-b) for showing the dependency on $W_j$. In (W-a), dashed edges are removed from $\g$ to form $\g'$.}
\label{fig:proof-nec-W}
\end{figure}
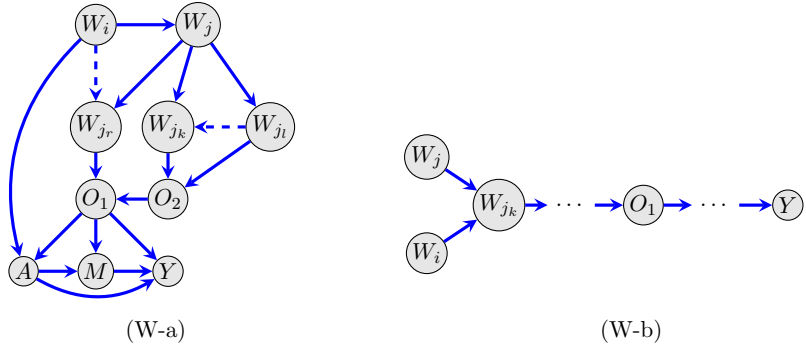

\begin{lemma} \label{lem:W-a}
Under the assumptions of \cref{thm:criteria}, suppose that variable $W_j \in \Wset(\g) \setminus \Oset(\g)$ satisfies condition (W-a) in $\g$. Then, there exists a non-degenerate law $P \in \model(\g, \Vset)$, under which $\Psi_{a,P,\text{eff}}^{1}(\Vset; \g)$ non-trivially depends on $W_j$.
\end{lemma}
\begin{proof}
Let $\g'$ be the subgraph of $\g$ by removing the edges into $W_{j_r}$ and $W_{j_k}$ other than $W_j \rightarrow W_{j_r}$ and $W_j \rightarrow W_{j_k}$; see \cref{fig:proof-nec-W} for an example. Thus, $\Pa(W_{j_r}, \g') = \Pa(W_{j_k}, \g') = \{W_j\}$. 

Choose $P$ that is Markov to $\g'$ such that the following hold almost surely:
\begin{enumerate}
\item $\E[W_{j_k} \mid W_j] = 0.$
\item $b(\opt) = W_{j_r} W_{j_k}$. This is always possible because $W_{j_r}$ and $W_{j_k}$ go to $Y$ through $\opt$ (either the same vertex or two different vertices from $\opt$).
\end{enumerate}
Rewriting \cref{eqs:ast-W}, our goal is to show 
\begin{multline*}
\Gamma(W_j) = \E[ W_{j_r} W_{j_k} \mid W_j, \Pa(W_j, \g)] \\
+ \E[ W_{j_r} W_{j_k} \mid W_{j_r}, \Pa(W_{j_r}, \g)] - \E[ W_{j_r} W_{j_k} \mid \Pa(W_{j_r}, \g)] \\
+ \E[ W_{j_r} W_{j_k} \mid W_{j_k}, \Pa(W_{j_k}, \g)] - \E[ W_{j_r} W_{j_k} \mid \Pa(W_{j_k}, \g)] \\
+ \sum_{l \neq r, k} \left\{ \E[W_{j_r} W_{j_k} \mid W_{j_l}, \Pa(W_{j_l}, \g)] - \E[W_{j_r} W_{j_k} \mid \Pa(W_{j_l}, \g)] \right\}
\end{multline*}
depends on $W_j$ under $P$.

Let us compute term by term. Using local Markov properties on $\g'$, it is easy to see that
\[ \E[ W_{j_r} W_{j_k} \mid W_j, \Pa(W_j, \g)] \stackrel{\g'}{=} \E[ W_{j_r} \mid W_j] \E[W_{j_k} \mid W_j] = 0, \]
We also have
\begin{equation*}
\begin{split}
& \quad \E[ W_{j_r} W_{j_k} \mid W_{j_r}, \Pa(W_{j_r}, \g)] - \E[ W_{j_r} W_{j_k} \mid \Pa(W_{j_r}, \g)]\\
&= \E[W_{j_r} W_{j_k} \mid W_{j_r}, W_j, \Pa(W_{j_r}, \g) \setminus \{W_j\}] - \E[W_{j_r} W_{j_k} \mid W_j, \Pa(W_{j_r}, \g) \setminus \{W_j \}] \\
&\stackrel{\g'}{=} W_{j_r} \E[W_{j_k} \mid W_j] - \E[W_{j_r} \mid W_j] \E[W_{j_k} \mid W_j] = 0,
\end{split}
\end{equation*}
and 
\begin{equation*}
\begin{split}
& \quad \E[W_{j_r} W_{j_k} \mid W_{j_k}, \Pa(W_{j_k}, \g)] - \E[W_{j_r} W_{j_k} \mid \Pa(W_{j_k}, \g)] \\
&\stackrel{\g'}{=}  W_{j_k} \E[W_{j_r} \mid W_j] - \E[W_{j_r} \mid W_j] \E[W_{j_k} \mid W_j] = W_{j_k} \E[W_{j_r} \mid W_j].
\end{split}
\end{equation*}
For any other child $W_{j_l}$ of $W_j$ ($l \neq k,r$) (if any), we claim that 
\[ \E[W_{j_r} W_{j_k} \mid W_{j_l}, \Pa(W_{j_l}, \g)] - \E[W_{j_r} W_{j_k} \mid \Pa(W_{j_l}, \g)] \stackrel{\g'}{=} 0.\]
This holds because $W_{j_r}$ and $W_{j_k}$ are non-descendants of $W_{j_l}$ on $\g'$ and hence 
\[ W_{j_l} \indep_{\g'} W_{j_r}, W_{j_k}, \Pa(W_{j_l}, \g) \setminus \Pa(W_{j_l}, \g') \mid \Pa(W_{j_l}, \g'),\]
which further implies
\[ W_{j_l} \indep_{\g'} W_{j_r}, W_{j_k} \mid \Pa(W_{j_l}, \g).\]
Finally, we are left with
\[ \Gamma(W_j) = W_{j_k} \E[W_{j_r} \mid W_j], \]
which can be chosen to depend on $W_j$. Finally, to finesse the fact that $P$ is degenerate, consider a sequence of non-degenerate laws $P_n$ that weakly converges to $P$ in $\model(\g', \Vset)$. Then, $\Gamma(W_j)$ depends on $W_j$ under $P_n$ for a large enough $n$. 
\end{proof}

\subsection{Case (W-b)}
\begin{lemma} \label{lem:W-b}
Under the assumptions of \cref{thm:criteria}, suppose that variable $W_j \in \Wset(\g) \setminus \Oset(\g)$ satisfies condition (W-b) in $\g$. Then, there exists a non-degenerate law $P \in \model(\g, \Vset)$, under which $\Psi_{a,P,\text{eff}}^{1}(\Vset; \g)$ non-trivially depends on $W_j$.
\end{lemma}
\begin{proof}
Let $p = \langle W_{j_k}, \dots, O_1 \rangle$ for $O_1 \in \opt$ be the shortest causal path from $W_{j_k}$ to $\opt$. Let $q$ to be the shortest causal path from $O_1$ to $Y$. Let $\g'$ be a subgraph of $\g$ that consists of the same set of vertices but only includes the edges on paths $p$ and $q$; see \cref{fig:proof-nec-W}.

Choose $P$ that is Markov to $\g'$ such that the following holds almost surely:
\begin{enumerate}[(i)]
\item $W_{j_k} = \dots = O_1$ along path $p$
\item $b(\opt) = O_1$,
\item $\E[W_{j_k} \mid W_i, W_j] = W_i W_j$.
\end{enumerate}
It then follows that $b(\opt) = W_{j_k}$. 
Now we shall show that $\Gamma(W_j)$ depends on $W_j$ under $P$. It holds that 
\begin{equation*}
\begin{split}
\E[b(\Oset) \mid W_j, \Pa(W_j, \g)] &= \E[W_{j_k} \mid W_j, \Pa(W_j, \g)]  \\
&\stackrel{\g'}{=} \E[W_{j_k} \mid W_j] \\
& = \E\left[ \E[W_{j_k} \mid W_i, W_j] \mid W_j \right]  \\
& = \E[ W_i W_j \mid W_j]  = W_j \E[W_i].
\end{split}
\end{equation*}
We also have
\begin{equation*}
\begin{split}
\E[b(\Oset) \mid W_{j_k}, \Pa(W_{j_k}, \g)] - \E[b(\Oset) \mid \Pa(W_{j_k}, \g)] &= \E[W_{j_k} \mid W_{j_k}, \Pa(W_{j_k}, \g)] - \E[W_{j_k} \mid \Pa(W_{j_k}, \g)] \\
&\stackrel{\g'}{=} W_{j_k} - \E[W_{j_k} \mid W_i, W_j] = W_{j_k} - W_i W_j.
\end{split}
\end{equation*}
For any $W_{j_l} \in \Ch(W_j, \g)$, $l \neq k$ (if any), it holds that 
\begin{equation*}
\E[b(\Oset) \mid W_{j_l}, \Pa(W_{j_l}, \g)] - \E[b(\Oset) \mid \Pa(W_{j_l}, \g)] = \E[W_{j_k} \mid W_{j_l}, \Pa(W_{j_l}, \g)] - \E[W_{j_k} \mid \Pa(W_{j_l}, \g)] \stackrel{\g'}{=} 0
\end{equation*}
by $W_{j_k}$ being a non-descendant of $W_{j_l}$ on $\g'$ and the Markov property. Hence, we are left with
\[ \Gamma(W_j) = W_j \E[W_i] - W_i W_j + W_{j_k},\]
which depends on $W_j$. Finally, to finesse the fact that $P$ is degenerate, consider a sequence of non-degenerate laws $P_n$ that weakly converges to $P$ in $\model(\g', \Vset)$. Then, $\Gamma(W_j)$ depends on $W_j$ under $P_n$ for a large enough $n$. 
\end{proof}

\subsection{Case (W-c)}
\begin{lemma} \label{lem:W-c}
Under the assumptions of \cref{thm:criteria}, suppose that variable $W_j \in \Wset(\g) \setminus \Oset(\g)$ satisfies (W-c), but neither (W-a) nor (W-b) in $\g$. Then, there exists a non-degenerate law $P \in \model(\g, \Vset)$, under which $\Psi_{a,P,\text{eff}}^{1}(\Vset; \g)$ non-trivially depends on $W_j$.
\end{lemma}

\smallskip \begin{proof}
Let path $p = \langle W_i, \dots, O_1 \rangle$ and path $q: W_{j_k} \rightarrow \dots \rightarrow O_2$ for $O_1, O_2 \in \opt(\g)$ be chosen according to \cref{lem:cond2-w-graph}.
Then, depending on whether $p$ and $q$ intersect, and if so how they intersect, (W-c) can be further divided into the following 4 sub-cases.

\begin{enumerate}[label=(W-c\arabic*)]
\item \label{case:W-c1} No vertex in $\Ch(W_j, \g) \setminus \{W_{j_k}\}$  is on either $p$ or $q$. Further, there is no vertex that is on both $p$ and $q$ ($W_i \equiv O_1$, or $W_{j_k} \equiv O_2$ is a special case).

\item \label{case:W-c2} No vertex in $\Ch(W_j, \g) \setminus \{W_{j_k}\}$  is on either $p$ or $q$, but there is a vertex that is on both $p$ and $q$. 

\item \label{case:W-c3} A vertex in $\Ch(W_j, \g) \setminus \{W_{j_k}\}$ is on $p$. 

\item \label{case:W-c4} A vertex in $\Ch(W_j, \g) \setminus \{W_{j_k}\}$ is on $q$. 

\end{enumerate}

The result is established under each case: \cref{lem:W-c1} proves \ref{case:W-c1}, \cref{lem:W-c2} proves \ref{case:W-c2}, and \cref{lem:W-c34} proves \ref{case:W-c3} and \ref{case:W-c4}.
\end{proof}

\begin{figure}[!htb]
\centering
\begin{tikzpicture}
\tikzset{rv/.style={circle,inner sep=1pt,fill=gray!20,draw,font=\sffamily}, 
redv/.style={circle,inner sep=1pt,fill=white,draw,dashed,font=\sffamily}, 
ov/.style={circle,inner sep=1pt,fill=gray!20,draw=red,thick,font=\sffamily}, 
sv/.style={circle,inner sep=1pt,fill=gray!20,draw,font=\sffamily,minimum size=1mm}, 
node distance=12mm, >=stealth, every node/.style={scale=0.8}}
\node[rv] (Wi) {$W_i$};
\node[rv, right of=Wi] (C1) {$C_1$};
\node[right of=C1] (E1) {$\dots$};
\node[rv, right of=E1] (C2) {$C_2$};
\node[right of=C2] (E2) {$\dots$};
\node[sv,  right of=E2] (F3) {};
\node[right of=F3] (E5) {$\dots$};
\node[rv, right of=E5] (C3) {$C_3$};
\node[rv, right of=C3] (O1) {$O_1$};
\draw[->, very thick, color=brown] (Wi) -- (C1); 
\draw[<-, very thick, color=brown] (C1) -- (E1);
\draw[->, very thick, color=brown] (E1) -- (C2);
\draw[<-, very thick, color=brown] (C2) -- (E2); 
\draw[->, very thick, color=brown] (E2) -- (F3);
\draw[->, very thick, color=brown] (F3) -- (E5);
\draw[->, very thick, color=brown] (E5) -- (C3);
\draw[<-, very thick, color=brown] (C3) -- (O1);
\node[below of=C1] (E3) {$\vdots$};
\node[sv, below of=E3] (F1) {};
\node[rv, below of=F1] (Wj) {$W_j$};
\draw[->, very thick, color=blue] (C1) -- (E3);
\draw[->, very thick, color=blue] (E3) -- (F1);
\draw[->, very thick, color=blue] (F1) -- (Wj);
\node[below of=E1, xshift=-4mm, yshift=-4mm] (E4) {\reflectbox{$\ddots$}};
\node[sv, below of=E1, xshift=4mm, yshift=4mm] (F2) {};
\draw[->, very thick, color=blue] (C2) -- (F2);
\draw[->, very thick, color=blue] (F2) -- (E4);
\draw[->, very thick, color=blue] (E4) -- (F1);
\node[below of=F3, yshift=-3mm, xshift=-4mm] (E6) {$\dots$};
\draw[->, very thick, color=blue, bend left] (C3) to (E6);
\draw[->, very thick, color=blue, out=180, in=-45] (E6) to (F2);
\draw[->, very thick, color=blue] (Wi) to (Wj);
\node[rv, below of=Wj] (Wjk) {$W_{j_k}$};
\node[right of=Wjk] (E7) {$\dots$};
\node[rv, right of=E7] (O2) {$O_2$};
\draw[->, very thick, color=blue] (Wj) to (Wjk);
\draw[->, very thick, color=red] (Wjk) to (E7);
\draw[->, very thick, color=red] (E7) to (O2);
\end{tikzpicture}
\caption{Case \ref{case:W-c1}: path $p$ ({\color{brown} $\bm{-}$}), path $q$ ({\color{red} $\bm{-}$})}
\label{fig:proof-nec-W-c1}
\end{figure} 

\medskip \begin{lemma} \label{lem:W-c1}
Under the assumptions of \cref{thm:criteria}, suppose that variable $W_j \in \Wset(\g) \setminus \Oset(\g)$ satisfies \ref{case:W-c1}, but neither (W-a) nor (W-b) in $\g$. Then, there exists a non-degenerate law $P \in \model(\g, \Vset)$, under which $\Psi_{a,P,\text{eff}}^{1}(\Vset; \g)$ non-trivially depends on $W_j$.
\end{lemma}

\begin{proof}
We shall prove the dependency on $W_j$ under a law $P$ that is Markov to a subgraph $\g'$ of $\g$. Let $\g'$ be chosen as the subgraph of $\g$ containing the same set of vertices, but only edges on the following paths (see \cref{fig:proof-nec-W-c1} for an example):
\begin{enumerate}[(1)]
\item $W_i \rightarrow W_j \rightarrow W_{j_k}$,
\item $p = \langle W_i, \dots, O_1 \rangle$,
\item $q = \langle W_{jk}, \dots, O_2 \rangle$,
\item when $p$ contains colliders, say $C_1, \dots, C_F$ ($F \geq 1$), then for each collider $C_f$ also include the shortest causal path $c_f$ from $C_f$ to $W_j$, which exist on $\g$ by \cref{lem:cond2-w-graph} \ref{case:p4},
\item and a path from $\{O_1, O_2\}$ to $Y$ (omitted in \cref{fig:proof-nec-W-c1}).
\end{enumerate}

Let $P$ be chosen such that the following hold almost surely:
\begin{enumerate}[(i)]
\item $\E[W_{j_k} \mid W_j] = 0$,
\item $W_{j_k} = \dots = O_2$ ($q$ is an identity path),
\item $b(\opt) = f(O_1)\, O_2$ for some function $f$.
\end{enumerate}

Note that $b(\opt) = f(O_1) W_{j_k}$. Rewriting \cref{eqs:ast-W}, our goal is to show that 
\begin{multline} 
\Gamma(W_j) = \E[f(O_1) W_{j_k} \mid W_j, \Pa(W_j, \g)] \\
+  \E[f(O_1) W_{j_k} \mid W_{j_k}, \Pa(W_{j_k}, \g)] - \E[f(O_1) W_{j_k} \mid \Pa(W_{j_k}, \g)] \\
+ \sum_{l \neq k} \left\{ \E[f(O_1) W_{j_k} \mid W_{j_l}, \Pa(W_{j_l}, \g)] - \E[f(O_1) W_{j_k} \mid \Pa(W_{j_l}, \g)] \right\}
\end{multline}
non-trivially depends on $W_j$. Note that the parent set is defined with respect to $\g$.

Now we compute each term under $P$. First, we have
\begin{equation*}
\begin{split}
\E[f(O_1) W_{j_k} \mid  W_j, \Pa(W_j, \g)] &\stackrel{\g'}{=} \E[W_{j_k} \mid W_j, \Pa(W_j, \g)] \E[f(O_1) \mid  W_j, \Pa(W_j, \g)] \\
&\stackrel{\g'}{=} \E[W_{j_k} \mid W_j] \E[f(O_1) \mid  W_j, \Pa(W_j, \g)] = 0
\end{split}
\end{equation*}
by (i) in our choice of $P$. The first equality follows from the local Markov property on $\g'$:
\[ W_{j_k} \indep_{\g'} O_1, \Pa(W_j, \g) \mid W_j, \]
where $O_1, \Pa(W_j, \g)$ are non-descendants of $W_j$. 

Second, we have 
\begin{equation*}
\begin{split}
\E[f(O_1) W_{j_k} \mid \Pa(W_{j_k}, \g)] &\stackrel{\g'}{=} \E[f(O_1) \mid \Pa(W_{j_k}, \g)] E[W_{j_k} \mid \Pa(W_{j_k}, \g)] \\
&\stackrel{\g'}{=}\E[f(O_1) \mid \Pa(W_{j_k}, \g)] \E[W_{j_k} \mid W_j] = 0,
\end{split}
\end{equation*}
where the first equality follows from the local Markov property
\[ W_{j_k} \indep_{\g'} O_1, \Pa(W_{j_k}, \g) \setminus \Pa(W_{j_k}, \g') \mid \Pa(W_{j_k}, \g'). \]

Third, note that every summand in the final term of $\Gamma(W_j)$ vanishes. For any other child $W_{j_l}$ of $W_j$ ($l \neq k$), $W_{j_l}$ is not on $p$ or $q$ by our assumption. Hence, $O_1$ and $W_{j_k}$ are non-descendants of $W_{j_l}$ on $\g'$. By the local Markov property
\[ W_{j_l} \indep O_1, W_{j_k}, \Pa(W_{j_l}, \g) \setminus \Pa(W_{j_l}, \g') \mid \Pa(W_{j_l}, \g'),\]
it holds that
\[\E[f(O_1) W_{j_k} \mid W_{j_l}, \Pa(W_{j_l}, \g)] - \E[f(O_1) W_{j_k} \mid \Pa(W_{j_l}, \g)] \stackrel{\g'}{=} 0. \]

Finally, we are left with
\begin{equation*}
\begin{split}
\Gamma(W_j) = \E[f(O_1) W_{j_k} \mid  W_{j_k}, \Pa(W_{j_k}, \g)] &= W_{j_k} \E[f(O_1) \mid  W_{j_k}, \Pa(W_{j_k}, \g)] \\
&\stackrel{\g'}{=} W_{j_k} \E[f(O_1) \mid \Pa(W_{j_k}, \g)] \\
&= W_{j_k} \E[f(O_1) \mid W_j, \Pa(W_{j_k}, \g) \setminus \{W_j\}],
\end{split}
\end{equation*}
where the second line follows from the local Markov property. Showing that $\Gamma(W_j)$ depends on $W_j$ for some choice of $f$ and $P$ satisfying (i)--(iii), by strong completeness of d-separation \citep{meek1995strong}, is equivalent to showing
\[ O_1 \not\indep_{\g'} W_j \mid \Pa(W_{j_k}, \g) \setminus \{W_j\}. \]
But this is true by \cref{lem:inducing-path} upon observing that $\langle W_j, W_i \rangle \oplus p$ is an inducing path between $W_j$ and $O_1$ with respect to $\Pa(W_{j_k}, \g) \setminus \{W_j\}$:
\begin{enumerate}
\item $W_i$ is a non-collider and $W_i \not\in \Pa(W_j, \g)$ by (W-c); no non-collider on $p$ is in $\Pa(W_{j_k}, \g)$ since otherwise $p$ does not d-connect given $\Pa(W_{j_k}, \g)$. These conditions hold on $\g$ and hence also on $\g'$.
\item Every collider on $p$ is ancestral to $W_j$ on $\g'$.
\end{enumerate}

Finally, to finesse the fact that $P$ is degenerate, consider a sequence of non-degenerate laws $P_n$ that weakly converges to $P$ in $\model(\g', \Vset)$. Then, $\Gamma(W_j)$ depends on $W_j$ under $P_n$ for a large enough $n$. 
\end{proof}

\begin{figure}[!htb]
\centering
\begin{tikzpicture}
\tikzset{rv/.style={circle,inner sep=1pt,fill=gray!20,draw,font=\sffamily}, 
redv/.style={circle,inner sep=1pt,fill=white,draw,dashed,font=\sffamily}, 
ov/.style={circle,inner sep=1pt,fill=gray!20,draw=red,thick,font=\sffamily}, 
sv/.style={circle,inner sep=1pt,fill=gray!20,draw,font=\sffamily,minimum size=1mm}, 
node distance=12mm, >=stealth, every node/.style={scale=0.8}}
\node[rv] (Wi) {$W_i$};
\node[rv, right of=Wi] (C1) {$C_1$};
\node[right of=C1] (E1) {$\dots$};
\node[rv, right of=E1] (C2) {$C_2$};
\node[right of=C2] (E2) {$\dots$};
\node[sv,  right of=E2] (F3) {};
\node[right of=F3] (E5) {$\dots$};
\node[rv, right of=E5] (C3) {$C_3$};
\node[rv, right of=C3] (Ws) {$W_s$};
\node[rv, right of=Ws] (Wl) {$W_l$};
\node[right of=Wl] (E8) {$\dots$};
\node[rv, right of=E8] (O1) {$O_1$};
\draw[->, very thick, color=brown] (Wi) -- (C1); 
\draw[<-, very thick, color=brown] (C1) -- (E1);
\draw[->, very thick, color=brown] (E1) -- (C2);
\draw[<-, very thick, color=brown] (C2) -- (E2); 
\draw[->, very thick, color=brown] (E2) -- (F3);
\draw[->, very thick, color=brown] (F3) -- (E5);
\draw[->, very thick, color=brown] (E5) -- (C3);
\draw[->, very thick, color=brown] (Ws) -- (C3);
\draw[->, very thick, color=brown] (Ws) -- (Wl);
\draw[->, very thick, color=brown] (Wl) -- (E8);
\draw[->, very thick, color=brown] (E8) -- (O1);
\node[below of=C1] (E3) {$\vdots$};
\node[sv, below of=E3] (F1) {};
\node[rv, below of=F1] (Wj) {$W_j$};
\draw[->, very thick, color=blue] (C1) -- (E3);
\draw[->, very thick, color=blue] (E3) -- (F1);
\draw[->, very thick, color=blue] (F1) -- (Wj);
\node[below of=E1, xshift=-4mm, yshift=-4mm] (E4) {\reflectbox{$\ddots$}};
\node[sv, below of=E1, xshift=4mm, yshift=4mm] (F2) {};
\draw[->, very thick, color=blue] (C2) -- (F2);
\draw[->, very thick, color=blue] (F2) -- (E4);
\draw[->, very thick, color=blue] (E4) -- (F1);
\node[below of=F3, yshift=-3mm, xshift=-4mm] (E6) {$\dots$};
\draw[->, very thick, color=blue, bend left] (C3) to (E6);
\draw[->, very thick, color=blue, out=180, in=-45] (E6) to (F2);
\draw[->, very thick, color=blue] (Wi) to (Wj);
\node[rv, below of=Wj] (Wjk) {$W_{j_k}$};
\node[rv, below of=Ws, yshift=-5mm] (Wt) {$W_t$};
\node[below of=E6, yshift=-5mm] (E7) { \rotatebox[origin=c]{30}{$\dots$}};
\draw[->, very thick, color=blue] (Wj) to (Wjk);
\draw[->, very thick, color=red] (Wjk) to (E7);
\draw[->, very thick, color=red] (E7) to (Wt);
\draw[->, very thick, color=red] (Wt) to (Wl);
\end{tikzpicture}
\caption{Case \ref{case:W-c2}: path $p$ ({\color{brown} $\bm{-}$}) and path $q$ ({\color{red} $\bm{-}$}) merge at $W_l$.}
\label{fig:proof-nec-W-c2}
\end{figure}

\begin{lemma} \label{lem:W-c2}
Under the assumptions of \cref{thm:criteria}, suppose that variable $W_j \in \Wset(\g) \setminus \Oset(\g)$ satisfies \ref{case:W-c2}, but neither (W-a) nor (W-b) in $\g$. Then, there exists a non-degenerate law $P \in \model(\g, \Vset)$, under which $\Psi_{a,P,\text{eff}}^{1}(\Vset; \g)$ non-trivially depends on $W_j$.
\end{lemma}

\begin{proof}
Let $W_s, W_l$ and $W_t$ be chosen to satisfy case \ref{case:qandp} of \cref{lem:cond2-w-graph}. Then, path $p$ and $q$ merge at $W_l$ and we have $O_1 = O_2$. Further, the subpath from $W_l$ to $O_1$ is causal since $q$ is causal. 

We now show that $\Gamma(W_j)$ depends on $W_j$ on a law $P$ that is Markov to a subgraph $\g'$ of $\g$. Let $\g'$ be chosen the subgraph of $\g$ on the same set of vertices, but only with edges on the following paths (see \cref{fig:proof-nec-W-c2} for an example):
\begin{enumerate}[(1)]
\item $W_i \rightarrow W_j \rightarrow W_{j_k}$,
\item $p = \langle W_i, \dots, O_1 \rangle$,
\item $q = \langle W_{jk}, \dots, O_1 \rangle$,
\item when $p$ contains colliders, say $C_1, \dots, C_F$ ($F \geq 1$), then for each collider $C_f$ also include the shortest causal path $c_f$ from $C_f$ to $W_j$, which exist on $\g$ by \cref{lem:cond2-w-graph} \ref{case:p4}.
\end{enumerate}

We choose $P$ Markov to $\g'$ such that almost surely, 
\begin{enumerate}[(i)]
\item $\E[W_{j_k} \mid W_j] = 0$,
\item $W_{l} = \dots = O_1$ along $q(W_l, O_1)$,
\item $W_{j_k} = \dots = W_t$ along $q(W_{j_k}, W_t)$,
\item $W_l = f(W_s) W_t$ for some function $f$,
\item $b(\opt) = O_1$.
\end{enumerate}
It follows that $b(\opt) = O_1 = W_l = f(W_s) W_{j_k}$ almost surely. 
Then, the rest of the proof follows similarly to that of \cref{lem:W-c1} with $W_s$ playing the role of $O_1$. In particular, it is easy to see that $\langle W_j, W_i \rangle \oplus p(W_i, W_s)$ is an inducing path between $W_j$ and $W_s$ on $\g'$ with respect to $\Pa(W_{j_k}, \g) \setminus \{W_j\}$. 

Finally, to finesse the fact that $P$ is degenerate, consider a sequence of non-degenerate laws $P_n$ that weakly converges to $P$ in $\model(\g', \Vset)$. Then, $\Gamma(W_j)$ depends on $W_j$ under $P_n$ for a large enough $n$. 
\end{proof}

\begin{figure}[!htb]
\centering
\begin{tikzpicture}
\tikzset{rv/.style={circle,inner sep=1pt,fill=gray!20,draw,font=\sffamily}, 
redv/.style={circle,inner sep=1pt,fill=white,draw,dashed,font=\sffamily}, 
ov/.style={circle,inner sep=1pt,fill=gray!20,draw=red,thick,font=\sffamily}, 
sv/.style={circle,inner sep=1pt,fill=gray!20,draw,font=\sffamily,minimum size=1mm}, 
node distance=12mm, >=stealth, every node/.style={scale=0.8}}
\begin{scope}
\node[rv] (Wi) {$W_i$};
\node[rv, below of=Wi] (Wj) {$W_j$};
\node[rv, below of=Wj] (Wjk) {$W_{j_k}$};
\node[rv, right of=Wi] (Wjl) {$W_{j_l}$};
\node[right of=Wjl] (E1) {$\dots$};
\node[rv, right of=E1] (O1) {$O_1$};
\draw[->, very thick, color=blue] (Wi) -- (Wj); 
\draw[->, very thick, color=blue] (Wj) -- (Wjk); 
\draw[->, very thick, color=brown] (Wi) -- (Wjl); 
\draw[->, very thick, color=blue] (Wj) -- (Wjl); 
\draw[->, very thick, color=blue] (Wjk) -- (Wjl); 
\draw[->, very thick, color=brown] (Wjl) -- (E1);
\draw[->, very thick, color=brown] (E1) -- (O1);
\node[below of=Wjk, xshift=20mm] {(W-c3)};
\end{scope}
\begin{scope}[xshift=5cm] \node[rv] (Wi) {$W_i$};
\node[rv, below of=Wi] (Wj) {$W_j$};
\node[rv, below of=Wj] (Wjk) {$W_{j_k}$};
\node[rv, right of=Wjk] (Wjl) {$W_{j_l}$};
\node[right of=Wjl] (E1) {$\dots$};
\node[rv, right of=E1] (O2) {$O_2$};
\draw[->, very thick, color=blue] (Wi) -- (Wj); 
\draw[->, very thick, color=blue] (Wj) -- (Wjk); 
\draw[->, very thick, color=blue] (Wi) -- (Wjl); 
\draw[->, very thick, color=blue] (Wj) -- (Wjl); 
\draw[->, very thick, color=red] (Wjk) -- (Wjl); 
\draw[->, very thick, color=red] (Wjl) -- (E1);
\draw[->, very thick, color=red] (E1) -- (O2);
\node[below of=Wjl, xshift=3mm] {(W-c4)};
\end{scope}
\end{tikzpicture}
\caption{Case \ref{case:W-c3} and \ref{case:W-c4}: path $p$ ({\color{brown} $\bm{-}$}), path $q$ ({\color{red} $\bm{-}$})}
\label{fig:proof-nec-W-c34}
\end{figure} 

\medskip \begin{lemma} \label{lem:W-c34}
Under the assumptions of \cref{thm:criteria}, suppose that variable $W_j \in \Wset(\g) \setminus \Oset(\g)$ satisfies either \ref{case:W-c3} or \ref{case:W-c4}, but neither (W-a) nor (W-b) in $\g$. Then, there exists a non-degenerate law $P \in \model(\g, \Vset)$, under which $\Psi_{a,P,\text{eff}}^{1}(\Vset; \g)$ non-trivially depends on $W_j$.
\end{lemma}

\begin{proof}
The same graphical structure (see \cref{fig:proof-nec-W-c34}) can be established in $\g$ under either \ref{case:W-c3} or \ref{case:W-c4}.

First, suppose $W_j$ fulfills \ref{case:W-c3}. Then, by \ref{case:p2} of \cref{lem:cond2-w-graph}, $p$ is of the form $W_i \to W_{j_l} \to \dots \to O_1$, and $W_{j_l} \in \Ch(W_{j_k}, \g) $ is the only vertex in $\Ch(W_j, \g) \setminus \{W_{j_k}\}$ on $p$. In this case, let $q' := \langle W_{j_k}, W_{j_l} \rangle \oplus p(W_{j_l}, O_1)$ and $p':=p$. 

Otherwise, suppose $W_j$ fulfills \ref{case:W-c4}. By \ref{case:q3} of \cref{lem:cond2-w-graph}, $q$ is of the form $W_{j_k} \to W_{j_l} \to \dots \to O_2$ and $W_{j_l}$ is the only vertex in $\Ch(W_j, \g) \setminus \{W_{j_k}\}$ on $q$. Further, we argue that $W_i$ and $W_{j_l}$ are adjacent by the choice of $k$, since otherwise $W_{j_l}$ would have been chosen as $W_{j_k}$ instead. Further, by acyclicity, we know $W_i \rightarrow W_{j_l}$. Now, let $p' := \langle W_i, W_{j_l} \rangle \oplus q(W_{j_l}, O_2)$ and $q':=q$. 

We now show that $\Gamma(W_j)$ depends on $W_j$ under a law $P$ that is Markov to subgraph $\g'$ of $\g$. Let $\g'$ be chosen as the maximal subgraph of $\g$ such that $\g'_{\mathbf{W},O_1}$ under \ref{case:W-c3} (or $\g'_{\mathbf{W},O_2}$ under \ref{case:W-c4}) only contains edges appearing on the following paths: (1) $W_i \rightarrow W_j \rightarrow W_{j_k}$, (2) $p'$, and (3) $q'$.

We choose $P \in \mathcal{M}_{\g'}$ such that the following hold almost surely:
\begin{enumerate}
\item $b(\opt) = O_1$ under \ref{case:W-c3}, or $b(\opt) = O_2$ under \ref{case:W-c4},
\item $W_{j_l} = \dots = O_1$ along $p'$ under \ref{case:W-c3} or $W_{j_l} = \dots = O_2$ along $q'$ under \ref{case:W-c4},
\item $W_{j_l} = W_{i} W_{j_k}$.
\end{enumerate}
Then it follows that $b(\opt) = W_i W_{j_k}$ almost surely. 

To show dependency, it suffices to show that 
\begin{multline*}
\Gamma(W_j) = \E[W_i W_{j_k} \mid W_j, \Pa(W_j, \g)] + \E[W_i W_{j_k} \mid W_{j_k}, \Pa(W_{j_k}, \g)] - \E[W_i W_{j_k} \mid \Pa(W_{j_k}, \g)] \\
+ \E[W_i W_{j_k} \mid W_{j_l}, \Pa(W_{j_l}, \g)] - \E[W_i W_{j_k} \mid \Pa(W_{j_l}, \g)] \\
+ \sum_{m \neq k,l} \left\{ \E[W_i W_{j_k} \mid W_{j_m}, \Pa(W_{j_m}, \g)] - \E[W_i W_{j_k} \mid \Pa(W_{j_m}, \g)] \right\}
\end{multline*}
depends on $W_{j}$ non-trivially, where the expectations are taken with respect to $P$. The terms are computed as follows. 

First, we have
\begin{equation*}
\begin{split}
\E[W_i W_{j_k} \mid W_j, \Pa(W_j, \g)] &= \E[W_i W_{j_k} \mid W_j, W_i, \Pa(W_j, \g) \setminus \{W_i\}] \\
&= W_i  \E[W_{j_k} \mid W_j, W_i, \Pa(W_j, \g) \setminus \{W_i\}] \\
&\stackrel{\g'}{=} W_i  \E[W_{j_k} \mid W_j]. 
\end{split}
\end{equation*}

Then, it holds that
\begin{equation*}
\begin{split}
& \quad \E[W_i W_{j_k} \mid W_{j_k}, \Pa(W_{j_k}, \g)] - \E[W_i W_{j_k} \mid \Pa(W_{j_k}, \g)] \\
&\stackrel{\g}{=} W_{j_k} \E[W_i \mid W_{j_k}, W_j, \Pa(W_{j_k}, \g) \setminus \{W_j\}] - \E[W_i  \mid \Pa(W_{j_k}, \g)] \E[W_{j_k} \mid \Pa(W_{j_k}, \g)] \\
&\stackrel{\g'}{=}  W_{j_k} \E[W_i \mid W_j] - \E[W_i \mid W_j] \E[W_{j_k} \mid W_j],
\end{split}
\end{equation*}
where the second step uses the local Markov property. 

Due to $W_{j_l} = W_i W_{j_k}$, it is clear that 
\begin{equation*}
\E[W_i W_{j_k} \mid W_{j_l}, \Pa(W_{j_l}, \g)] - \E[W_i W_{j_k} \mid \Pa(W_{j_l}, \g)] = 0.
\end{equation*} 

And for any other child $W_{j_m} \in \Ch(W_j)$ ($m \neq k,l$), we know
\begin{equation*}
\E[W_i W_{j_k} \mid W_{j_m}, \Pa(W_{j_m}, \g)] - \E[W_i W_{j_k} \mid \Pa(W_{j_m}, \g)] = 0,
\end{equation*}
because $W_{j_l}$ is a non-descendant of $W_{j_m}$ on $\g'$ and the local Markov property holds. 

Putting the terms together, we see 
\begin{equation*}
\Gamma(W_j) =  W_i  \E[W_{j_k} \mid W_j] + W_{j_k} \E[W_i \mid W_j] - \E[W_i \mid W_j] \E[W_{j_k} \mid W_j],
\end{equation*}
which, upon further choosing $P$ such that $\E[W_{j_k} \mid W_j]=0$, reduces to 
\begin{equation*}
\Gamma(W_j) = W_{j_k} \E[W_i \mid W_j].
\end{equation*}
Clearly, this non-trivially depends on $W_j$, e.g., when $W_i, W_j$ are bivariate normal. Finally, to finesse the fact that $P$ is degenerate, consider a sequence of non-degenerate laws $P_n$ that weakly converges to $P$ in $\model(\g', \Vset)$. Then, $\Gamma(W_j)$ depends on $W_j$ under $P_n$ for a large enough $n$. 
\end{proof}

\section{Completeness proof of M-criterion} \label{sec:app-M-complete}
Similar to the previous section, this section provides the following supporting result for the proof of \cref{thm:criteria}. 

\medskip \begin{lemma} \label{lem:M-complete}
Under the assumptions of \cref{thm:criteria}, suppose that variable $M_i \in \Mset(\g) \setminus \{Y\}$ fails the M-criterion in \cref{lem:M-criterion}. Then, there exists a non-degenerate law $P \in \model(\g, \Vset)$, under which $\Psi_{a,P,\text{eff}}^{1}(\Vset; \g)$ non-trivially depends on $M_i$.
\end{lemma}

\smallskip \begin{proof}
By \cref{lem:cond1-m-graph}, $M_i$ must of one of the following cases. 
\begin{enumerate}
\item[(M-a)] $M_i$ has two children $M_{i_m}$ and $M_{i_r}$ that are not adjacent. 
\item[(M-b)] $M_{i} \to M_{i_r} \leftarrow B_j$ for $B_j \in \{A\} \cup \opt(\g) \cup \Mset(\g)$. $M_i$ and $B_j$ are non-adjacent.
\item[(M-c)] We have $B_{j} \rightarrow  M_{i} \rightarrow M_{i_r}$, but $B_j \not\rightarrow M_{i_r}$. There is a path $p = \langle B_j, \dots , S \rangle$ for $S \in \{A,Y\} \cup  \Omin(\g) \setminus \Pa(M_{j_r},\g)$  that is d-connecting given $\Pa(M_{i_r},\g)$. As a special case, if $B_j \in \{A\} \cup  \Omin(\g)$, then $B_j \equiv S$ and $|p| = 0$.
\end{enumerate}
By splitting into these cases, the result is shown to hold by \cref{lem:M-a,lem:M-b,lem:M-c}.
\end{proof}

\medskip Let us write $\rho(\Omin) \equiv P(A=1 \mid \Omin)$ for short. To show that $\Psi_{a,P,\text{eff}}^{1}(\Vset; \g)$ non-trivially depends on $M_i$, by \cref{lem:EIF}, it is suffices to show under $P$, 
\begin{multline} \label{eqs:ast-M}
\Gamma(M_i) \equiv \E[AY / \rho(\Omin) \mid M_i, \Pa(M_i, \g)] + \\
\sum_{l=1}^{k} \left\{ \E[AY / \rho(\Omin) \mid M_{i_l}, \Pa(M_{i_l}, \g)] - \E[AY / \rho(\Omin) \mid \Pa(M_{i_l}, \g)] \right \},
\end{multline}
non-trivially depends on $M_i$, where $\Ch(M_i) \cap \Mset = \{M_{i_1}, \dots, M_{i_k}\}$ for $k \geq 1$.
For convenience, in this section we suppose $A$ is binary and choose $a=1$. 

\subsection{Case (M-a)}
\begin{figure}[!htb]
\centering
\begin{tikzpicture}
\tikzset{rv/.style={circle,inner sep=1pt,fill=gray!20,draw,font=\sffamily}, 
redv/.style={circle,inner sep=1pt,fill=white,draw,dashed,font=\sffamily}, 
ov/.style={circle,inner sep=1pt,fill=gray!20,draw=red,thick,font=\sffamily}, 
sv/.style={circle,inner sep=1pt,fill=gray!20,draw,font=\sffamily,minimum size=1mm}, 
node distance=12mm, >=stealth, every node/.style={scale=0.8}}
\begin{scope}
\node[rv] (Mj) {$M_i$};
\node[rv, right of=Mj, yshift=7mm] (Mjl) {$M_{i_m}$};
\node[rv, right of=Mj, yshift=-7mm] (Mjr) {$M_{i_r}$};
\node[right of=Mjl, yshift=-2mm] (E1) {\rotatebox[origin=c]{-15}{$\cdots$}};
\node[right of=Mjr, yshift=2mm] (E2) {\rotatebox[origin=c]{15}{$\cdots$}};
\node[rv, right of=Mj, xshift=30mm] (U) {$U$};
\node[right of=U] (E3) {$\dots$};
\node[rv, right of=E3] (Y) {$Y$};
\draw[->, very thick, color=blue] (Mj) -- (Mjl); 
\draw[->, very thick, color=blue] (Mj) -- (Mjr); 
\draw[->, very thick, color=blue] (Mjl) -- (E1);
\draw[->, very thick, color=blue] (Mjr) -- (E2);
\draw[->, very thick, color=blue] (E1) -- (U);
\draw[->, very thick, color=blue] (E2) -- (U);
\draw[->, very thick, color=blue] (U) -- (E3);
\draw[->, very thick, color=blue] (E3) -- (Y);
\node[below of=Mjr, xshift=20mm] {(M-a)};
\end{scope} \begin{scope}[xshift=8cm] \node[rv] (Mjr) {$M_{i_r}$};
\node[rv, left of=Mjr, yshift=5mm] (Mj) {$M_i$};
\node[rv, left of=Mjr, yshift=-5mm] (Bi) {$B_j$};
\node[right of=Mjr] (E1) {$\cdots$};
\node[rv, right of=E1] (Y) {$Y$};
\draw[->, very thick, color=blue] (Mj) -- (Mjr);
\draw[->, very thick, color=blue] (Bi) -- (Mjr);
\draw[->, very thick, color=blue] (Mjr) -- (E1);
\draw[->, very thick, color=blue] (E1) -- (Y);
\node[below of=Mjr, xshift=5mm, yshift=-8mm] {(M-b)};
\end{scope}
\end{tikzpicture}
\caption{Case (M-a) and (M-b)}
\label{fig:proof-nec-M-ab}
\end{figure}

\begin{lemma} \label{lem:M-a}
Under the assumptions of \cref{thm:criteria}, suppose that variable $M_i \in \Mset(\g) \setminus \{Y\}$ satisfies satisfies (M-a) in $\g$. Then, there exists a non-degenerate law $P \in \model(\g, \Vset)$, under which $\Psi_{a,P,\text{eff}}^{1}(\Vset; \g)$ non-trivially depends on $M_i$.
\end{lemma}

\smallskip \begin{proof}
By definition of a mediator, let $p$ and $q$ be the shortest causal path from $M_{i_m}$ and $M_{i_r}$ to $Y$ respectively. Let $\g'$ be the subgraph of $\g$ on the same set of vertices but only with edges from $p$, $q$ and $M_{i_m} \leftarrow M_i \rightarrow M_{i_r}$; see \cref{fig:proof-nec-M-ab}.

We choose $P \in \mathcal{M}_{\g'}$ such that the following hold almost surely:
\begin{enumerate}
\item $A =1$.
\item Suppose $p$ and $q$ merge at $U$, which could be $Y$ or a vertex preceding $Y$. Let $M_{i_m} = \dots = U$ on $p$ and $M_{i_r} = \dots = U$ on $q$.
\item $U = M_{i_m} M_{i_r}$. If $U \neq Y$, further let $U = \dots = Y$. 
\item $\E[M_{i_m} \mid M_i] = 0$.
\end{enumerate}
It then follows that $AY / \rho(\Omin) = M_{i_m} M_{i_r}$ almost surely. 
Rewriting \cref{eqs:ast-M}, our goal is to show 
\begin{equation*}
\begin{split}
\Gamma(M_i) &= \E[M_{i_m} M_{i_r} \mid M_i, \Pa(M_i, \g)] \\
& \quad + \E[M_{i_m} M_{i_r} \mid M_{i_m}, \Pa(M_{i_m}, \g)] - \E[M_{i_m} M_{i_r} \mid \Pa(M_{i_m}, \g)] \\
& \quad + \E[M_{i_m} M_{i_r} \mid M_{i_r}, \Pa(M_{i_r}, \g)] - \E[M_{i_m} M_{i_r} \mid \Pa(M_{i_r}, \g)] \\
& \quad + \sum_{m \neq l,r} \left\{\E[M_{i_m} M_{i_r} \mid M_{i_m}, \Pa(M_{i_m}, \g)] - \E[M_{i_m} M_{i_r} \mid \Pa(M_{i_m}, \g)] \right\}
\end{split}
\end{equation*}
depends on $M_i$ non-trivially under $P$. 

Invoking local Markov properties on $\g'$, it is easy to show that 
\begin{equation*}
\E[M_{i_m} M_{i_r} \mid M_i, \Pa(M_i, \g)] \stackrel{\g'}{=} \E[M_{i_m} \mid M_i] \E[M_{i_r} \mid M_i]  = 0.
\end{equation*}
Then, we have
\begin{equation*}
\E[M_{i_m} M_{i_r} \mid M_{i_m}, \Pa(M_{i_m}, \g)] \stackrel{\g'}{=} M_{i_m} \E[M_{i_r} \mid M_i],
\end{equation*}
and 
\begin{equation*}
\E[M_{i_m} M_{i_r} \mid \Pa(M_{i_m}, \g)] \stackrel{\g'}{=} \E[M_{i_m} \mid M_i] \E[M_{i_r} \mid M_i] = 0.
\end{equation*}
Similarly, 
\begin{equation*}
\E[M_{i_m} M_{i_r} \mid M_{i_r}, \Pa(M_{i_r}, \g)] \stackrel{\g'}{=}  M_{i_r} \E[M_{i_m} \mid M_i] = 0,
\end{equation*}
and 
\begin{equation*}
\E[M_{i_m} M_{i_r} \mid \Pa(M_{i_r}, \g)] \stackrel{\g'}{=} \E[M_{i_m} \mid M_i] \E[M_{i_r} \mid M_i] = 0.
\end{equation*}
Finally, for any other child $M_{i_l}$ ($l \neq m,r$), it holds that
\begin{equation*}
\E[M_{i_m} M_{i_r} \mid M_{i_l}, \Pa(M_{i_l}, \g)] - \E[M_{i_m} M_{i_r} \mid \Pa(M_{i_l}, \g)] \stackrel{\g'}{=} 0
\end{equation*}
by $M_{i_m}, M_{i_r}$ being non-descendants of $M_{i_l}$ on $\g'$ and the local Markov property.
Hence, the final sum in $\Gamma(M_i)$ vanishes. Finally, we are left with
\begin{equation*}
\Gamma(M_i) = M_{i_m} \E[M_{i_r} \mid M_i],
\end{equation*}
which can be chosen to depend on $M_i$ non-trivially. Finally, to finesse the fact that $P$ is degenerate, consider a sequence of non-degenerate laws $P_n$ that weakly converges to $P$ in $\model(\g', \Vset)$. Then, $\Gamma(M_i)$ depends on $M_i$ under $P_n$ for a large enough $n$.  
\end{proof}

\subsection{Case (M-b)}

\begin{lemma} \label{lem:M-b}
Under the assumptions of \cref{thm:criteria}, suppose that variable $M_i \in \Mset(\g) \setminus \{Y\}$ satisfies satisfies (M-b) in $\g$. Then, there exists a non-degenerate law $P \in \model(\g, \Vset)$, under which $\Psi_{a,P,\text{eff}}^{1}(\Vset; \g)$ non-trivially depends on $M_i$.
\end{lemma}

\begin{proof}
Since $M_{i_r}$ is a mediator, let $p=\langle M_{i_r}, \dots, Y \rangle$ be the shortest path from $M_{i_r}$ to $Y$. Let $\g'$ be a subgraph of $\g$ that is on same set of vertices but only contains edges $M_i \rightarrow M_{i_r}$, $B_j \rightarrow M_{i_r}$ and all edges on $p$; see \cref{fig:proof-nec-M-ab}.

Choose $P$ that is Markov to $\g'$ such that the following hold almost surely:
\begin{enumerate}
\item $A = 1$.
\item $M_{i_r} = \dots = Y$ along path $p$.
\item $\E[M_{i_r} \mid M_i, B_j] = M_i B_j$.
\end{enumerate}
Hence, $AY / \rho(\Omin) = M_{i_r}$ almost surely.
By \cref{eqs:ast-M}, our goal is to show that
\begin{equation*}
\begin{split}
\Gamma(M_i) &= \E[ M_{i_r} \mid M_i, \Pa(M_i, \g)]  + \E[ M_{i_r} \mid M_{i_r}, \Pa(M_{i_r}, \g)] - \E[ M_{i_r} \mid \Pa(M_{i_r}, \g)] \\
& \quad + \sum_{l \neq r} \left\{\E[ M_{i_r} \mid M_{j_l}, \Pa(M_{j_l}, \g)] - \E[ M_{i_r} \mid \Pa(M_{j_l}, \g)] \right\}
\end{split}
\end{equation*}
depends on $M_i$ non-trivially under $P$. 
By the local Markov property on $\g'$, it is easy to see that
\begin{equation*}
\begin{split}
\E[M_{i_r} \mid M_i, \Pa(M_i, \g)] &= \E\left\{\E[M_{i_r} \mid M_i, B_j, \Pa(M_i, \g)] \mid M_i, \Pa(M_i, \g) \right\}\\
&\stackrel{\g'}{=} \E\left\{\E[M_{i_r} \mid M_i, B_j] \mid M_i, \Pa(M_i, \g) \right\}\\
&= \E[M_i B_j \mid M_i, \Pa(M_i, \g)] \stackrel{\g'}{=} M_i \E[B_j],
\end{split}
\end{equation*}
where in the last step we used the fact that $B_j \notin \Pa(M_i, \g)$. 
We also have
\[\E[M_{i_r} \mid M_{i_r}, \Pa(M_{i_r}, \g)] - \E[M_{i_r} \mid \Pa(M_{i_r}, \g)] \stackrel{\g'}{=} M_{i_r} - \E[M_{i_r} \mid M_i, B_j] = M_{i_r} - M_i B_j\]
Further, for any other child $M_{i_l}$ of $M_i$ ($l \neq r$), 
\begin{equation*}
\E[M_{i_r} \mid M_{i_l}, \Pa(M_{i_l}, \g)] - \E[M_{i_r} \mid \Pa(M_{i_l}, \g)] \stackrel{\g'}{=} 0
\end{equation*}
by $M_{i_r}$ being a non-descendant of $M_{i_l}$ on $\g'$ and the local Markov property.
Finally, we have 
\[ \Gamma(M_i) = M_i (\E[B_j] - B_j) + M_{i_r} ,\]
which depends on $M_i$ whenever $B_j$ is not a constant. Finally, to finesse the fact that $P$ is degenerate, consider a sequence of non-degenerate laws $P_n$ that weakly converges to $P$ in $\model(\g', \Vset)$. Then, $\Gamma(M_i)$ depends on $M_i$ under $P_n$ for a large enough $n$. 
\end{proof}

\subsection{Case (M-c)}

\begin{lemma} \label{lem:M-c}
Under the assumptions of \cref{thm:criteria}, suppose that variable $M_i \in \Mset(\g) \setminus \{Y\}$ satisfies (M-c), but neither (M-a) nor (M-b) in $\g$. Then, there exists a non-degenerate law $P \in \model(\g, \Vset)$, under which $\Psi_{a,P,\text{eff}}^{1}(\Vset; \g)$ non-trivially depends on $M_i$.
\end{lemma}

\smallskip \begin{proof}
Let path $p = \langle B_j, \dots, S\rangle$ and $q: M_{i_r} \rightarrow \dots \rightarrow Y$ be chosen according to \cref{lem:cond2-m-graph}. Then, depending on whether $p$ and $q$ intersect, and if so how they intersect, they are 4 further subcases.

\begin{enumerate}[label=(M-c\arabic*)]
\item \label{case:M-c1}  No vertex in $\Ch(M_i, \g) \setminus \{M_{i_r}\}$ is on either $p$ or $q$. Further, there is no vertex that is on both $p$ and $q$.

Then, necessarily, $S \equiv A$ or $S \equiv O_1 \in \Omin$. This also entails the special case when $B_j \in \{A\} \cup  \Omin$ with $|p|=0$.

\item \label{case:M-c2} No vertex in $\Ch(M_i, \g) \setminus \{M_{i_r}\}$ is on either $p$ or $q$, but there is a vertex that is on both $p$ and $q$. 

\item \label{case:M-c3} A vertex in $\Ch(M_i, \g) \setminus \{M_{i_r}\}$ is on $p$. 

\item \label{case:M-c4} A vertex in $\Ch(M_i, \g) \setminus \{M_{i_r}\}$ is on $q$. 
\end{enumerate}

The result is established under each case: \cref{lem:M-c1} proves \ref{case:M-c1}, \cref{lem:M-c2} proves \ref{case:M-c2}, and \cref{lem:M-c34} proves \ref{case:M-c3} and \ref{case:M-c4}.
\end{proof}

\begin{figure}[!htb]
\centering
\begin{tikzpicture}
\tikzset{rv/.style={circle,inner sep=1pt,fill=gray!20,draw,font=\sffamily}, 
redv/.style={circle,inner sep=1pt,fill=white,draw,dashed,font=\sffamily}, 
ov/.style={circle,inner sep=1pt,fill=gray!20,draw=red,thick,font=\sffamily}, 
sv/.style={circle,inner sep=1pt,fill=gray!20,draw,font=\sffamily,minimum size=1mm}, 
node distance=12mm, >=stealth, every node/.style={scale=0.8}}
\begin{scope} \node[rv] (Bi) {$B_j$};
\node[rv, right of=Bi] (C1) {$C_1$};
\node[right of=C1] (E1) {$\dots$};
\node[rv, right of=E1] (C2) {$C_2$};
\node[right of=C2] (E2) {$\dots$};
\node[rv, right of=E2] (C3) {$C_3$};
\node[rv, right of=C3] (O1) {$A$};
\draw[->, very thick, color=brown] (Bi) -- (C1); 
\draw[<-, very thick, color=brown] (C1) -- (E1);
\draw[->, very thick, color=brown] (E1) -- (C2);
\draw[<-, very thick, color=brown] (C2) -- (E2); 
\draw[->, very thick, color=brown] (E2) -- (C3);
\draw[<-, very thick, color=brown] (C3) -- (O1);
\node[below of=C1] (E3) {$\vdots$};
\node[sv, below of=E3] (F1) {};
\node[rv, below of=F1] (Mj) {$M_i$};
\draw[->, very thick, color=blue] (C1) -- (E3);
\draw[->, very thick, color=blue] (E3) -- (F1);
\draw[->, very thick, color=blue] (F1) -- (Mj);
\node[below of=E1, xshift=-4mm, yshift=-4mm] (E4) {\reflectbox{$\ddots$}};
\node[sv, below of=E1, xshift=4mm, yshift=4mm] (F2) {};
\draw[->, very thick, color=blue] (C2) -- (F2);
\draw[->, very thick, color=blue] (F2) -- (E4);
\draw[->, very thick, color=blue] (E4) -- (F1);
\node[below of=E2, yshift=-3mm, xshift=-4mm] (E6) {$\dots$};
\draw[->, very thick, color=blue, bend left] (C3) to (E6);
\draw[->, very thick, color=blue, out=180, in=-45] (E6) to (F2);
\draw[->, very thick, color=blue] (Bi) to (Mj);
\node[rv, below of=Mj] (Mjr) {$M_{i_r}$};
\node[right of=Mjr] (E7) {$\dots$};
\node[rv, right of=E7] (Y) {$Y$};
\draw[->, very thick, color=blue] (Mj) to (Mjr);
\draw[->, very thick, color=red] (Mjr) to (E7);
\draw[->, very thick, color=red] (E7) to (Y);
\node[below of=Y] {(a) $S \equiv A$};
\end{scope}
\begin{scope}[xshift=7cm] \node[rv] (Bi) {$B_j$};
\node[rv, right of=Bi] (C1) {$C_1$};
\node[right of=C1] (E1) {$\dots$};
\node[rv, right of=E1] (C2) {$C_2$};
\node[right of=C2] (E2) {$\dots$};
\node[rv, right of=E2] (C3) {$C_3$};
\node[rv, right of=C3] (O1) {$O_1$};
\draw[->, very thick, color=brown] (Bi) -- (C1); 
\draw[<-, very thick, color=brown] (C1) -- (E1);
\draw[->, very thick, color=brown] (E1) -- (C2);
\draw[<-, very thick, color=brown] (C2) -- (E2); 
\draw[->, very thick, color=brown] (E2) -- (C3);
\draw[<-, very thick, color=brown] (C3) -- (O1);
\node[below of=C1] (E3) {$\vdots$};
\node[sv, below of=E3] (F1) {};
\node[rv, below of=F1] (Mj) {$M_i$};
\draw[->, very thick, color=blue] (C1) -- (E3);
\draw[->, very thick, color=blue] (E3) -- (F1);
\draw[->, very thick, color=blue] (F1) -- (Mj);
\node[below of=E1, xshift=-4mm, yshift=-4mm] (E4) {\reflectbox{$\ddots$}};
\node[sv, below of=E1, xshift=4mm, yshift=4mm] (F2) {};
\draw[->, very thick, color=blue] (C2) -- (F2);
\draw[->, very thick, color=blue] (F2) -- (E4);
\draw[->, very thick, color=blue] (E4) -- (F1);
\node[below of=E2, yshift=-3mm, xshift=-4mm] (E6) {$\dots$};
\draw[->, very thick, color=blue, bend left] (C3) to (E6);
\draw[->, very thick, color=blue, out=180, in=-45] (E6) to (F2);
\draw[->, very thick, color=blue] (Bi) to (Mj);
\node[rv, below of=Mj] (Mjr) {$M_{i_r}$};
\node[right of=Mjr] (E7) {$\dots$};
\node[rv, right of=E7] (Y) {$Y$};
\draw[->, very thick, color=blue] (Mj) to (Mjr);
\draw[->, very thick, color=red] (Mjr) to (E7);
\draw[->, very thick, color=red] (E7) to (Y);
\node[below of=Y] {(b) $S \equiv O_1 \in \Omin$};
\end{scope}
\end{tikzpicture}
\caption{Case \ref{case:M-c1}: path $p$ ({\color{brown} $\bm{-}$}), path $q$ ({\color{red} $\bm{-}$})}
\label{fig:proof-nec-M-c1}
\end{figure}

\medskip \begin{lemma} \label{lem:M-c1}
Under the assumptions of \cref{thm:criteria}, suppose that variable $M_i \in \Mset(\g) \setminus \{Y\}$ satisfies \ref{case:M-c1}, but neither (M-a) nor (M-b) in $\g$. Then, there exists a non-degenerate law $P \in \model(\g, \Vset)$, under which $\Psi_{a,P,\text{eff}}^{1}(\Vset; \g)$ non-trivially depends on $M_i$.
\end{lemma}

\smallskip \begin{proof}
There are two cases.
\begin{enumerate}
\item $S \equiv A$ (\cref{fig:proof-nec-M-c1}(a)): Let $P$ be chosen such that $\rho(\Omin) = c$ for constant $c \in (0,1)$. Then the proof follows from that of \cref{lem:W-c1} for \ref{case:W-c1} by  (i) replacing $W_i, W_j, W_{j_k}$ with $B_j, M_i, M_{i_r}$, and (ii) replacing $O_1, O_2$ with $A, Y$.

\item $S \equiv O_1 \in \Omin$ (\cref{fig:proof-nec-M-c1}(b)): We shall prove the dependency on $M_i$ under a law $P$ that is Markov to a subgraph $\g'$ of $\g$. Let $\g'$ be chosen as a subgraph of $\g$ on the same set of vertices, but only with edges on the following paths:
\begin{enumerate}[(1)]
\item $B_j \rightarrow M_i \rightarrow M_{i_r}$,
\item $p = \langle B_j, \dots, O_1 \rangle$,
\item $q = \langle M_{i_r}, \dots, Y \rangle$,
\item when $p$ contains colliders, say $C_1, \dots, C_F$ ($F \geq 1$), then for each collider $C_f$ also include the shortest causal path $c_f$ from $C_f$ to $W_j$, which exist on $\g$ by \cref{lem:cond2-m-graph} \ref{case:p4},
\item also a path between $O_1$ and $A$ (omitted from \cref{fig:proof-nec-M-c1}(b)).
\end{enumerate}

Let $P$ be chosen such that the following hold almost surely:
\begin{enumerate}[(i)]
\item $\E[M_{i_r} \mid M_i] = 0$,
\item $M_{i_r} = \dots = Y$ ($q$ is an identity path),
\item $\rho(\Omin) = \rho(O_1)$ for some function $\rho$.
\end{enumerate}
It follows that $AY / \rho(\Omin) = AM_{i_r} / \rho(O_1)$ almost surely.
Rewriting \cref{eqs:ast-M}, we shall prove that 
\begin{multline*}
\Gamma(M_i) = \E[AM_{i_r} / \rho(O_1) \mid M_i, \Pa(M_i, \g)] \\
 + \E[AM_{i_r} / \rho(O_1) \mid M_{i_r}, \Pa(M_{i_r}, \g)] - \E[AM_{i_r} / \rho(O_1) \mid \Pa(M_{i_r}, \g)]\\
+ \sum_{l \neq r} \left\{ \E[AM_{i_r} / \rho(O_1) \mid M_{i_l}, \Pa(M_{i_l}, \g)] - \E[AM_{i_r} / \rho(O_1) \mid \Pa(M_{i_l}, \g)] \right \}
\end{multline*}
depends on $M_i$ under $P$. Invoking local Markov properties on $\g'$, with a similar argument to that of \cref{lem:W-c1}, one can show that 
\[ \E[AM_{i_r} / \rho(O_1) \mid M_i, \Pa(M_i, \g)] =  \E[AM_{i_r} / \rho(O_1) \mid \Pa(M_{i_r}, \g)] = 0\]
and 
\[ \E[AM_{i_r} / \rho(O_1) \mid M_{i_l}, \Pa(M_{i_l}, \g)] - \E[AM_{i_r} / \rho(O_1) \mid \Pa(M_{i_l}, \g)] = 0, \quad l \neq r.\]
Hence, we are left with
\begin{equation*}
\begin{split}
\Gamma(M_i) &= \E[AM_{i_r} / \rho(O_1) \mid M_{i_r}, \Pa(M_{j_r}, \g)] \\
&= M_{i_r} \E[A / \rho(O_1) \mid M_{i_r}, \Pa(M_{i_r}, \g)] \\
&= M_{i_r} \E[A / \rho(O_1) \mid \Pa(M_{i_r}, \g)] \\
&= M_{i_r} \E[A \mid \Pa(M_{i_r}, \g)] \E[1 / \rho(O_1) \mid M_i, A=1, \Pa(M_{i_r}, \g) \setminus \{M_i\}],
\end{split}
\end{equation*}
where the third step follows from $A, O_1$ being non-descendants of $M_{i_r}$ on $\g'$ and the local Markov property. Note that path $\langle M_i, B_j \rangle \oplus p$ is an inducing path between $M_i$ and $O_1$ with respect to $\{A\}\cup \Pa(M_{i_r}, \g) \setminus \{M_i\}$:
\begin{enumerate}
\item $B_j$ is a non-collider and $B_j \notin \Pa(M_{i_r}, \g)$ by (M-c); $A$ is not on path $p$ since otherwise a shorter $p$ can be chosen for $S \equiv A$; no non-collider on $p$ is in $\Pa(M_{i_r}, \g)$ by (M-c). These conditions holds on $\g$ and hence also on $\g'$.
\item Every collider on $p$ is an ancestor of $M_i$ in $\g'$.
\end{enumerate}
Hence, by \cref{lem:inducing-path} and strong completeness of d-separations \citep{meek1995strong}, $\E[1 / \rho(O_1) \mid M_i, A=1, \Pa(M_{i_r}, \g) \setminus \{M_i\}]$ depends on $M_i$ for some choice of $\rho(\cdot)$. 
Finally, to finesse the fact that $P$ is degenerate, consider a sequence of non-degenerate laws $P_n$ that weakly converges to $P$ in $\model(\g', \Vset)$. Then, $\Gamma(M_i)$ depends on $M_i$ under $P_n$ for a large enough $n$. 
\end{enumerate}
\end{proof}

\medskip \begin{lemma} \label{lem:M-c2}
Under the assumptions of \cref{thm:criteria}, suppose that variable $M_i \in \Mset(\g) \setminus \{Y\}$ satisfies \ref{case:M-c2}, but neither (M-a) nor (M-b) in $\g$. Then, there exists a non-degenerate law $P \in \model(\g, \Vset)$, under which $\Psi_{a,P,\text{eff}}^{1}(\Vset; \g)$ non-trivially depends on $M_i$.
\end{lemma}

\begin{proof}
In this case, $p$ and $q$ will merge and eventually lead to $Y$. Choose $P$ such that $A=1$ almost surely, under which $AY/\rho(\Omin) = Y$. Then the proof goes similarly to that of \cref{lem:W-c2} for \ref{case:W-c2}, where $O_1$ is replaced by $Y$. 
\end{proof}

\medskip \begin{lemma} \label{lem:M-c34}
Under the assumptions of \cref{thm:criteria}, suppose that variable $M_i \in \Mset(\g) \setminus \{Y\}$ satisfies either \ref{case:M-c3} or \ref{case:M-c4}, but neither (M-a) nor (M-b) in $\g$. Then, there exists a non-degenerate law $P \in \model(\g, \Vset)$, under which $\Psi_{a,P,\text{eff}}^{1}(\Vset; \g)$ non-trivially depends on $M_i$
\end{lemma}

\begin{proof}
Path $p$ (under \ref{case:M-c3}) or $q$ (under \ref{case:M-c4}) eventually leads to $Y$. Choose $P$ such that $A=1$ almost surely, under which $AY/\rho(\Omin) = Y$. Then, the proof of \cref{lem:W-c34} for \ref{case:W-c3} and \ref{case:W-c4} can be adapted by replacing $O_1$ or $O_2$ by $Y$.
\end{proof}

\section{Auxiliary graphical results for completeness proofs of W- and M-criterion} \label{sec:app-aug-graph}
In this section, we establish certain graphical configurations should a vertex in $\Wset$ or $\Mset$ fail the corresponding criterion. Then, these configurations are exploited in \cref{sec:app-W-complete,sec:app-M-complete} for proving the completeness of W-criterion and M-criterion. 
The following additional notations are used. For a vertex $A$ in graph $\g$, $\Adj(A,\g) \equiv \Pa(A,\g) \cup \Ch(A, \g)$.
For a path $p = \langle V_1, \dots , V_k \rangle, k > 1$, $p(V_i, V_j), 1 \le i < j \le k$ denotes the subpath $\langle V_i, \dots , V_j \rangle $ of $p$ consisting of exactly the same sequence of vertices as $p$ on the segment between $V_i$ and $V_j$.

\medskip \begin{lemma}\label{lem:cond1-w-graph}
Suppose that $\g$ satisfies \cref{assump:A-goes-to-Y}. Let  $(W_1, \dots, W_J)$, $J \ge 1$ be a topological ordering of $\Wset$ in $\g$.
Suppose that $W_j \in \Wset \setminus \Oset, j \in \{1, \dots , J\}$ fails the W-criterion (\cref{lem:W-criterion}) and let $(W_{j_1}, \dots, W_{j_r})$ be a topological ordering of $\Ch(W_j, \g) \cap \Wset$ in $\g$.
Then one of the following graphical configurations holds in $\g$:

\begin{enumerate}[(a)]
\item\label{case:a1} $W_{j_s}$ and $W_{j_k}$ are not adjacent in $\g$, for some $k,s \in \{1, \dots, r\}, k \neq s$.
\item\label{case:b1} $W_{j} \to W_{j_k} \leftarrow W_i$, $i \neq j$, $k \in  \{1, \dots, r\}$, and $W_j \notin \Adj(W_{i}, \g)$.
\item\label{case:c1} $W_{i} \to  W_{j} \to W_{j_k}$, $k \in \{1, \dots r\}$, $W_i \in \Pa(W_j,\g) \setminus \Pa(W_{j_k},\g)$, and there is a path $p = \langle W_i, \dots , O' \rangle$, $O' \in \Oset \setminus \Pa(W_{j_k},\g)$  that is d-connecting given $\Pa(W_{j_k},\g)$. If $W_i \in \Oset$, then $W_i \equiv O'$ and $|p|=0$.  
\end{enumerate}
\end{lemma}

\begin{proof}
Let $U = S = \Oset$, and $D = \Wset$. By \cref{assump:A-goes-to-Y}, $D \neq \emptyset$. Additionally, $\An(U, \g) = \An(\Oset, \g) \subseteq \An(Y,\g)$, and by \cref{lemma:basicpropertyWandM}, for all $W' \in \Wset$, $\Ch(W',\g) \cap \An(\Oset,\g) \subseteq \Wset$.
Similarly,$S = \Oset \subseteq \An(\Oset,\g) = \An(U, \g)$ and by  \cref{lem:osetandw}, $\De(W',\g) \cap \Oset \neq \emptyset$. Hence, our choice of $U, S$, and $D$ sets satisfies the properties required by 
\cref{lemma:necessarygraphicalEIFmerge}. The result then follows by \cref{lemma:necessarygraphicalEIFmerge}, while noting that $W_i \in \Wset$ because by Lemma \ref{lem:osetandw} $\An(\Wset, \g) = \Wset$ and $W_i$ is a parent of a node in $\Wset$. 
\end{proof}

\medskip \begin{lemma}\label{lem:cond2-w-graph}
Suppose that $\g$ satisfies \cref{assump:A-goes-to-Y}.  Let  $(W_1, \dots, W_J)$, $J \ge 1$ be a topological ordering of $\Wset$ in $\g$.
Suppose that $W_j \in \Wset \setminus \Oset$ does not satisfy cases \ref{case:a1} or \ref{case:b1}, but does satisfy case \ref{case:c1} of \cref{lem:cond1-w-graph} and let $(W_{j_1}, \dots, W_{j_r}), r \ge 1$ be a topological ordering of $\Ch(W_{j}, \g) \cap \Wset $ in $\g$. 

Let $k \in \{1, \dots, r\}$ be chosen as the largest index such that $\Pa(W_{j},\g) \setminus \Pa(W_{j_k}, \g) \not \indep_{\g} \Oset \setminus \Pa(W_{j_k},\g)|\Pa(W_{j_k},\g)$. 
Let path $q = \langle W_{j_k}, \dots , O_2\rangle$, $O_2 \in \Oset$ be chosen as a shortest causal path from $W_{j_k}$ to $\Oset$. If $W_{j_k} \in \Oset$, then $W_{j_k} \equiv O_2$ and $|q| = 0$.

Let $p = \langle W_i, \dots, O_1 \rangle$, $W_i \in  \Pa(W_{j},\g) \setminus \Pa(W_{j_k}, \g)$, $O_1 \in \Oset \setminus \Pa(W_{j_k},\g)$   be chosen as a shortest among all paths from $\Pa(W_{j},\g) \setminus \Pa(W_{j_k}, \g)$ to  $\Oset \setminus \Pa(W_{j_k},\g)$ that have a shortest distance-to-$\Pa(W_{j_k},\g)$.  If $W_i \in \Oset$, then $W_i \equiv O_1$ and $|p| = 0$.

Then paths $p$ and $q$ satisfy the following:

\begin{enumerate}[(i)]
\item All vertices on $q$ are in $\Wset$.
\item \label{case:p1} The only vertex in $\Oset$ that is on $p$ is $O_1$.
\item \label{case:q1} The only vertex in $\Oset$ that is on $q$ is $O_2$.
\item \label{case:q2} $q$ does not contain any vertices in $\Pa(W_j, \g) \cup \{W_j\}$.
\item \label{case:q3} If a vertex on $q$ is in $(\Ch(W_j,\g) \cap \Wset ) \setminus \{W_{j_k}\}$, then $|q| \ge 1$, $q = \langle W_{j_k}, W_{q_2}, \dots,  O_2 \rangle$, and the only vertex on $q$ in $(\Ch(W_j,\g) \cap \Wset ) \setminus \{W_{j_k}\}$ is $W_{q_2}$, and $W_{q_2} \in \Ch(W_{j_k},\g) \cap \Wset $.
\item \label{case:p2} if a vertex on $p$ is in $\Ch(W_j,\g) \cap \Wset$, then $|p| \ge 1$, $p$ is of the form $W_i \to W_{p_2} \to \dots O_1$, and the only vertex on $p$ in $\Ch(W_{j},\g) \cap \Wset$ is $W_{p_2}$ and $W_{p_2} \in \Ch(W_{j_k},\g)\cap \Wset$.
\item \label{case:p3} $p$ is d-connecting given $\Pa(W_{j}, \g) \cup \{W_j\} \setminus \{W_i\}$.
\item \label{case:p4} if there is a collider on $p$, then let $\{C_1, \dots, C_F\}$, $F \ge 1$ be the set of all collider on $p$, and let $c_f$ be a shortest path from $C_h$ to $\pa(W_j, \g) \cup \{W_j\}$ in $\g$ for all $f \in \{1, \dots, F\}$. Then
\begin{enumerate}
\item  Vertices from $\Oset$ are not on $c_f$, and
\item $c_f$ does not contain any vertex that is on $q$, and 
\item the only vertex that $p$ and $c_f$ have in common is $C_f$.
\end{enumerate}
\item \label{case:qandp}
\begin{enumerate}[(1)]
\item  If a vertex in $\Ch(W_{j_k},\g) $ is on $p$, or 
\item  if a vertex in $\Ch(W_{j_k},\g)$ is $q$, or 
\item if there is a vertex that is on both $p$ and $q$, then
\end{enumerate}
\begin{enumerate}
\item $W_i \neq W_l \neq W_{j_k}$,  and 
\item  $W_s \to W_{l} \leftarrow W_t$, $t \neq s$ is in $\g$, where $W_s$ is on $p$, $W_t$ is on $q$, and  $W_t \notin \Adj(W_s, \g)$.
\item if $W_l$ is on $p$, then $p(W_l,O_1)$ is a causal path and $O_1 \equiv O_2$.
\end{enumerate}
\end{enumerate}
\end{lemma}

\begin{proof}
Let $U = S = \Oset$, and $D = \Wset$. By \cref{assump:A-goes-to-Y}, $D \neq \emptyset$. Additionally, $\An(U, \g) = \An(\Oset, \g) \subseteq \An(Y,\g)$, and by \cref{lemma:basicpropertyWandM}, for all $W' \in \Wset$, $\Ch(W',\g) \cap \An(\Oset,\g) \subseteq \Wset$.
Similarly,$S = \Oset \subseteq \An(\Oset,\g) = \An(U, \g)$ and by  \cref{lem:osetandw}, $\De(W',\g) \cap \Oset \neq \emptyset$. Hence, our choice of $U, S$, and $D$ sets satisfies the properties required by 
\cref{lemma:necessarygraphicalEIFmerge,lemma:parentsofchildren} and \cref{lemma:condition-merge}. The result then follows by \cref{lemma:necessarygraphicalEIFmerge,lemma:parentsofchildren} and \cref{lemma:condition-merge}, while noting that $W_i, W_s \in \Wset$ because by Lemma \ref{lem:osetandw} $\An(\Wset, \g) = \Wset$ and $W_i,W_s$ are both parents of a node in $\Wset$.
\end{proof}

\medskip \begin{lemma}\label{lem:cond1-m-graph}
Suppose that $\g$ satisfies \cref{assump:A-goes-to-Y}. Let $(M_1, \dots, M_K)$, $K \ge 1$ be a topological ordering of $\Mset$ in $\g$.  Suppose that $M_i \in \Mset \setminus \{Y\}, i \in \{1, \dots, K\}$ fails the M-criterion (\cref{lem:M-criterion}) and let $(M_{i_1}, \dots, M_{i_k})$, $k \ge 1$ be a topological ordering of  $\Ch(M_{i}, \g) \cap \Mset$ in $\g$. Furthermore, let $M_i \equiv M_{i_0}$.
Then one of the following graphical configurations holds in $\g$:

\begin{enumerate}[(M-a)]
\item\label{case:am} $M_{i_m}$ and $M_{i_j}$ are not adjacent in $\g$, for some $m,r \in \{1, \dots, k\}$, $r \neq m$.
\item\label{case:bm} $M_{i} \to M_{i_r} \leftarrow B_j$, $B_j \in \{A\} \cup \Oset \cup \Mset \setminus \{M_i\}$, $r \in  \{1, \dots, k\}$, and $M_i \notin \Adj(B_{j}, \g)$.
\item\label{case:cm} $B_{j} \to  M_{i} \to M_{i_r}$, $B_j \in \in \{A\} \cup \Oset \cup \Mset \setminus \{M_i\}$,, $r \in  \{1, \dots, k\}$, $B_j \notin \Pa(M_{i_r},\g) \cup \{M_{i_r}\}$, and there is a path $p = \langle B_j, \dots , S \rangle$, $S \in \{A,Y\} \cup  \Omin \setminus \Pa(M_{i_k},\g)$  that is d-connecting given $\Pa(M_{i_r},\g)$. If $B_j \in \{A,Y\} \cup  \Omin$, then $B_j \equiv S$ and $|p| = 0$.
\end{enumerate}
\end{lemma}

\begin{proof} Let $U = Y$, $S = \{A,Y\} \cup \Omin$, $D = \Mset$. By \cref{assump:A-goes-to-Y}, $D \neq \emptyset$. Additionally, $\An(U, \g) = \An(Y, \g)$, and by \cref{lemma:basicpropertyWandM}, for all $M' \in \Mset$, we have $\Ch(M',\g) \cap \An(Y,\g) \subseteq \Mset$. Similarly, by \cref{assump:A-goes-to-Y} and by definitions of $\Mset$, $\Omin$, we have that $S = \{A,Y\} \cup \Omin \subseteq \An(Y,\g) = \An(U, \g)$, and  $\De(M',\g) \cap (\{A,Y\} \cup \Omin) = \{Y\} \neq \emptyset$. Hence, our choice of $U, S$, and $D$ sets satisfies the properties required by 
 \cref{lemma:necessarygraphicalEIFmerge}. The result then follows from \cref{lemma:necessarygraphicalEIFmerge}, while noting that $B_j \in \{A\} \cup \Oset \cup \Mset$ by definitions of $\Mset$ and $\Oset$ since $B_j$ is a parent of a node in $\Mset$.
\end{proof}

\medskip \begin{lemma}\label{lem:cond2-m-graph}
Suppose that $\g$ satisfies \cref{assump:A-goes-to-Y}. Let $(M_1, \dots, M_K)$, $K \ge 1$ be a topological ordering of $\Mset$ in $\g$.  
Suppose that $M_i \in \Mset \setminus \{Y\}, i \in \{1, \dots, K\}$ does not satisfy \ref{case:am} or \ref{case:bm}, but does satisfy \ref{case:cm} of  \cref{lem:cond1-m-graph} and let $(M_{i_1}, \dots, M_{i_k})$, $k \ge 1$ be a topological ordering of  $\Ch(M_{i}, \g) \cap \Mset$ in $\g$. Furthermore, let $M_i \equiv M_{i_0}$.

Let $r \in \{1, \dots, k\}$ be chosen as the largest index such that $\Pa(M_{i},\g) \setminus \Pa(M_{i_r}, \g) \not \indep_{\g} ( \Omin \cup \{A,Y\} )\setminus \Pa(M_{i_r},\g)|\Pa(M_{i_r},\g)$.  Let path $q = \langle M_{i_r}, \dots , Y\rangle$,  be chosen as a shortest causal path from $M_{i_r}$ to $Y$.  Possibly $M_{i_r} \equiv Y$  and $|q| = 0$.

Let $p = \langle B_j, \dots, S \rangle$, $B_j \in  \Pa(M_{i},\g) \setminus \Pa(M_{i_r}, \g)$, $S \in  (\Omin \cup \{A,Y\} )\setminus \Pa(M_{i_r},\g)$  be chosen as a shortest among all paths from $\Pa(M_{i},\g) \setminus \Pa(M_{i_r}, \g)$ to  $(\Omin \cup \{A,Y\} )\setminus \Pa(M_{i_r},\g)$ that have a shortest distance-to-$\Pa(M_{i_r},\g)$.  Note that $B_j \neq Y$, but it is possible that $B_j \in \Omin \cup \{A\} $, then $B_j \equiv S$ and $|p| = 0$.

Then there are paths $p$ and $q$  in $\g$ that satisfy the following:
\begin{enumerate}[(i)]
\item All vertices on $q$ are in $\Mset$.
\item \label{case:p1m} The only vertex in $\Omin \cup \{A,Y\} $ that is on $p$ is $S$.
\item \label{case:q1m} There is no vertex on $q$ that is in $\Omin \cup \{A\}$.
\item \label{case:q2m} $q$ does not contain any vertices in $\Pa(M_i, \g) \cup \{M_i\}$.
\item \label{case:q3m} If a vertex on $q$ is in $(\Ch(M_i,\g) \cap \Mset)  \setminus \{M_{i_r}\}$, then $|q| \ge 1$,  $q = \langle M_{i_r}, M_{q_2}, \dots,  Y \rangle$ and the only vertex on $q$ that is in  $(\Ch(M_i,\g) \cap \Mset)  \setminus \{M_{i_r}\}$ is $M_{q_2}$ and $M_{q_2} \in \Ch(M_{i_r},\g) \cap \Mset$.
\item \label{case:p2m} If  a vertex on $p$ is in $\Ch(M_i,\g) \cap \Mset$, then $|p| \ge 1$, and $p$ is of the form $B_j \to B_{p_2} \to \dots \to Y$, that is, $S=Y$. Furthermore, the only vertex on $p$ that is in   $\Ch(M_i,\g) \cap \Mset$ is $B_{p_2}$ and $B_{p_2} \in \Ch(M_{i_r},\g) \cap \Mset$.
\item \label{case:p3m} $p$ is d-connecting given $(\Pa(M_{i}, \g) \cup \{M_i\}) \setminus \{B_j\}$.
\item \label{case:p4m} if there is a collider on $p$, then let $\{C_1, \dots, C_F\}$, $F \ge 1$ be the set of all collider on $p$, and let $c_f$ be a shortest path from $C_f$ to $\pa(M_i, \g) \cup \{M_i\}$ in $\g$ for all $f \in \{1, \dots, F\}$. Then
\begin{enumerate}
\item  Vertices from $\{A,Y\} \cup \Omin$ are not on $c_f$, and
\item $c_f$ does not contain any vertex that is on $q$, and 
\item the only vertex that $p$ and $c_f$ have in common is $C_f$.
\end{enumerate}
\item \label{case:qandpm} 
\begin{enumerate}[(1)]
\item  If a vertex in $\Ch(M_{i},\g) $ is on $p$, or 
\item  if a vertex in $\Ch(M_{i},\g) \setminus \{M_{i_r}\}$ is $q$, or 
\item if there is a vertex that is on both $p$ and $q$, then
\end{enumerate}
 there exists a vertex  $M_l$ on $p$ or on $q$ such that
\begin{enumerate}
\item $M_j \neq M_l \neq M_{i_r}$,  and 
\item  $B_s \to M_{l} \leftarrow M_t$, $B_s \neq M_t$, is in $\g$, where $B_s$ is on $p$, $M_t \in \Mset$ is on $q$, and $M_t \notin \Adj(B_s, \g)$.
\item if $M_l$ is on $p$, then $p(M_l,S)$ is a causal path and $S \equiv Y$.
\end{enumerate}
\end{enumerate}
\end{lemma}

\begin{proof} Let $U = Y$, $S = \{A,Y\} \cup \Omin$, $D = \Mset$. By \cref{assump:A-goes-to-Y}, $D \neq \emptyset$. Additionally, $\An(U, \g) = \An(Y, \g)$, and by \cref{lemma:basicpropertyWandM}, for all $M' \in \Mset$, we have $\Ch(M',\g) \cap \An(Y,\g) \subseteq \Mset$. Similarly, by \cref{assump:A-goes-to-Y} and by definitions of $\Mset$, $\Omin$, we have that $S = \{A,Y\} \cup \Omin \subseteq \An(Y,\g) = \An(U, \g)$, and  $\De(M',\g) \cap (\{A,Y\} \cup \Omin) = \{Y\} \neq \emptyset$. Hence, our choice of $U, S$, and $D$ sets satisfies the properties required by 
 \cref{lemma:necessarygraphicalEIFmerge,lemma:parentsofchildren} and \cref{lemma:condition-merge}. The result then follows from \cref{lemma:necessarygraphicalEIFmerge,lemma:parentsofchildren} and \cref{lemma:condition-merge}, while noting that $B_j \in \{A\} \cup \Oset \cup \Mset$ by definitions of $\Mset$ and $\Oset$ since $B_j$ is a parent of a node in $\Mset$. 
\end{proof}

\subsection{General Results}

To prove the results in this section we additionally rely on \cref{lemma:basicpropertyWandM} and \cref{def:distto}.

\medskip \begin{lemma}\label{lemma:basicpropertyWandM}
Suppose that $\g$ satisfies \cref{assump:A-goes-to-Y}. 
\begin{enumerate}[(i)]
\item Let $M_i \in \Mset$. Then $\Ch(M_i ,\g) \cap \An(Y, \g) \subseteq \Mset$ and $Y \in \De(M_i, \g)$.
\item Let $W_j \in \Wset$. Then $\Ch(W_{j},\g) \cap \An(\Oset, \g)  \subseteq \Wset$ and $\Oset \cap \De(W_j,\g) \neq \emptyset$.
\end{enumerate}
\end{lemma}

\begin{proof}
Follows directly from definitions of $\Wset, \Mset$, and $\Oset$ and by \cref{lem:osetandw}.
\end{proof}

\medskip \begin{definition}[c.f.\ \cite{zhang2006causal}]\label{def:distto} Let $\g$ be a directed acyclic graph and $A,B$ and $D$ pairwise disjoint vertex sets in $\g$ such that $A \not\indep_{\g} B | D$. 
For any path $p$ from $A$ to $B$ that is d-connecting given $D$ in $\g$ we define distance-to-$D$ in the following way:
\begin{enumerate}
\item If there are no colliders on $p$ then distance-to-$D$ of $p$ is zero.
\item If there are colliders on $p$, let $\{C_1, \dots C_H\}$ be the set of all collider on $p$ and let $c_h$ be a shortest causal paths from $C_h$ to $D$ in $\g$ for all $h \in \{1, \dots , H\}$. If $C_h \in D$, then $c_h$ is of length zero. 
The distance-to-$D$ of $p$ is then equal to $\sum_{h=1}^{H} (|c_h| +1) =\sum_{h=1}^{H} |c_h| +  H$. 
\end{enumerate}
\end{definition}

We can now introduce the general results \cref{lemma:necessarygraphicalEIFmerge}, \cref{lemma:parentsofchildren}, and \cref{lemma:condition-merge} from which \cref{lem:cond1-w-graph}, \cref{lem:cond2-w-graph}, \cref{lem:cond1-m-graph}, \cref{lem:cond2-m-graph} are derived.

\medskip \begin{lemma}\label{lemma:necessarygraphicalEIFmerge}
Let $Y$ be a vertex in a directed acyclic graph $\g$ and let $U \subseteq \An(Y,\g)$. Furthermore, let $(D_1, \dots, D_E)$, $E \ge 1$ be a topological ordering of a vertex set $D$ in $\g$. Suppose also that for all $D' \in D$, $\Ch(D', \g) \cap \An(U,\g) \subset D$.  Let $S$, $S \subseteq \An(U, \g)$, be a set such that for all $D' \in D$, $\De(D', \g) \cap S \neq \emptyset$.  
Let $D_e \in D \setminus U$, $e \in \{1, \dots, E\}$ and suppose $\Ch(D_e, \g) \cap D =  \{D_{e_1}, \dots, D_{e_f}\}$, $f \ge 1$  is indexed topologically in $\g$. Furthermore, let $D_e \equiv D_{e_0}$.

\begin{enumerate}[(i)]
\item\label{cond0:necessary0merge} If  $D_{e} \not \indep_{\g} S | \pa(D_{e_f},\g) \cup \{D_{e_f}\} \setminus \{D_e \}$, or 
\item\label{cond1:allnecessary} if there exists a $t \in \{1, \dots, f\}$ such that 
\begin{enumerate}[(1)]
\item\label{cond1:necessary1merge}  $D_{e_{t-1}}  \not\in \pa(D_{e_t}, \g) $, or
\item\label{cond1:necessary2merge}  $\pa(D_{e_t}, \g)   \not\subseteq \pa(D_{e_{t-1}}, \g) \cup \{D_{e_{t-1}}\}$, or
\item\label{cond1:necessary3merge}  $\pa(D_{e_{t-1}}, \g) \setminus \pa(D_{e_t} ,\g) \not \indep_{\g} S \  | \  \pa(D_{e_t}, \g)$.
\end{enumerate}
\end{enumerate}
then one of the following graphical configurations holds in $\g$:

\begin{enumerate}[(a)]
\item\label{case:a} $D_{e_m}$ and $D_{e_h}$ are not adjacent in $\g$, for some $h,m \in \{1, \dots, r\}$, $h \neq m$.
\item\label{case:b} $D_{e} \to D_{e_h} \leftarrow B_i$, $h \in  \{1, \dots, f\}$, $B_i \neq D_{e}$, and $D_e \notin \Adj(B_{i}, \g)$.
\item\label{case:c}  $B_{i} \to  D_{e} \to D_{e_h}$, $h \in  \{1, \dots, f\}$,  $B_i \in \Pa(D_{e},\g) \setminus \Pa(D_{e_h},\g)$, and there is a path $p = \langle B_i, \dots , S' \rangle$, $S' \in S$  that is d-connecting given $\Pa(D_{e_h},\g)$.  If $B_i \in S$, then $B_i \equiv S'$ and $|p| =0$.
\end{enumerate}
\end{lemma}

\begin{proof}

\ref{cond0:necessary0merge} 
Let $p = \langle D_e ,\dots, S'\rangle $, $S' \in S \setminus \pa(D_{e_f},\g) \cup \{D_{e_f}\}$ be a shortest path from $D_{e}$ to $S$ that is d-connecting given $\pa(D_{e_f},\g) \cup \{D_{e_f}\} \setminus \{D_e \}$.

Suppose that $p$ starts with an edge $D_{e} \to B$. We first show that $B \in D$. Note that $p$ is either a causal path to $S'$, or $B$ is an ancestor of a collider on $p$. Since $S' \in \An(U,\g)$, and a collider on $p$ would be in  $\An(D_{e_f},\g) \subseteq \An(U,\g)$, in both cases $B \in \An(U, \g)$. 
Therefore, $B \in  \Ch(D_e, \g) \cap \An(U,\g)$, so by choice of set $D$, $B \in D$. 

Hence, let $B = D_{e_t}$ for some $t \in \{1, \dots , f-1\}$. If $D_{e_t} \notin \Pa(D_{e_f}, \g)$, then by the topological ordering of $\Ch(D_e,\g)\cap D$, $D_{e_t} \notin \Adj(D_{e_f},\g)$ and we are in case \ref{case:a} with $h=t$. 

If  $D_{e_t} \in \Pa(D_{e_f}, \g)$, then $D_{e_{t}}$ is a collider on $p$, that is $p$ is of the form $D_{e} \to D_{e_{t}} \leftarrow B_{l}$, and $B_{l} \notin \Pa(D_{e_f},\g) \cup \{D_{e_f}\}$ for $p$ to be d-connecting given $\Pa(D_{e_f},\g) \cup \{D_{e_f}\}  \setminus \{D_e\}$.  However, $B_l \in \An(\Pa(D_{e_f},\g) \cup \{D_{e_f}\} ,\g)$ because $D_e \rightarrow D_{e_t} \leftarrow B_l$ is on $p$. 
Now, it follows that there cannot be an edge between $D_e$ and $B_l$ in $\g$, since otherwise, we can choose the path made up of edge between $D_e$ and $B_l$ and subpath $p(B_l,S')$
as a path that is shorter than $p$ and d-connecting given $\Pa(D_{e_f},\g) \cup \{D_{e_f}\} \setminus \{D_e\}$. Hence, we are in case \ref{case:b}, with $i = l$ and $h =t$.

Lastly, suppose that $p$ starts with an edge $D_{e} \leftarrow B_l$ in $\g$. Then $B_l \in \Pa(D_e,\g) \setminus \Pa(D_{e_f},\g)$.  We then only need to show that $p$ is d-connecting given $ \Pa(D_{e_f},\g) $ for us to be in case \ref{case:c}.

Note that since $p(B_l, S')$ is d-connecting given  $\Pa(D_{e_f},\g) \cup \{D_{e_f} \}$, $D_{e_f}$ is not a non-collider on $p$. Additionally, $D_{e_f}$ cannot be a collider on $p(B_l, S')$, since that would imply that a non-collider on $p(B_l, S')$ is in $\Pa(D_{e_f},\g)$ (due to  $B_l \notin \Pa(D_{e_f},\g) \cup \{D_{e_f}\}$).  Therefore $p(B_l, S')$ is d-connecting given $\Pa(D_{e_f},\g)$ and we are in case \ref{case:c}, with $i =l$.

\ref{cond1:allnecessary}: \ref{cond1:necessary1merge} Suppose $D_{e_{t-1}}  \not\in \pa(D_{e_t}, \g) $. Note that in this case, $t =1$ is not possible, by assumption. 
Due to the topological ordering of $\Ch(D_e,\g)\cap D$, $D_{e_{t-1}} \notin \Adj(D_{e_t},\g)$, in which case we have reached case \ref{case:a}, with $m=t-1$ and $h = t$.

\ref{cond1:allnecessary}:  $\lnot$ \ref{cond1:necessary1merge} $\land$ \ref{cond1:necessary2merge}: There is a $t \in \{1, \dots, k\}$, such that $\pa(D_{e_t}, \g)   \not\subseteq \pa(D_{e_{t-1}}, \g) \cup \{D_{e_{t-1}}\}$ and $D_{e_{s-1}}  \in \pa(D_{e_s}, \g) $, for all $s \in \{1,\dots, f\}$ , since otherwise, we are back in \ref{cond1:necessary1merge}. 
 Therefore, $\pa(D_{e_t}, \g)   \not\subseteq \pa(D_{e_{t-1}}, \g) \cup \{D_{e_{t-1}}\}$, implies that $\pa(D_{e_t}, \g) \setminus \pa(D_{e_{t-1}}, \g) \neq \emptyset$.
 
 Let $B_l \in \pa(D_{e_t}, \g) \setminus ( \pa(D_{e_{t-1}}, \g) \cup \{D_{e_{t-1}}\})$. Since $\De(D_{e_t},\g) \cap S \neq \emptyset$ by assumption, $B_l \in \An(S, \g) \subseteq \An(U,\g)$.
 
Suppose first that $t=1$. Then $B_l \in \pa(D_{e_1}, \g) \setminus ( \pa(D_{e}, \g) \cup \{D_{e}\})$. Note that  $B_l \notin \Ch(D_e, \g)$, since $B_l \in \Ch(D_e, \g) \cap \An(U, \g)$, implies $B_l \in D \cap \Ch(D_e,\g)$ which together with $B_l \to D_{e_1}$ in $\g$ would contradict the topological ordering of $\Ch(D_e, \g) \cap D$. 
Hence, $B_l \notin \Pa(D_e,\g) \cup \Ch(D_e, \g)$ and therefore, $B_l \notin \Adj(D_{e},\g)$. Since additionally, $B_{l} \to D_{e_1} \leftarrow D_{e}$ is in $\g$, we are  in \ref{case:b}, with $B_i=B_l$ and $h=1$. 
 
 For the rest of this case, suppose that $t>1$.
 If $B_l \notin \Adj(D_e, \g)$, then since $B_l \to D_{e_t} \leftarrow D_e$ is in $\g$, we are in \ref{case:b}, with $B_i=B_l$ and $h=t$.
 Otherwise, suppose $B_l \in \Pa(D_e, \g)$. We can use that $B_l \to D_e \to D_{e_{t-1}}$ is in $\g$ and $B_l \notin \Pa(D_{e_{t-1}},\g)$ to conclude that $B_l \notin \Adj(D_{e_{t-1}},\g)$. Since additionally, there exists a path of the form $B_l \to D_{e_t} \to \dots \to S'$,  $S' \in S$ in $\g$ that is d-connecting given $\Pa(D_{e_{t-1}},\g)$, we are in \ref{case:c} with $B_i = B_l$ and $h = t-1$.
 Lastly, suppose that $B_l \in \Ch(D_e, \g)$. Then $B_l \in \Ch(D_e, \g) \cap \An(U, \g)$, so $B_l \in \Ch(D_e,\g) \cap D$ by properties of the set $D$. Then it must be that  $B_l \equiv D_{e_s}$, for some $s \in \{1, \dots , t-3 \}, t>3.$ Since additionally, $D_{e_s} \notin \Adj(D_{e_{t-1}},\g)$, we are in \ref{case:a} with $m=s$ and $h = t-1$.

\ref{cond1:allnecessary}:  $\lnot$ \ref{cond1:necessary1merge} $\land \lnot$ \ref{cond1:necessary2merge} $\land$ \ref{cond1:necessary3merge}: There is a $t \in \{1, \dots, f\}$, such that $\pa(D_{e_{t-1}}, \g) \setminus \pa(D_{e_t} ,\g) \not \indep_{\g}  S \  | \  \pa(D_{e_t}, \g)$, and $D_{e_{s-1}}  \in \pa(D_{e_s}, \g) $ and $\Pa(D_{e_s}, \g) \subseteq \Pa(D_{e_{s-1}},\g) \cup \{D_{e_{s-1}}\}$, for all $s \in \{1, \dots, f\}$. Note that in this case, we have $\pa(D_{e_t},\g) \subset \pa(D_{e_{t-1}},\g) \cup \{D_{e_{t-1}}\} \subseteq \dots \subseteq \pa(D_e, \g) \cup \{D_e, D_{e_1}, \dots , D_{e_{t-1}}\}$.

Let $B_l \in \pa(D_{e_{t-1}}, \g) \setminus \pa(D_{e_t} ,\g).$ Then 
$B_l \in \pa(D_e, \g) \cup \{ D_{e_1}, \dots, D_{e_{t-2}} \}$. Note that $B_l \neq D_e$, because $D_e \in \pa(D_{e_t},\g)$. If $B_l \in \{D_{e_1}, \dots, D_{e_{t-2}} \}$, then since $B_l \notin \Pa(D_{e_t},\g)$ it follows that 
$B_l \notin \Adj(D_{e_t},\g)$, meaning that we are in case \ref{case:a} with $h=t$ and $D_{e_m}= B_l$.

Otherwise, $B_l \in \pa(D_e, \g) \setminus \Pa(D_{e_t},\g)$, meaning that  $B_l \to D_e \to D_{e_{t}}$ is in $\g$ (possibly $t=1$) and since  there is a d-connecting path from $B_{l}$ to $S$ given $\Pa(D_{e_t}, \g)$ we are in case \ref{case:c} with $i = l$ and $h = t$.
\end{proof}

\medskip \begin{lemma}\label{lemma:parentsofchildren}
Let $Y$ be a vertex in a directed acyclic graph $\g$ and let $U \subseteq \An(Y,\g)$. Furthermore, let $(D_1, \dots, D_E)$, $E \ge 1$ be a topological ordering of a vertex set $D$ in $\g$. Suppose also that for all $D' \in D$, $\Ch(D', \g) \cap \An(U,\g) \subset D$.    
Let $D_e \in D \setminus U, e \in \{1, \dots, E\}$ suppose that $D_e$ does not satisfy \ref{case:a} or \ref{case:b} of \cref{lemma:necessarygraphicalEIFmerge}. Suppose further that $\Ch(D_e, \g) \cap D =  \{D_{e_1}, \dots, D_{e_f}\}$, $f \ge 1$  is indexed topologically in $\g$. Furthermore, let $D_e \equiv D_{e_0}$.
Then for every $D_{e_t}$, $t \in \{1, \dots, f\}$ the following hold:
\begin{enumerate}[(i)]
\item\label{1}  if $B \in \Pa(D_{e_i},\g)$, then $B \in \Adj(D_e, \g)$, and
\item\label{2}  for any $i,j \in \{1, \dots, f\}$, such that $i \neq j$,  $D_{e_i} \in \Adj(D_{e_j},\g)$.
\item  \label{case:parents} $ \{D_{e_0}, \dots, D_{e_{t-1}}\} \subseteq \Pa(D_{e_t}, \g) \subseteq \Pa(D_{e}, \g) \cup \{D_{e_0}, \dots, D_{e_{t-1}}\}$.
\item \label{case:children} $\Ch(D_e, \g) \cap D \subseteq \{D_{e_1}, \dots,  D_{e_t}\} \cup \Ch(D_{e_t},\g)$. 
\end{enumerate}
\end{lemma}

\begin{proof}
Cases \ref{1} and \ref{2} follow directly from the fact that $D_e$ does not satisfy  \ref{case:a} or \ref{case:b} of \cref{lemma:necessarygraphicalEIFmerge}.

\ref{case:parents}: Consider claim $\{D_{e_0}, \dots, D_{e_{t-1}}\} \subseteq \Pa(D_{e_t}, \g) $. Note that $D_{e_t} \in \Ch(D_e,\g)$, so $D_{e_0} \in \Pa(D_{e_t},\g)$. If $t =1$, then we are done.
If $t>1$, then note that $D_{e_1}, \dots, D_{e_{t-1}}$ precede $D_{e_t}$ in the topological ordering of $\Ch(D_e,\g)$ in $\g$. Since all pairs of children of $D_e$ in $D$ must be adjacent, by \ref{2},  $D_{e_1}, \dots, D_{e_{t-1}}$ are parents of $D_{e_t}$ in $\g$. 

To prove the rest of case \ref{case:parents}, we only need to show that $\Pa(D_{e_t}, \g) \setminus \{D_{e_0}, \dots, D_{e_{t-1}}\} \subseteq \Pa(D_{e},\g)$. This follows directly from \ref{1}.

To show that case \ref{case:children}  holds, we only need to show that  $(\Ch(D_e, \g) \cap D) \setminus \{D_{e_1}, \dots,  D_{e_t}\} \subseteq \Ch(D_{e_t},\g)$. By \ref{2}, all pairs of children of $D_e$ in $D$ must be adjacent in $\g$. Since all vertices in $(\Ch(D_e,\g)  \cap D) \setminus \{D_{e_1} , \dots, D_{e_t} \}$ are after $D_{e_t}$ in the topological ordering of $(\Ch(D_e,\g) \cap D)$ in $\g$, it follows that $(\Ch(D_e, \g) \cap D) \setminus \{D_{e_1}, \dots,  D_{e_t}\} \subseteq \Ch(D_{e_t},\g)$.
\end{proof}

\medskip \begin{lemma}\label{lemma:condition-merge}
Let $Y$ be a vertex in a directed acyclic graph $\g$ and let $U \subseteq \An(Y,\g)$. Furthermore, let $(D_1, \dots, D_E)$, $E \ge 1$ be a topological ordering of a vertex set $D$ in $\g$. Suppose also that for all $D' \in D$, $\Ch(D', \g) \cap \An(U,\g) \subset D$.     Let $S$, $S \subseteq \An(U, \g)$, be a set such that for all $D' \in D$, $\De(D', \g) \cap S \neq \emptyset$.  

Suppose that for $D_e \in D \setminus U, e \in \{1, \dots, E\}$, $D_e$ does not satisfy \ref{case:a} or \ref{case:b}, but does satisfy \ref{case:c} of  \cref{lemma:necessarygraphicalEIFmerge}. Furthermore,  let $D_e \equiv D_{e_0}$ and suppose that  $\Ch(D_e, \g) \cap D =  \{D_{e_1}, \dots, D_{e_f}\}$, $f \ge 1$  is indexed  topologically in $\g$.

Let $t \in \{1, \dots, f\}$ be chosen as the largest index such that $\Pa(D_{e},\g) \setminus \Pa(D_{e_t}, \g) \not \indep_{\g} S \setminus \Pa(D_{e_t},\g) |\Pa(D_{e_t},\g)$.  Let path $q = \langle D_{e_t}, \dots , S'\rangle$,  be chosen as a shortest causal path from $D_{e_t}$ to $S$.  Possibly $D_{e_t} \equiv Y$  and $|q| = 0$.

Let $p = \langle B_j, \dots, S'' \rangle$, $B_j \in  \Pa(D_{e},\g) \setminus \Pa(D_{e_t}, \g)$, $S'' \in   S\setminus \Pa(D_{e_t},\g)$  be chosen as a shortest among all paths from $\Pa(D_{e},\g) \setminus \Pa(D_{e_t}, \g)$ to  $S \setminus \Pa(D_{e_t},\g)$ that have a shortest distance-to-$\Pa(D_{e_t},\g)$.  If $B_j \equiv S''$ then $|p| = 0$.

Then there are paths $p$ and $q$  in $\g$ that satisfy the following:
\begin{enumerate}[(i)]
\item \label{case:q0d} All vertices on $q$ are in $D$.
\item \label{case:p1d} The only vertex in $S$ that is on $p$ is $S''$.
\item \label{case:q1d}The only vertex in $S$ that is on $q$ is $S'$.
\item \label{case:q2d} $q$ does not contain any vertices in $\Pa(D_e, \g) \cup \{D_e\}$.
\item \label{case:q3d} If a vertex on $q$ is in $(\Ch(D_e,\g) \cap D)  \setminus \{D_{e_t}\}$, then $|q| \ge 1$, $q = \langle D_{e_t}, D_{q_2}, \dots,  S' \rangle$   and the only vertex on $q$ that is in $(\Ch(D_e,\g) \cap D) \setminus D_{e_t}$ is $D_{q_2}$, and  $D_{q_2} \in \Ch(D_{e_t},\g)  \cap D$.
\item \label{case:p2d} If  a vertex on $p$ is in $\Ch(D_e,\g) \cap D$, then $|p| \ge 1$, and $p$ is of the form $B_j \to B_{p_2} \to \dots \to S''$ and the only vertex on $p$ that is in  $\Ch(D_e,\g) \cap D$ is $B_{p_2}$ and $B_{p_2} \in \Ch(D_{e_t},\g) \cap D$.
\item \label{case:p3d} $p$ is d-connecting given $(\Pa(D_{e}, \g) \cup \{D_e\}) \setminus \{B_j\}$.
\item \label{case:p4d} if there is a collider on $p$, then let $\{C_1, \dots, C_H\}$, $H \ge 1$ be the set of all collider on $p$, and let $c_h$ be a shortest path from $C_h$ to $\pa(D_e, \g) \cup \{D_e\}$ in $\g$ for all $h \in \{1, \dots, H\}$. Then
\begin{enumerate}
\item  Vertices from $S$ are not on $c_h$, and
\item $c_h$ does not contain any vertex that is on $q$, and 
\item the only vertex that $p$ and $c_h$ have in common is $C_h$.
\end{enumerate}
\item \label{case:qandpd} 
\begin{enumerate}[(1)]
\item  If a vertex in $\Ch(D_{e},\g) $ is on $p$, or 
\item  if a vertex in $(\Ch(D_{e},\g) \cap D) \setminus \{D_{e_t}\}$ is $q$, or 
\item if there is a vertex that is on both $p$ and $q$, then
\end{enumerate}
 there exists a vertex  $D_l \in D$ on $p$ or on $q$ such that
\begin{enumerate}
\item $B_j \neq D_l \neq D_{e_t}$,  and 
\item  $B_s \to D_{l} \leftarrow D_r$, $B_s \neq D_r$ is in $\g$, where $B_s$ is on $p$, $D_r $ is on $q$, and $D_r \notin \Adj(B_s, \g)$.
\item if $D_l$ is on $p$, then $p(D_l,S'')$ is a causal path and $S' \equiv S''$.
\end{enumerate}
\end{enumerate}
\end{lemma}

\begin{proof}
Cases \ref{case:q0d}, \ref{case:p1d}, \ref{case:q1d}, and \ref{case:q2d}  follow immediately by properties of $D$ and choice of $p$ and $q$. 

For case \ref{case:q3d}, note that,  by \ref{case:children} of \cref{lemma:parentsofchildren}, $\Ch(D_e, \g) \cap D  \subseteq (\Pa(D_{e_t}, \g) \cap D) \cup \{D_{e_t}\} \cup (\Ch(D_{e_t},\g) \cap D)$. Since $q$ consists of descendants of $D_{e_t}$, vertices in $\Pa(D_{e_t},\g) \cap D$ are not on $q$.
Hence, if a vertex from $(\Ch(D_e, \g) \cap D) \setminus \{D_{e_t}\}$ is on $q$, then this vertex can only be  in $\Ch(D_{e_t},\g) \cap \Ch(D_e, \g)  \cap D$. Additionally, if any vertex other than $D_{q_2}$ on $q$ is in $\Ch(D_{e_t},\g) \cap D$, that would contradict the choice of $q$ as a shortest causal path from $D_{e_t}$ to $S$.

To show case \ref{case:p2d}, note as above that by \ref{case:children} of  \cref{lemma:parentsofchildren}, $(\Ch(D_e, \g) \cap D)  \subseteq (\Pa(D_{e_t}, \g) \cap \Ch(D_e, \g) \cap D) \cup \{D_{e_t}\} \cup (\Ch(D_{e_t},\g) \cap \Ch(D_e, \g) \cap D )$. Also, $D_{e_t}$ is not on $p$, because $p$ is d-connecting given $\pa(D_{e_t},\g)$. For the same reason, any vertex in $ \Pa(D_{e_t},\g) \cap \Ch(D_e, \g) \cap D$ cannot be a non-collider on $p$. Additionally, a vertex in $\Ch(D_{e_t},\g) \cap \Ch(D_e, \g) \cap D$ cannot be a collider, or an ancestor of a collider on $p$ (due to $p$ being d-connecting given $\Pa(D_{e_t},\g)$), or even an ancestor of $B_j$ (due to acyclicity).

If a vertex  $D_{e_l} \in \Ch(D_{e_t},\g)\cap \Ch(D_e, \g) \cap D$, $t < l \le f$ is a non-collider on $p$, then $p(D_{e_l}, S'')$ is of the form $D_{e_l} \to \dots \to S''$ and is therefore, d-connecting path given $\Pa(D_{e_l},\g)$. By choice of $t$, since $l>t$, we have to have that $B_j \to D_{e_l}$. Hence, $D_{e_l} \equiv B_{p_2}$ otherwise, we can choose a shorter path $p$ with the same or shorter distance-to-$\pa(D_{e_t},\g)$. Thus, $p$ is of the form $B_{j} \to B_{p_2} \to \dots \to S''$.

It is only left to show that a vertex in  $\Ch(D_{e},\g) \cap \Pa(D_{e_t},\g) \cap D = \Ch(D_{e},\g) \cap \Pa(D_{e_t},\g)$ cannot be a collider on $p$. Hence, suppose that  $\Ch(D_{e},\g) \cap \Pa(D_{e_t},\g) \neq \emptyset$, meaning that $t \neq 1$, since by   \cref{lemma:parentsofchildren}, $\Ch(D_{e},\g) \cap \Pa(D_{e_t},\g) = \{D_{e_1},\dots , D_{e_{t-1}} \}$.
Suppose for a contradiction that for a collider $C$ on $p$, $C  = D_{e_s}, s \in \{1, \dots ,t-1\}$ and let $B \to C \leftarrow R$ be a subpath of $p$ that contains $C$. 
Then $B, R \in \Pa(D_{e_s},\g) \subset \Pa(D_{e}, \g) \cup \{D_{e_0}, \dots, D_{e_{s-1}}\}$ (\ref{case:parents} of  \cref{lemma:parentsofchildren}). 
Note that if $R \in \{D_{e_0}, \dots, D_{e_{s-1}}\} \cup (\Pa(D_{e},\g) \cap \Pa(D_{e_t},\g)) \subseteq \Pa(D_{e_t},\g)$. Then since $R$ is a non-collider on $p$ or $R \equiv S''$, we reach a contradiction with the choice of $p$ as a path that is d-connecting given $\Pa(D_{e_t},\g)$.
If however, $R \in \Pa(D_e, \g) \setminus \Pa(D_{e_t},\g)$, then since $p(R,S'')$ is d-connecting given $\Pa(D_{e_t},\g)$, we  have a contradiction with the choice of $p$ as a shortest path among all paths with the shortest distance-to-$\Pa(D_{e_t},\g)$.

For case \ref{case:p3d}, note as before  that $\{D_{e_0}, \dots, D_{e_{t-1}}\}  \subseteq \Pa(D_{e_t}, \g) \subseteq \Pa(D_e, \g) \cup \{D_{e_0}, \dots,D_{e_{r-1}}\}$, for $D_{e_0} \equiv D_e$, based on \ref{case:parents} of  \cref{lemma:parentsofchildren}. 
We have already concluded in the proof of case \ref{case:p2d} that a vertex in $\Pa(D_{e},\g) \setminus \Pa(D_{e_t},\g)$ is not a non-collider on $p$. Since $p$ is d-connecting given $\Pa(D_{e_t},\g)$, a vertex in $ \Pa(D_{e_t},\g)$ is also not a non-collider on $p$. Note also that $D_e \in \pa(D_{e_t},\g)$. Hence, a vertex in $\Pa(D_{e},\g) \cup \{D_e\}$ is not a non-collider on $p$.

By Lemma \ref{lemma:parentsofchildren}, every collider on $p$ is in $\An(\Pa(D_{e_t},\g), \g) \subseteq \An((\Pa(D_e, \g) \cup  \{D_{e_0}, \dots, D_{e_{t-1}}\}) \setminus \{B_j\} ,\g)$. If $t = 1$, then   $\An(\Pa(D_{e_t},\g), \g) \subseteq  \An((\Pa(D_e, \g) \cup \{D_e\}) \setminus \{B_j\}, \g)$ and we are done. 

Hence, suppose that $t \neq 1$. Then $ \An(\Pa(D_{e_t},\g), \g) \subseteq \An((\Pa(D_e, \g) \cup \{D_e\}) \setminus \{B_j\},\g) \cup \{D_{e_1},\dots, D_{e_{t-1}}\}$, by applying \ref{case:parents} of  \cref{lemma:parentsofchildren}. Hence, for $p$ to be d-connecting given $(\Pa(D_e,\g) \cup \{D_{e}\} )\setminus \{B_j\}$ it is enough to show that a collider on $p$ is not in $\{D_{e_1},\dots , D_{e_{t-1}}\}$ which is something we have already proved in case \ref{case:p2d}. 

\ref{case:p4d} Note that the choice of paths $c_1, \dots , c_H$ is possible due to \ref{case:p3d}. The fact that $q$ does not contain any vertex on $c_h$, for all $h \in \{1, \dots, H\}$ follows by acyclicity of $\g$. Similarly, $p$ and $c_h$ only intersect at $C_H$, and no vertex on $c_h$ is in $S$ since otherwise we could choose a path from $\Pa(D_{e}, \g) \setminus \Pa(D_{e_t},\g)$ to $S \setminus \Pa(D_{e_t},\g)$ that has a shorter distance-to-$\Pa(D_{e_t},\g)$, or a  shorter path with the same distance-to-$\Pa(D_{e_t},\g)$.

Now, we move on to case \ref{case:qandpd}. Note first that $B_j$ cannot be on $q$ and that $D_{e_t}$ is not on $p$ by case \ref{case:p2d}. 

(1) Suppose a vertex in $\Ch(D_{e_t},\g)$ is on $p$.  By case \ref{case:p2d}, $p$ is of the form $B_j \to B_{p_2} \to \dots \to S''$ and $B_{p_2} \in \Ch(D_{e_t},\g) \cap D$. Then, clearly, $D_{e_t} \to B_{p_2}$ is in $\g$ and since $B_j \notin \Adj(D_{e_t},\g)$, $D_l \equiv B_{p_2}, B_{s} \equiv B_j$, $D_r \equiv D_{e_t}$.

The proof for case $\lnot (1) \land  (2)$  is analogous using the result of case \ref{case:q2d}, and yields $D_l \equiv D_{q_2}, B_{s} \equiv B_j$, $D_r \equiv D_{e_t}$.

$\lnot (1) \land \lnot (2) \land (3)$ Since $D_{e_t}$ is not on $p$ and $B_j$ is not on $q$,  $|p| \neq 1 \neq |q|$. Hence, let $D_l^1$ be the closest vertex to $B_j$ on $p$ that is also on $q$.   Let $p(B_j, D_l^1) = \langle B_j =B_{p_1}, \dots, B_{p_{l_1}} = D_l^1 \rangle$ and $q(D_{e_t}, D_l^1) = \langle D_{e_t} =D_{q_1}, \dots, D_{q_{l_2}} = D_l^1 \rangle$.

Since $D_l^1 \in \De(D_{e_t},\g)$, it follows that $D_l^1$ is not a collider on $p$, nor an ancestor of a collider on $p$ (otherwise, $p$ would not be d-connecting given $\Pa(D_{e_t},\g)$), nor an ancestor of $B_j$ (due to acyclicity). Therefore, $p(B_{p_{l_1-1}}, S'')$ is of the form $B_{p_{l_1-1}} \to D_l^1 \to \dots \to S''$. Since $q(D_l^1, S')$ is also a causal path and d-connecting given $\pa(D_{e_t},\g)$,  it must be that $p(D_l^1, S'') \equiv q(D_l^1, S')$ (otherwise, we can choose a shorter path for $p$ or $q$).

Next, note that $D_l^1 \neq D_{e_t}$, because $D_{e_t}$ cannot be on $p$. Additionally, $D_l^1 \neq B_j$, as $B_j$ cannot be on $q$, due to acyclicity of $\g$.
Since $D_{l}^{1}\neq D_{e_t}$, edge  $D_{q_{l_2 -1}} \to D_{q_{l_2}}$ is on $q$. Hence, we can consider the possibility that $(B_s,D_l,D_r) \equiv (B_{p_{l_1-1}}, D_l^{1}, D_{q_{l_2 -1}})$ depending on whether the edge $\langle B_{p_{l_1-1}}, D_{q_{l_2 - 1}}\rangle$ is in $\g$.
 
 If $B_{p_{l_1 -1}} \to  D_{q_{l_2 - 1}}$  is in $\g$, then $l_2 \neq 2$, and $D_{q_{l_2-2}} \neq D_{e_t}$. Hence, we can consider the possibility that $(B_s,D_l,D_r) \equiv (B_{p_{l_1-1}}, D_{q_{l_2-1}}, D_{q_{l_2 -2}})$  depending on the existence of edge $\langle B_{p_{l_1-1}}, D_{q_{l_2 -2}}\rangle$ in $\g$.
 
Alternatively, $B_{p_{l_1 -1}} \leftarrow D_{q_{l_2 - 1}}$ is in $\g$, implying that $l_1 \neq 1$ meaning that $B_{p_{l_1-2}} \neq D_{e}$ and that $B_{p_{l_1-3}} \to B_{p_{l_1 -2}} \to B_{p_{l_1 -1}} \to \dots \to S''$ is in $\g$. Hence, in this case,  we can consider the possibility that $(B_s,D_l,D_r) \equiv (B_{p_{l_1-2}}, B_{p_{l_1-1}}, D_{q_{l_2 -1}})$  depending on the existence of edge $\langle B_{p_{l_1-2}}, D_{q_{l_2 -1}}\rangle$ in $\g$.
Since $p(B_j, D_l)$ and $q(D_{e_t},D_l)$ are of finite length, and since $B_j \notin \Adj(D_{e_t},\g)$, we can continue the above arguments until we find $B_s, D_l$ and $D_r$ described in the case \ref{case:qandpd}.
\end{proof}
  \end{document}